\documentclass[11pt,reqno]{amsart}
\usepackage[text={160mm, 236mm},centering]{geometry}
\usepackage{float}
\usepackage{appendix}
\usepackage{graphicx}
\usepackage{chngpage}
\usepackage{multicol}
\usepackage{subcaption}
\usepackage{setspace}
\usepackage{multirow}
\usepackage[colorlinks=true, citecolor=blue, urlcolor=blue, linkcolor=blue]{hyperref}
\usepackage[all,cmtip]{xy}
\usepackage{amssymb}
\usepackage{tikz-cd}
\usepackage{amsmath}
\usepackage{relsize}
\usepackage{longtable}
\IfFileExists{pstricks.sty}{\usepackage{pstricks}}{}
\usepackage{color}
\usepackage{array,booktabs,colortbl}
\usepackage{ mathrsfs }
\usepackage[all]{xy}
\usepackage{mathabx,epsfig}

\theoremstyle{plain}
\newtheorem{thm}{Theorem}[section]
\newtheorem{theorem}[thm]{Theorem}

\newtheorem{lemma}[thm]{Lemma}
\newtheorem{corollary}[thm]{Corollary}
\newtheorem{proposition}[thm]{Proposition}
\theoremstyle{plain}

\theoremstyle{definition}
\newtheorem{remark}[thm]{Remark}

\newtheorem{definition}[thm]{Definition}

\newtheorem{example}[thm]{Example}

\numberwithin{equation}{section}

\newcommand{\romea}{\uppercase\expandafter{\romannumeral1}}
\newcommand{\romeb}{\uppercase\expandafter{\romannumeral2}}

\newcommand{\Pic}{{\rm Pic}}

\newcommand{\Aut}{{\rm Aut}}

\newcommand{\Diag}{{\rm diag}}

\newcommand{\GL}{{\rm GL}}

\newcommand{\PSp}{{\rm PSp}}

\newcommand{\PGL}{{\rm PGL}}
\newcommand{\PSL}{{\rm PSL}}

\newcommand{\SL}{{\rm SL}}

\newcommand{\Bir}{{\rm Bir}}

\newcommand{\C}{{\mathbb C}}

\renewcommand{\P}{{\mathbb P}}

\newcommand{\Z}{{\mathbb Z}}

\newcommand{\id}[2]{[#1,#2]}
\makeatletter
\@namedef{subjclassname@2020}{%
	\textup{2020} Mathematics Subject Classification}
\makeatother

\begin{document}
	
	\title{Automorphism groups of smooth quartic threefolds}
	
	\author{Zigang Zhu}
	
	\address{Center for Applied Mathematics and KL-AAGDM, Tianjin University, Weijin Road 92, Tianjin 300072, P.R. China}
	
	\email{zhzg0313@tju.edu.cn}
	
	
	\begin{abstract}
		We classify finite groups of automorphisms of smooth quartic threefolds over $\C$. Every such group embeds into one of $17$ maximal groups, which we realize by explicit equations and matrix generators. Among finite subgroups of $\PGL(5,\C)$ preserving a smooth quartic threefold, there are exactly $24$ maximal ones up to conjugacy, and we use them to determine the $194$ groups that occur as full automorphism groups. As an application, we classify the finite groups of projective automorphisms of smooth quartic surfaces and show that there are exactly $12$ maximal ones up to conjugacy in $\PGL(4,\C)$.
	\end{abstract}
	\maketitle
	\setcounter{tocdepth}{1}
	\tableofcontents
	\section{Introduction}
	Symmetries of smooth projective hypersurfaces have been studied since the nineteenth century. Klein showed that the plane quartic $x_1^3x_2+x_2^3x_3+x_3^3x_1=0$ admits $168$ automorphisms~\cite{Kle79}, the maximum allowed by Hurwitz's bound in genus three, and the classification of automorphism groups of smooth cubic surfaces goes back to Segre~\cite{Seg42} (see also~\cite{Hos97}, \cite{Dol12}). For a smooth hypersurface $X\subset\P^{n+1}$ of degree $d\geq3$ with $(n,d)\neq(1,3),(2,4)$, the group $\Aut(X)$ is finite, and every automorphism of $X$ is induced by an element of $\PGL(n+2,\C)$ (see~\cite[Theorems~1 and~2]{MM63} for $n\geq2$ and~\cite[Theorem~2]{Cha78} for $n=1$). Determining the possible groups $\Aut(X)$ is thus a problem about finite linear groups preserving a form, going back to the classical work of Klein, Jordan, Blichfeldt, and Brauer \cite{Kle84}, \cite{Jor78}, \cite{Bli17}, \cite{Bra67}. For quartic surfaces, which are K3 surfaces, $\Aut(X)$ may be infinite (see~\cite{KOY26} for the Fermat quartic surface), and the finite groups of symplectic automorphisms of K3 surfaces are described by Mukai's celebrated theorem in terms of the Mathieu group $M_{23}$~\cite{Muk88}. In dimension at least three, complete classifications are more recent.
	Oguiso--Yu~\cite{OY19} developed a systematic method and gave a complete classification of finite groups acting faithfully on smooth quintic threefolds. Corresponding classifications were subsequently obtained for smooth cubic threefolds, with six maximal groups~\cite{WY20}, and for smooth cubic fourfolds and fivefolds~\cite{YYZ24} (see also~\cite{LZ22} for the classification of symplectic automorphism groups of smooth cubic fourfolds). Moreover, for every $n\geq1$ and $d\geq3$, the smooth hypersurfaces of degree $d$ in $\P^{n+1}$ with projective automorphism groups of largest order have been determined (\cite{YYZ25}, \cite{EL25}).
	
	In $\P^4$, smooth quartic threefolds sit between the two threefold
	cases just mentioned: cubic threefolds, which are Fano of index two,
	and quintic threefolds, which are Calabi--Yau. A fundamental feature of quartic threefolds is their birational rigidity. Building on ideas of Fano, Iskovskikh--Manin~\cite{IM71} proved that every birational map between smooth quartic threefolds is biregular (see also~\cite{IP99}, \cite{Che05}, \cite{Kol19}). In particular, smooth quartic threefolds are nonrational. Since some of them are unirational, they are among the first counterexamples to the L\"uroth problem, along with those found at about the same time by Clemens--Griffiths~\cite{CG72} and Artin--Mumford~\cite{AM72}. The work of Iskovskikh--Manin initiated the theory of birational (super)rigidity (see e.g. \cite{Puk98}, \cite{dF13}, \cite{dF16}). For a smooth quartic threefold $X$, birational superrigidity implies that the group $\Bir(X)$ of birational automorphisms of $X$ equals $\Aut(X)$, a finite subgroup of $\PGL(5,\C)$. This contrasts with the large birational automorphism groups of rational varieties (see e.g.~\cite{DI09}, \cite{Pro12}).
	
	It is therefore natural to ask which finite groups act faithfully on smooth quartic threefolds. Partial results are known: the possible prime orders of automorphisms and the cyclic groups admitting faithful actions were determined in~\cite{GAL13} and~\cite{Zhe22}, respectively. The Fermat quartic threefold has an automorphism group of largest possible order (\cite{YYZ25}, \cite{EL25}). 
	\begin{theorem}\label{thm:main}
		A finite group $G$ acts faithfully on a smooth quartic threefold if and only if $G$ is isomorphic to a subgroup of one of the following $17$ groups:
		\begin{table}[H]
			\centering
			\begin{tabular}{c l r@{\hspace{10pt}}c l r}
				\hline
				No.&group&order&No.&group&order\\
				\hline
				1&$C_4^4\rtimes S_5$&$30720$
				&10&$(C_4. A_4)\times C_9$&$432$\\
				2&$\PSL(2,7)\times(C_4. A_4)$&$8064$
				&11&$C_5\rtimes(C_{16}\rtimes C_4)$&$320$\\
				3&$C_4.(C_2^4\rtimes S_5)$&$7680$
				&12&$C_{61}\rtimes C_5$&$305$\\
				4&$\PSL(2,7)\times(C_4^2\rtimes C_2)$&$5376$
				&13&$C_9\times(C_4^2\rtimes C_2)$&$288$\\
				5&$(C_4. A_4)\times(C_4^2\rtimes S_3)$&$4608$
				&14&$C_8. S_4$&$192$\\
				6&$(C_4. A_4)^2\rtimes C_2$&$4608$
				&15&$C_{108}$&$108$\\
				7&$C_4.\PSL(2,7)$&$672$
				&16&$C_{81}$&$81$\\
				8&$\GL(2,5)$&$480$
				&17&$C_3\times D_{18}$&$54$\\
				9&$C_4. S_5$&$480$&&&\\
				\hline
			\end{tabular}
		\end{table}
		These are exactly the maximal groups under abstract subgroup inclusion.
	\end{theorem}
	Beyond determining which finite groups act faithfully, it is natural to ask how they act on $\P^4$. We deduce Theorem~\ref{thm:main} from a classification, up to conjugacy, of the finite subgroups of $\PGL(5,\C)$ that preserve a smooth quartic threefold. Let $G_1,\dots,G_{17}\subset\PGL(5,\C)$ be the groups of Example~\ref{ex:17maximal}, which realize the groups in Theorem~\ref{thm:main}, and let $G_{18},\dots,G_{24}$ be the groups of Example~\ref{ex:seven-projective-actions}.
	\begin{theorem}\label{thm:24-projective-actions}
		A finite subgroup of $\PGL(5,\C)$ preserves a smooth quartic threefold if and only if it is conjugate to a subgroup of one of $G_1,\dots,G_{24}$.
	\end{theorem}
	None of these $24$ groups is conjugate to a subgroup of another (Lemma~\ref{lem:24-irredundant}).
	Consequently, $G_i=\Aut(X_i)$ for every smooth quartic threefold $X_i$ in Examples~\ref{ex:17maximal} and~\ref{ex:seven-projective-actions}. As abstract groups, $G_{18},\dots,G_{24}$ embed into $G_{10}$, $G_{15}$, or $G_{17}$; for instance, $G_{20}$, $G_{21}$, and $G_{22}$ are all cyclic of order $27$. 
	
	Our results are also related to the Jordan property of birational automorphism groups (see \cite{PS14}, \cite{PS17}). Recall that the {\it Jordan constant} $J(G)$ (resp.\ the {\it weak Jordan constant} $\overline J(G)$) of a group $G$ is the smallest integer $J$ such that every finite subgroup of $G$ contains a normal abelian (resp.\ an abelian) subgroup of index at most $J$ (see~\cite[Definition~1.2.1]{PS17}). Prokhorov--Shramov~\cite[Theorem~1.2.4]{PS17} proved that $\overline J(\Bir(Y))\leq10\,368$ for every rationally connected threefold $Y$ over a field of characteristic zero. For smooth hypersurfaces $X\subset\P^4$ of degree at least two, they obtained the sharper bound $\overline J(\Aut(X))\leq960$~\cite[Lemma~6.1.5]{PS17}. Since $\Bir(X)=\Aut(X)$, Theorem~\ref{thm:main} gives the following improvement for quartic threefolds.
	\begin{corollary}\label{cor:jordan}
	Let $X$ be a smooth quartic threefold. Then
	$$
	J(\Bir(X))\leq2016\qquad\text{and}\qquad\overline J(\Bir(X))\leq120.
	$$
	Both bounds are sharp, attained by the threefolds $X_2$ and $X_1$ of Example~\ref{ex:17maximal}, respectively.
	\end{corollary}
It is also natural to ask which groups occur as the full automorphism group of a smooth quartic threefold, and how this group varies in moduli. For smooth cubic threefolds and fourfolds, these questions were recently answered by Fu--Wang--Zheng~\cite{FWZ26}, who classified the full automorphism groups together with their projective actions and the corresponding families. Following their approach, we show in Section~\ref{sec:full-automorphism-groups} that exactly $194$ finite groups, including the trivial group, occur as $\Aut(X)$, realized by $311$ conjugacy classes of subgroups of $\PGL(5,\C)$ (Theorem~\ref{thm:full-groups}). In particular, the locus of smooth quartic threefolds with nontrivial automorphisms has exactly $7$ irreducible components and is the singular locus of the coarse moduli space, which has only terminal quotient singularities (Remarks~\ref{rem:nontrivial-locus} and~\ref{rem:terminal-moduli}).

Theorem~\ref{thm:24-projective-actions} also answers the corresponding question for smooth quartic surfaces. For a smooth quartic form $f$ in four variables, the K3 surface $V(f)\subset\P^3$ may have infinitely many automorphisms, but the group $G_f\subset\PGL(4,\C)$ of projective automorphisms of $V(f)$ is finite.  Oguiso~\cite{Ogu05} characterized the Fermat quartic surface among all K3 surfaces, by means of its
finite group symmetries. Brandhorst--Hofmann~\cite{BH23} classified finite automorphism actions on K3 surfaces up to deformation in terms of Hodge-theoretic data. The projective actions on smooth quartic surfaces, which correspond to invariant very ample polarizations of degree~$4$, are not listed there explicitly. For smooth quartic surfaces, Avila--Ortiz--Troncoso~\cite{AOT24} determined those invariant under the finite primitive subgroups of $\PGL(4,\C)$, and some special cases were studied earlier in~\cite{MPK16}, \cite{FMPK16}, \cite{Dol18}, and~\cite{PZ21}. To our knowledge, an explicit classification of the maximal finite projective groups preserving smooth quartic surfaces has not been given (see~\cite[Example~1.2]{Dol19} and~\cite[Section~1]{AOT24}). Since $G_f$ is the quotient of the centralizer of $[\Diag(I_4,i)]$ in $\Aut(V(f+x_5^4))$ by the subgroup generated by this element, Theorem~\ref{thm:24-projective-actions} yields such a classification. For $i\leq15$ with $i\neq2,10,12$ and for $i=23,24$, let $h_i=F_i(x_1,x_2,x_3,x_4,0)$, where $F_i$ is the form of Example~\ref{ex:17maximal} or~\ref{ex:seven-projective-actions}, with general $t_1,t_2$ for $i=23$.
\begin{theorem}\label{thm:quartic-surfaces}
	\begin{enumerate}
		\item A finite subgroup of $\PGL(4,\C)$ preserves a smooth quartic surface if and only if it is conjugate to a subgroup of one of the following $12$ groups $G_{h_i}$, none of which is conjugate to a subgroup of another:
		\begin{table}[H]
			\centering
			\begin{tabular}{c l r@{\hspace{10pt}}c l r}
				\hline
				$i$& structure of $G_{h_i}$&order&$i$& structure of $G_{h_i}$&order\\
				\hline
				1&$C_4^3\rtimes S_4$&$1536$&8&$S_5$&$120$\\
				3&$C_2^4\rtimes S_5$&$1920$&11&$C_4\times(C_5\rtimes C_4)$&$80$\\
				4&$C_4\times\PSL(2,7)$&$672$&13&$C_{36}$&$36$\\
				5&$C_4.(A_4\times D_8)$&$384$&15&$C_{27}$&$27$\\
				6&$C_4.(A_4^2\rtimes C_2)$&$1152$&23&$C_9$&$9$\\
				7&$\PSL(2,7)$&$168$&24&$C_2^2\rtimes C_9$&$36$\\
				\hline
			\end{tabular}
		\end{table}
		\item A finite group acts faithfully on a smooth quartic surface in $\P^3$ by projective automorphisms if and only if it is isomorphic to a subgroup of $G_{h_i}$ for some $i\in\{1,3,4,5,6,\allowbreak11,13,15,24\}$. None of these $9$ groups is isomorphic to a subgroup of another.
	\end{enumerate}
\end{theorem}
After this work was completed, we learned that Fu--Wang--Zheng~\cite{FWZ26b} have independently classified the full projective automorphism groups of smooth quartic surfaces by a different approach.

We now explain what is new in the proof. Theorem~\ref{thm:main} is the quartic counterpart of the classifications for quintic threefolds~\cite{OY19}, cubic threefolds~\cite{WY20}, and cubic fourfolds and fivefolds~\cite{YYZ24}. These classifications combine restrictions on subgroups of prime-power order with representation theory, smoothness conditions, and computations in GAP~\cite{GAP} to rule out candidate groups. In the quartic case, the diagonal symmetries of the Fermat quartic threefold already form the $2$-group $C_4^4$, and $6$ of the $17$ groups in Theorem~\ref{thm:main} have Sylow $2$-subgroups of order at least $2^7$, the largest being of order $2^{11}$ for $G_1$. There are more than ten million groups of order $2^9$ and almost $5\cdot10^{10}$ groups of order $2^{10}$, and the latter are not even in the Small Groups Library of GAP, so a direct enumeration of candidate $2$-groups is impractical. Instead, we work with the action of $\Aut(X)$ on $\P^4$ throughout and classify the projective actions directly (Theorem~\ref{thm:24-projective-actions}); Theorem~\ref{thm:main} then follows by comparing the groups $G_i$ as abstract groups. This finer classification is also the input for the results on quartic surfaces (Theorem~\ref{thm:quartic-surfaces}) and on full automorphism groups (Theorem~\ref{thm:full-groups}).

Since $\gcd(4,5)=1$, every subgroup of $\Aut(X)$ lifts isomorphically into $\Aut(F)$, where $X=V(F)$ (\cite{GLM23}). We decompose $\C^5$ into blocks permuted by this lift, with primitive block stabilizers. This framework of {\it primitive constituents} and {\it subdegree sequences} (Definition~\ref{def:primitive-constituents}) was developed in~\cite[Section~3]{YYZ25} to bound automorphism groups, and we refine it to classify the projective actions. Smoothness allows only $A_4$, $\PSL(2,7)$, and four abstract groups as primitive constituents of dimensions two, three, and four, respectively (Proposition~\ref{prop:possi-primitive-cons}). Hence the nonmonomial case splits into five cases by subdegree sequence, governed by binary quartics, the Klein quartic, and quartic surfaces with primitive symmetry groups (Theorem~\ref{thm:nonmonomial}). For instance, $G_3$, $G_7$, $G_8$, and $G_9$ come from such surfaces, including those of Burnside and Edge, and $G_{14}\cong C_8.S_4$ comes from a form coupling two binary blocks.

In the monomial case, $\Aut(X)$ is an extension of a diagonal group by a permutation group of degree at most five. After classifying the abelian groups (Theorem~\ref{thm:maximal-abelian}), we use normal elementary abelian $2$-subgroups to replace $F$ by an invariant smooth sum of forms in fewer variables (Lemma~\ref{lem:monomial-smooth-sum}). This lemma holds in any degree $d$ for abelian groups of exponent dividing $d$. Similarly, Proposition~\ref{prop:p-subgroups-abelian} shows that $p$-subgroups are abelian if $p\nmid d$ and there are fewer than $2p$ variables. In both the nonmonomial and the monomial case, we show directly that $\Aut(X)$ is conjugate to a subgroup of some $G_i$ (Theorems~\ref{thm:nonmonomial} and~\ref{thm:monomial}), and we deduce the main results in Section~\ref{sec:maximal-projective-actions}. Thus the proofs of Theorems~\ref{thm:main} and~\ref{thm:24-projective-actions} enumerate no candidate groups from the Small Groups Library, and use computer algebra only for explicit verifications and for comparing the groups $G_i$. The enumeration in Section~\ref{sec:full-automorphism-groups} is computer-aided, and Appendix~\ref{app:abelian-representations} is not needed for these proofs.
\subsection*{Acknowledgements}
I sincerely thank my advisor, Professor Xun Yu, for introducing me to this problem and for his valuable advice and insightful discussions.

\subsection*{AI disclosure} The author developed the main ideas and established the main results of this paper. The code for Section~\ref{sec:full-automorphism-groups} was originally written in a mixture of GAP~\cite{GAP}, SageMath~\cite{Sage}, and Mathematica~\cite{Wo}. OpenAI's GPT-6 Astra was used to identify the corresponding GAP functions, so that the code could be rewritten entirely in GAP. It also helped organize the output of the results file. After $X_8$, $X_9$, $X_{14}$, $X_{24}$, and $X_{17}(\lambda,\mu)$ were constructed from invariant forms of linear groups, GPT-6 Astra was used to find linear changes of coordinates giving the equations in Examples~\ref{ex:17maximal} and~\ref{ex:seven-projective-actions} with fewer monomials and simpler coefficients. The author also used GPT-6 Astra to help locate earlier results in the literature and to adjust the formatting of Table~\ref{tab:full-groups}, and used both GPT-6 Astra and Claude Opus 5.5 to polish the English. The author takes full responsibility for the content of this paper.
\section{Notation and conventions}\label{sec:notation}
In this section, we introduce notation and recall some preliminary facts about automorphism groups of smooth quartic threefolds over $\C$. For linear actions and homogeneous polynomials, we mainly follow the notation and conventions of \cite[Section~2]{YYZ25}, specialized to quartic threefolds.

\noindent{\bf (2.1)} Let $(e_1,\dots,e_5)$ be the standard basis of $\C^5$, and let $(x_1,\dots,x_5)$ be the dual coordinates. Matrices act on column vectors from the left and on polynomials from the right by
$$
(F,A)\longmapsto A(F):=
F\left(\sum_{i=1}^5a_{1i}x_i,\dots,
\sum_{i=1}^5a_{5i}x_i\right),
\qquad A=(a_{ij})\in\GL(5,\C).
$$
Thus $(A_1A_2)(F)=A_2(A_1(F))$ for $A_1,A_2\in\GL(5,\C)$. Two polynomials $F,F'\in\C[x_1,\dots,x_5]$ are called {\it isomorphic} if they lie in the same $\GL(5,\C)$-orbit. The same conventions apply to polynomials in fewer variables. We write $\xi_m:=e^{2\pi\sqrt{-1}/m}$ and $i:=\xi_4$, and $I_r$ for the $r\times r$ identity matrix.

\noindent{\bf (2.2)} We use the notions of {\it form} and {\it homogeneous polynomial} interchangeably. The automorphism group of a homogeneous polynomial $F\in\C[x_1,\dots,x_r]$ of degree $d\geq2$ is
$$
\Aut(F):=\{A\in\GL(r,\C)\mid A(F)=F\}.
$$
We say that $F$ is {\it smooth} if $(0,\dots,0)$ is the only common zero of $\partial F/\partial x_1,\dots,\partial F/\partial x_r$.
Let $\pi\colon\GL(r,\C)\to\PGL(r,\C)$ be the natural projection. For $A\in\GL(r,\C)$, we write $[A]:=\pi(A)$, and we set $G_F:=\pi(\Aut(F))$.
A subgroup $H\subset\GL(r,\C)$ (resp.\ $H\subset\PGL(r,\C)$) {\it preserves} $F$ if $H\subset\Aut(F)$ (resp.\ $H\subset G_F$); in this case, we call $F$ an {\it $H$-invariant form} of degree $d$. A subgroup $H\subset\GL(r,\C)$ (resp.\ $\PGL(r,\C)$) is {\it conjugate} to another such subgroup $H'$ if $gHg^{-1}=H'$ for some $g\in\GL(r,\C)$ (resp.\ $\PGL(r,\C)$). Note that $\Aut(A(F))=A^{-1}\Aut(F)A$ for $A\in\GL(r,\C)$.
A monomial $\mathfrak{m}$ of degree $d$ is {\it in} $F$, written $\mathfrak{m}\in F$, if its coefficient in $F$ is nonzero.
A smooth quartic form $F\in\C[x_1,\dots,x_5]$ defines a smooth quartic threefold $X_F:=V(F)\subset\P^4$. By Matsumura--Monsky \cite[Theorems~1--2]{MM63}, $\Aut(X_F)$ is finite and every automorphism of $X_F$ is induced by an element of $\PGL(5,\C)$; we therefore identify $\Aut(X_F)$ with $G_F$. 

\noindent{\bf (2.3)} As in \cite[Definitions~4.4 and~4.6]{OY19}, let $H\subset\PGL(r,\C)$ be a finite subgroup. A {\it lifting} (resp.\ an {\it $F$-lifting}) of $H$ is a subgroup $\widetilde H\subset\GL(r,\C)$ (resp.\ $\widetilde H\subset\Aut(F)$) such that $\pi$ maps $\widetilde H$ isomorphically onto $H$. We say that $H$ is {\it $F$-liftable} if it admits an $F$-lifting. By \cite[(1.1)]{YYZ25}, we have the central extension
$$
1\longrightarrow\langle\xi_dI_r\rangle\longrightarrow\Aut(F)\stackrel{\pi}{\longrightarrow}G_F\longrightarrow1.
$$
Thus $H\subset G_F$ is $F$-liftable if and only if this extension splits over $H$, and any two $F$-liftings of $H$ differ by a homomorphism $H\to\langle\xi_dI_r\rangle$. If $F$ is a smooth quartic form in five variables, then every subgroup of $\Aut(X_F)$ is $F$-liftable by \cite[Theorem~3.5]{GLM23}, since $\gcd(4,5)=1$. For such $F$, the group $\langle\widetilde H,\xi_4I_5\rangle=\pi^{-1}(H)\cap\Aut(F)$ does not depend on the choice of an $F$-lifting $\widetilde H$ of $H\subset G_F$.

\noindent{\bf (2.4)} For $F_1,\dots,F_k\in\C[x_1,\dots,x_r]$, we write $(F_1,\dots,F_k)$ for the ideal of $\C[x_1,\dots,x_r]$ generated by $F_1,\dots,F_k$, and $(F_1,\dots,F_k)^a$ for its $a$-th power, where $a$ is a positive integer.
We write $\Diag(A_1,\dots,A_k)$ for the block-diagonal matrix with diagonal blocks $A_1,\dots,A_k$.
For a permutation $\sigma$ of $\{1,\dots,r\}$, let $P_\sigma$ be the permutation matrix with $P_\sigma e_i=e_{\sigma(i)}$, so that $P_\sigma(x_i)=x_{\sigma^{-1}(i)}$. 

We denote by $C_m$ the cyclic group of order $m$, by $D_{2m}$ the dihedral group of order $2m$, by $S_m$ and $A_m$ the symmetric and alternating groups on $m$ letters, respectively, and by $Q_8$ the quaternion group of order $8$.
For a group $G$ and a subgroup $H\subset G$, we denote by $Z(G)$ the center of $G$, by $G'$ its derived subgroup, and by $C_G(H)$ the centralizer of $H$ in $G$.
We write $N\rtimes H$ for a semidirect product with normal subgroup $N$, and $N.H$ for a group that fits into a short exact sequence $1\to N\to N.H\to H\to1$.
Finally, $\id{n}{k}$ denotes the $k$-th group of order $n$ in the Small Groups Library of GAP \cite{GAP}.

\section{Examples and abelian groups}\label{ss:17maximal}
We begin with explicit examples realizing the groups in Theorems~\ref{thm:main} and~\ref{thm:24-projective-actions}, and then study the finite abelian groups acting faithfully on smooth quartic threefolds.
\subsection{The examples}
The following examples are listed in nonincreasing order of $|G_i|$.
\begin{example}\label{ex:17maximal}
	\begin{enumerate}
		\item
		The Fermat quartic threefold is $X_1=V(F_1)$, where $F_1=x_1^4+x_2^4+x_3^4+x_4^4+x_5^4$. Its automorphism group is $\Aut(X_1)=G_1\cong C_4^4\rtimes S_5$, of order $2^{11}\cdot3\cdot5=30720$, generated by $[\Diag(i,-i,1,1,1)]$, $[P_{(1\,2)}]$, and $[P_{(1\,2\,3\,4\,5)}]$ (see \cite[p.~147]{HS80} and \cite{Shi88}).
		\item
		Let $X_2=V(F_2)$, where $F_2=x_1^3x_2+x_2^3x_3+x_3^3x_1+x_4^3x_5+x_5^4$. Then $\Aut(X_2)$ contains a subgroup $G_2\cong\PSL(2,7)\times(C_4.A_4)$ of order $2^{7}\cdot3^{2}\cdot7=8064$.
		\item
		Let $X_3=V(F_3)$, where $F_3=x_1^4+x_2^4+x_3^4+x_4^4+12x_1x_2x_3x_4+x_5^4$. Then $\Aut(X_3)$ contains a subgroup $G_3\cong C_4.(C_2^4\rtimes S_5)$ of order $2^{9}\cdot3\cdot5=7680$.
		\item
		Let $X_4=V(F_4)$, where $F_4=x_1^3x_2+x_2^3x_3+x_3^3x_1+x_4^4+x_5^4$. Then $\Aut(X_4)$ contains a subgroup $G_4\cong\PSL(2,7)\times(C_4^2\rtimes C_2)$ of order $2^{8}\cdot3\cdot7=5376$.
		\item
		Let $X_5=V(F_5)$, where $F_5=x_1^3x_2+x_2^4+x_3^4+x_4^4+x_5^4$. Then $\Aut(X_5)$ contains a subgroup $G_5\cong(C_4.A_4)\times(C_4^2\rtimes S_3)$ of order $2^{9}\cdot3^{2}=4608$.
		\item
		Let $X_6=V(F_6)$, where $F_6=x_1^3x_2+x_2^4+x_3^3x_4+x_4^4+x_5^4$. Then $\Aut(X_6)$ contains a subgroup $G_6\cong(C_4.A_4)^2\rtimes C_2$ of order $2^{9}\cdot3^{2}=4608$.
		\item
		Let $X_7=V(F_7)$, where $F_7=2x_1^4+x_2^3x_3+x_3^3x_4+x_2x_4^3+6x_1x_2x_3x_4+x_5^4$. Then $\Aut(X_7)$ contains a subgroup $G_7\cong C_4.\PSL(2,7)$ ($\id{672}{1047}$) of order $2^{5}\cdot3\cdot7=672$.
		\item
		Let $X_8=V(F_8)$, where $F_8=(x_1^2+x_2^2)(x_1x_2+3x_3x_4)+4(x_1x_3^3+x_2x_4^3)+x_5^4$. Then $\Aut(X_8)$ contains a subgroup $G_8\cong\GL(2,5)$ ($\id{480}{218}$) of order $2^{5}\cdot3\cdot5=480$.
		\item
		Let $X_9=V(F_9)$, where $F_9=5(x_1^4+x_2^4)-6x_1^2x_2^2+4(x_1x_3^3+x_2x_4^3)-24x_1x_2x_3x_4+3x_3^2x_4^2+x_5^4$. Then $\Aut(X_9)$ contains a subgroup $G_9\cong C_4.S_5$ ($\id{480}{946}$) of order $2^{5}\cdot3\cdot5=480$.
		\item
		Let $X_{10}=V(F_{10})$, where $F_{10}=x_1^3x_2+x_2^4+x_3^3x_4+x_4^3x_5+x_5^4$. Then $\Aut(X_{10})$ contains a subgroup $G_{10}\cong(C_4.A_4)\times C_9$ ($\id{432}{329}$) of order $2^{4}\cdot3^{3}=432$.
		\item
		Let $X_{11}=V(F_{11})$, where $F_{11}=x_1^3x_2+x_2^3x_3+x_3^3x_4+x_4^3x_1+x_5^4$. Then $\Aut(X_{11})$ contains the subgroup
		$
		G_{11}=\bigl\langle[\Diag(\xi_{80},\xi_{80}^{77},\xi_{80}^{9},\xi_{80}^{53},1)],\ [P_{(1\,2\,3\,4)}]\bigr\rangle\cong C_5\rtimes(C_{16}\rtimes C_4)
		$
		($\id{320}{184}$), of order $2^{6}\cdot5=320$.
		\item
		The Klein quartic threefold is $X_{12}=V(F_{12})$, where $F_{12}=x_1^3x_2+x_2^3x_3+x_3^3x_4+x_4^3x_5+x_5^3x_1$. Its automorphism group is $\Aut(X_{12})=G_{12}\cong C_{61}\rtimes C_5$ ($\id{305}{1}$), of order $5\cdot61 = 305$, generated by $[\Diag(\xi_{61},\xi_{61}^{58},\xi_{61}^{9},\xi_{61}^{34},\xi_{61}^{20})]$ and $[P_{(1\,2\,3\,4\,5)}]$ (see \cite[Theorem~B]{GLMV24}).
		\item
		Let $X_{13}=V(F_{13})$, where $F_{13}=x_1^3x_2+x_2^3x_3+x_3^4+x_4^4+x_5^4$. Then $\Aut(X_{13})$ contains a subgroup $G_{13}\cong C_9\times(C_4^2\rtimes C_2)$ ($\id{288}{54}$) of order $2^{5}\cdot3^{2}=288$.
		\item
		Let $X_{14}=V(F_{14})$, where $F_{14}=x_1x_3(x_1^2+x_3^2)+x_2x_4(x_2^2+x_4^2)-\sqrt{3}\,i\,(x_1x_2-x_3x_4)(x_1x_4+x_2x_3)+x_5^4$. Then $\Aut(X_{14})$ contains the subgroup $G_{14}\cong C_8.S_4$ ($\id{192}{964}$) of order $2^{6}\cdot3=192$ generated by the classes of $\Diag\bigl(\tfrac{\xi_3}{\sqrt2}\bigl(\begin{smallmatrix}1&-i\\1&i\end{smallmatrix}\bigr),-\tfrac{\xi_3^2}{\sqrt2}\bigl(\begin{smallmatrix}1&-i\\1&i\end{smallmatrix}\bigr),1\bigr)$, $\Diag(\xi_8I_2,\xi_8^5I_2,1)$ and $P_{(1\,3)(2\,4)}\Diag(1,i,1,i,1)$.
		\item
		Let $X_{15}=V(F_{15})$, where $F_{15}=x_1^3x_2+x_2^3x_3+x_3^3x_4+x_4^4+x_5^4$. Then $\Aut(X_{15})$ contains the subgroup $G_{15}\cong C_{108}$ ($\id{108}{2}$) generated by $[\Diag(\xi_{108},\xi_{36}^{35},\xi_{12},-i,1)]$.
		\item
		Let $X_{16}=V(F_{16})$, where $F_{16}=x_1^3x_2+x_2^3x_3+x_3^3x_4+x_4^3x_5+x_5^4$. Its automorphism group is $\Aut(X_{16})=G_{16}\cong C_{81}$ ($\id{81}{1}$), generated by $[\Diag(\xi_{81},\xi_{27}^{26},\xi_{9},\xi_{3}^{2},1)]$ (see \cite[Proposition~3.2]{GLMV24}).
		\item
		For $\lambda,\mu\in\C$, let $X_{17}(\lambda,\mu)=V(F_{17}(\lambda,\mu))$, where
		$
		F_{17}(\lambda,\mu)=x_1^3x_3+x_2^3x_4+\lambda x_5^4+\mu x_3x_4x_5^2+x_3^2x_4^2+x_5(x_3^3+x_4^3).
		$
		This form is smooth if and only if $\lambda\bigl(16\lambda^2-8\lambda\mu^2+36\lambda\mu-27\lambda+\mu^4-\mu^3\bigr)\neq0$.
		For every such $(\lambda,\mu)$, $\Aut(X_{17}(\lambda,\mu))$ contains the subgroup
		$
		G_{17}=\bigl\langle[\Diag(\xi_3,\xi_3,1,1,1)],\ [P_{(1\,2)(3\,4)}],\ [\Diag(\xi_9,\xi_9^{8},\xi_3^{2},\xi_3,1)]\bigr\rangle\cong C_3\times D_{18}
		$
		($\id{54}{3}$), of order $2\cdot3^{3}=54$.
	\end{enumerate}
	For the form $x_1^3x_2+x_2^4$, which occurs (in suitable variables) in (2), (5), (6), and (10), we have
	\begin{equation}\label{eq:A4-generators}
		\Aut(x_1^3x_2+x_2^4)=\left\langle iI_2,\ \Diag(\xi_3,1),\frac{1}{\sqrt3}\begin{pmatrix}-1&2\\1&1\end{pmatrix}\right\rangle\cong C_4. A_4.
	\end{equation}

	For the Klein quartic form $x_1^3x_2+x_2^3x_3+x_3^3x_1$, which occurs in (2) and (4), we have
	\begin{equation}\label{eq:klein-curve-generators}
		\Aut(x_1^3x_2+x_2^3x_3+x_3^3x_1)=\left\langle iI_3, \Diag(\xi_7,\xi_7^4,\xi_7^2),-\frac1{i\sqrt7}
		\begin{pmatrix}
			\xi_7-\xi_7^{6} & \xi_7^2-\xi_7^{5} & \xi_7^4-\xi_7^{3}\\
			\xi_7^2-\xi_7^{5} & \xi_7^4-\xi_7^{3} & \xi_7-\xi_7^{6}\\
			\xi_7^4-\xi_7^{3} & \xi_7-\xi_7^{6} & \xi_7^2-\xi_7^{5}
		\end{pmatrix}\right\rangle,
	\end{equation}
	which is isomorphic to $C_4\times\PSL(2,7)$ (see \cite[Sections~1, 4--5]{Kle79}).

	In (2), (4)--(6), (10), and (13), $F_i$ is a sum of forms in disjoint sets of variables, and $G_i$ is generated by the automorphisms of these forms, given by \eqref{eq:A4-generators}, \eqref{eq:klein-curve-generators}, and diagonal matrices, and by the permutations of variables interchanging copies of the same form. In (3) and (7)--(9), $G_i=\{[\Diag(A,1)]\mid A\in\Aut(h_i)\}$, where $h_i=F_i|_{x_5=0}$ as before Theorem~\ref{thm:quartic-surfaces} and
		\allowdisplaybreaks
	\begin{align*}
		\Aut(h_3)&=\Biggl\langle \Diag(i,-i,1,1),\ P_{(1\,2\,3\,4)},\ \frac12\begin{pmatrix}1&1&1&1\\1&1&-1&-1\\1&-1&1&-1\\1&-1&-1&1\end{pmatrix}\Biggr\rangle,\\[1ex]
	\Aut(h_7)&=\Biggl\langle \Diag(1,\xi_7,\xi_7^4,\xi_7^2),\ \frac{1}{\sqrt7}\begin{pmatrix}1&1&1&1\\2&\xi_7+\xi_7^6&\xi_7^2+\xi_7^5&\xi_7^3+\xi_7^4\\2&\xi_7^2+\xi_7^5&\xi_7^3+\xi_7^4&\xi_7+\xi_7^6\\2&\xi_7^3+\xi_7^4&\xi_7+\xi_7^6&\xi_7^2+\xi_7^5\end{pmatrix}\Biggr\rangle,\\[1ex]
		\Aut(h_8)&=\Biggl\langle \Diag(1,1,\xi_3^2,\xi_3),\ \begin{pmatrix}0&\xi_8^5&0&0\\\xi_8&0&0&0\\0&0&0&\xi_8\\0&0&\xi_8^5&0\end{pmatrix},\ \frac{1}{2\sqrt3}\begin{pmatrix}\sqrt3&-1&2&2\sqrt3\\1&\sqrt3&2\sqrt3&-2\\-1&\sqrt3&0&2\\\sqrt3&1&-2&0\end{pmatrix}\Biggr\rangle,\\[1ex]
		\Aut(h_9)&=\Biggl\langle\Diag(-i,i,-i,i), \begin{pmatrix}0&1&0&0\\1&0&0&0\\0&0&0&\xi_3\\0&0&\xi_3^2&0\end{pmatrix},  \frac{1}{2\sqrt3}\begin{pmatrix}-1&\sqrt3&1&\sqrt3\\\sqrt3&1&\sqrt3&-1\\2&2\sqrt3&-2&0\\2\sqrt3&-2&0&2\end{pmatrix}\Biggr\rangle.
	\end{align*}
	The smoothness of $F_i$ and the inclusions $G_i\subset\Aut(X_i)$ can be verified with Mathematica \cite{Wo}.
\end{example}

The following $7$ groups are isomorphic to subgroups of $G_{10}$, $G_{15}$, or $G_{17}$, but each gives a maximal projective action (Lemma~\ref{lem:24-irredundant}).
\begin{example}\label{ex:seven-projective-actions}
	We continue the numbering of Example~\ref{ex:17maximal}. In (18) and (20)--(23), the $t_i$ are complex parameters.
	\begin{enumerate}
		\setcounter{enumi}{17}
		\item
		Let $X_{18}=V(F_{18})$, where $F_{18}=x_1^3x_4+x_2^3x_4+x_3^3x_5+x_4^3x_5+x_5^4+t_1x_3^2x_4^2+t_2x_3x_4x_5^2+t_3x_1^2x_2x_3+t_4x_1x_2^2x_5.
		$ Then  $G_{18}=\langle[\Diag(\xi_9,\xi_9^4,\xi_3,\xi_3^2,1)]\rangle\cong C_9$ is a subgroup of $\Aut(X_{18})$.
		\item
		Let $X_{19}=V(F_{19})$, where $F_{19}=x_1^4+x_1^2x_2^2+x_2^4+x_2x_3^3+x_3^2x_4x_5+x_1(x_4^3+x_5^3)+x_2(x_4^3-x_5^3)$. Then $\Aut(X_{19})$ contains the subgroup $G_{19}\cong C_3\times S_3$ ($\id{18}{3}$) of order $2\cdot3^2=18$ generated by $[\Diag(1,1,\xi_3,\xi_3^2,\xi_3^2)]$, $[\Diag(1,1,1,\xi_3,\xi_3^2)]$, and $[\Diag(1,-1,-1,1,1)P_{(4\,5)}]$.
		\item
		Let $X_{20}=V(F_{20})$, where $F_{20}=x_1^3x_2+x_2^3x_4+x_3^3x_5+x_4^3x_5+x_5^4+t_1x_3^2x_4^2+t_2x_3x_4x_5^2$. Then $[\Diag(\xi_{27},\xi_{9}^{8},\xi_{3}^{2},\xi_{3},1)]$ generates a subgroup $G_{20}\cong C_{27}$ of $\Aut(X_{20})$.
		\item
		Let $X_{21}=V(F_{21})$, where $F_{21}=x_1^3x_2+x_2^3x_4+x_3^3x_4+x_4^3x_5+x_5^4+t_1x_2x_3^2x_5$. Then $[\Diag(\xi_{27},\xi_{9}^{8},\xi_{9}^{5},\xi_{3},1)]$ generates a subgroup $G_{21}\cong C_{27}$ of $\Aut(X_{21})$.
		\item
		Let $X_{22}=V(F_{22})$, where $F_{22}=x_1^3x_3+x_2^3x_4+x_3^3x_4+x_4^3x_5+x_5^4+t_1x_2x_3^2x_5$. Then $[\Diag(\xi_{27},\xi_{9}^{2},\xi_{9}^{8},\xi_{3},1)]$ generates a subgroup $G_{22}\cong C_{27}$ of $\Aut(X_{22})$.
		\item
		Let $X_{23}=V(F_{23})$, where $F_{23}=x_1^3x_3+x_2^3x_4+x_3^3x_4+x_4^4+x_5^4+t_1x_2^2x_3^2+t_2x_2x_3x_4^2$. Then $[\Diag(\xi_{9}^{7},\xi_{3},\xi_{3}^{2},1,i)]$ generates a subgroup $G_{23}\cong C_{36}$ of $\Aut(X_{23})$.
		\item
		Let $X_{24}=V(F_{24})$, where $F_{24}=x_1^4+8x_1x_2^3+3x_1x_3^2x_4+x_2x_3^3-4x_2x_4^3+x_5^4$. Then $\Aut(X_{24})$ contains $G_{24}\cong C_{12}.A_4$ ($\id{144}{36}$) of order $2^4\cdot3^2=144$ generated by
		$[\Diag(1,\xi_3,\xi_9^2,\xi_9^5,1)]$, $[\Diag(-i,-i,-i,-i,1)]$ and $[\Diag(B,B,1)]$,
		where $B=\frac{1}{\xi_3-\xi_3^2}\begin{pmatrix}1&2\\1&-1\end{pmatrix}.$
	\end{enumerate}
	The forms $F_{19}$ and $F_{24}$ are smooth, and so are the other forms $F_i$ for general $t_i$. This and the inclusions $G_i\subset\Aut(X_i)$ can be verified with Mathematica \cite{Wo}.
\end{example}

\begin{remark}\label{rmk:classic-surf}
	Several of the forms above, restricted to the hyperplane $x_5=0$, define classical quartic surfaces.
	The surfaces $V(h_3)$ and $V(h_6)$ are known as Burnside's quartic and Schur's quartic, respectively (see \cite[Section~272, Exercise~6]{Bur11} and \cite{Sch82}), and $V(h_7)$ is projectively equivalent to the $\PSL(2,7)$-invariant quartic studied by Edge (\cite[Section~39]{Edge47}).
	Moreover, $V(h_4)$, $V(h_3)$, $V(h_1)$, and $V(h_6)$ are projectively equivalent to the quartic surfaces in \cite[Example~(0.4), n\textsuperscript{o}~1, 4, 5, and~7]{Muk88}, respectively.
\end{remark}

\subsection{Smoothness and cyclic actions}\label{sec:sm-cyc}
We first recall some smoothness criteria for quartic forms and known results on cyclic group actions on smooth quartic threefolds.
\begin{lemma}[{\cite[Lemma~3.2 and Proposition~3.3]{OY19}}]\label{lem:quartic-sm}
	Let $F\in\C[x_1,\dots,x_5]$ be a quartic form.
	Then $F$ is not smooth if one of the following conditions holds:
	\begin{enumerate}
		\item There exists $i$ such that $x_i^3x_j\notin F$ for all $1\leq j\leq5$.
		\item There exist distinct indices $p,q$ such that $F\in(x_p,x_q)$.
		\item There exist distinct indices $p,q,r$ such that $F\in(x_p)+(x_q,x_r)^2$.
	\end{enumerate}
\end{lemma}
\begin{theorem}[{\cite[Theorem~4.8]{Zhe22}}]
	\label{thm:cyclic}
	A finite cyclic group acts faithfully on some smooth quartic threefold if and only if it is isomorphic to a subgroup of one of
	$$
	C_{56},\quad C_{61},\quad C_{72},\quad C_{80},\quad
	C_{81},\quad C_{84},\quad C_{108}.
	$$
\end{theorem}

For the classification of noncyclic abelian groups, we also need to rule out the following groups.
\begin{lemma}\label{lem:abelian-obstruc-7}
	None of the following seven groups acts faithfully on a smooth quartic threefold:
	$$
	C_3^3,\quad C_9^2,\quad C_{27}\times C_3,\quad C_2^3\times C_3^2,\quad C_8\times C_3^2,\quad C_2^2\times C_9\times C_3,\quad C_7\times C_3^2.
	$$
\end{lemma}
\begin{proof}
	Suppose that one of these groups $G$ acts faithfully on a smooth quartic threefold $X=V(F)$. By (2.3), $G$ has an $F$-lifting $\widetilde G\subset\Aut(F)$.
	Since $\widetilde G$ is abelian, we may assume that $\widetilde G$ consists of diagonal matrices and every monomial in $F$ is $\widetilde G$-invariant. For a diagonal matrix $A=\Diag(\xi_n^{s_1},\dots,\xi_n^{s_5})$,
	since $A(x_k)=\xi_n^{s_k}x_k$, the monomial $x_i^3x_j$ is $A$-invariant if and only if $3s_i+s_j\equiv0\mod n$.
	By Lemma~\ref{lem:quartic-sm}(1), for every $i$ there exists $j$ such that $x_i^3x_j\in F$, and this monomial is then $\widetilde G$-invariant.
	In most cases, we use these conditions to normalize the generators of $\widetilde G$ and to determine the $\widetilde G$-invariant quartic monomials, and then apply Lemma~\ref{lem:quartic-sm}(2) or~(3).
	We give the details for $C_9^2$ and $C_8\times C_3^2$. The remaining five groups are treated similarly, and for them we only indicate the main steps.
	
	{\it Case $G\cong C_9^2$.}
	Let $A_1=\Diag(\xi_9^{r_1},\dots,\xi_9^{r_5})$ and $A_2=\Diag(\xi_9^{s_1},\dots,\xi_9^{s_5})$ generate $\widetilde G$.
	Since $\widetilde G$ acts faithfully on $\C^5$, the pairs $(r_k,s_k)$, $1\leq k\leq5$, generate $(\Z/9)^2$, so two of them form a basis.
	Hence, after changing the generators and permuting the coordinates, we may assume $(r_1,s_1)=(1,0)$ and $(r_2,s_2)=(0,1)$.
	By Lemma~\ref{lem:quartic-sm}(1), applied to $x_1$ and $x_2$, two more coordinates have pairs $(6,0)$ and $(0,6)$, and, applied to these two coordinates, the fifth coordinate has pair $(0,0)$.
	Reordering the coordinates, we may take
	$$
	A_1=\Diag(\xi_9,\xi_3^{2},1,1,1),\qquad A_2=\Diag(1,1,\xi_9,\xi_3^{2},1).
	$$
	The $\widetilde G$-invariant quartic monomials are precisely $x_1^3x_2$, $x_2^3x_5$, $x_3^3x_4$, $x_4^3x_5$, and $x_5^4$.
	Thus $F\in(x_5)+(x_1,x_3)^2$, contradicting Lemma~\ref{lem:quartic-sm}(3).
	
	{\it Case $G\cong C_8\times C_3^2$.}
	Let $A_1,A_2,A_3$ be generators of $\widetilde G$ of orders $8,3,3$, respectively.
	By Lemma~\ref{lem:quartic-sm}(1), after permuting the coordinates and changing the generators, we may take
	$$
	A_1=\Diag(\xi_8,\xi_8^{5},\xi_8^{r_1},\xi_8^{r_2},\xi_8^{r_3}),\quad A_2=\Diag(1,1,\xi_3,1,\xi_3^{s_1}),\quad A_3=\Diag(1,1,1,\xi_3,\xi_3^{s_2}).
	$$
	If $r_1,r_2,r_3$ are all odd, then $A_1^4=-I_5$, contradicting $|\langle\pi(A_1)\rangle|=8$.
	Hence one of $r_i$ is even.
	Let $x_k$ be a coordinate at which the exponent of $A_1$ is even, and choose $j$ with $x_k^3x_j\in F$ by Lemma~\ref{lem:quartic-sm}(1).
	Invariance of $x_k^3x_j$ under $A_1$ forces the exponent of $A_1$ at $x_j$ to be even, and since $A_2,A_3$ have order $3$, invariance under them forces $A_2(x_j)=A_3(x_j)=x_j$.
	As the exponents of $A_1$ at $x_1$ and $x_2$ are odd, while $A_2(x_3)\neq x_3$ and $A_3(x_4)\neq x_4$, we get $j=5$.
	Thus $s_1=s_2=0$ and $r_3$ is even. Since $\xi_4I_5\in\Aut(F)$, replacing $A_1$ by a suitable odd power of $\xi_4^{-r_3/2}A_1$, we may assume $r_3=0$.
	
	Now $A_2=\Diag(1,1,\xi_3,1,1)$ and $A_3=\Diag(1,1,1,\xi_3,1)$, so a quartic monomial is invariant under $A_2$ and $A_3$ if and only if it is a monomial in $x_1,x_2,x_5$ or has the form $x_3^3x_j$ or $x_4^3x_j$ with $j\in\{1,2,5\}$.
	As the exponents of $A_1$ at $x_1,x_2,x_5$ are $1,5,0$, the $A_1$-invariant quartic monomials in $x_1,x_2,x_5$ are $x_1^3x_2$, $x_1x_2^3$, and $x_5^4$.
	Moreover, since these exponents are distinct, at most one monomial $x_3^3x_j$ with $j\in\{1,2,5\}$ is $A_1$-invariant, and likewise for $x_4$. By Lemma~\ref{lem:quartic-sm}(1) such monomials exist, which gives $r_1,r_2\in\{0,1,5\}$.
	If $r_1=r_2$, the same $j$ occurs for $x_3$ and $x_4$, and hence $F\in(x_5)+(x_1,x_2)^2$, $F\in(x_1)+(x_2,x_5)^2$, and $F\in(x_2)+(x_1,x_5)^2$ for $j=5$, $1$, and $2$, respectively,  contradicting Lemma~\ref{lem:quartic-sm}(3).
	Hence $r_1\neq r_2$.
	Replacing $A_1$ by $A_1^5$ and interchanging $x_1$ and $x_2$ if necessary, and then interchanging $x_3$ and $x_4$ together with $A_2$ and $A_3$ if necessary, we may assume $r_1=5$ and $r_2\in\{0,1\}$.
	Then the $\widetilde G$-invariant quartic monomials are exactly $x_1^3x_2$, $x_1x_2^3$, $x_3^3x_1$, $x_4^3x_j$, and $x_5^4$, where $j=5$ if $r_2=0$ and $j=2$ if $r_2=1$.
	In either case, $F\in(x_1)+(x_4,x_5)^2$, contradicting Lemma~\ref{lem:quartic-sm}(3).
	
	{\it Case $G\cong C_3^3$.}
	After permuting the coordinates and changing the generators, we may take the generators of $\widetilde G$ to be
	$
	\Diag(\xi_3,1,1,\xi_3^{a_1},\xi_3^{a_2})$, $\Diag(1,\xi_3,1,\xi_3^{a_3},\xi_3^{a_4})$,
	$\Diag(1,1,\xi_3,\xi_3^{a_5},\xi_3^{a_6})$,
	where $a_1,\dots,a_6\in\{0,1,2\}$.
	A quartic monomial in $x_1,x_2,x_3$ is invariant only if its degree in each of $x_1,x_2,x_3$ is divisible by $3$, which is impossible.
	Hence $F\in(x_4,x_5)$, contradicting Lemma~\ref{lem:quartic-sm}(2).
	
	{\it Case $G\cong C_{27}\times C_3$.}
	We may take the generators of $\widetilde G$ to be
	$\Diag(\xi_{27},\xi_9^8,\xi_3,1,\xi_{27}^b)$ and $\Diag(1,1,1,1,\xi_3)$,
	where $b\in\{0,1,6\}$. Applying Lemma~\ref{lem:quartic-sm}(1) to $x_5$ gives $b\equiv0,1,6\pmod9$, and multiplying the first generator by a power of the second changes $b$ by a multiple of $9$.
	For $b=0,1,6$, the $\widetilde G$-invariant quartic monomials show that $F\in(x_4)+(x_1,x_2)^2$, $F\in(x_2)+(x_3,x_4)^2$, and $F\in(x_3)+(x_1,x_4)^2$, respectively, contradicting Lemma~\ref{lem:quartic-sm}(3).
	
	{\it Case $G\cong C_2^3\times C_3^2$.}
	Let $\widetilde G_2$ and $\widetilde G_3$ be the Sylow $2$- and $3$-subgroups of $\widetilde G$.
	After permuting the coordinates and changing the generators, we may take the generators of $\widetilde G_3$ to be
	$\Diag(\xi_3,1,\xi_3^{a_1},\xi_3^{a_2},\xi_3^{a_3})$ and $\Diag(1,\xi_3,\xi_3^{a_4},\xi_3^{a_5},\xi_3^{a_6})$,
	where $a_1,\dots,a_6\in\{0,1,2\}$.
	By Lemma~\ref{lem:quartic-sm}(1), $x_1^3x_j,x_2^3x_k\in F$ for some $j,k$.
	Invariance under $\widetilde G_3$ shows that $\widetilde G_3$ fixes $x_j$ and $x_k$, so $j,k\in\{3,4,5\}$, and invariance under $\widetilde G_2$ shows that each element of $\widetilde G_2$, a diagonal matrix with entries $\pm1$, has equal entries at $x_1$ and $x_j$ and at $x_2$ and $x_k$.
	Such matrices form a group of order $8$ containing $-I_5$, whose image under $\pi$ has order $4$; this contradicts $\pi(\widetilde G_2)\cong C_2^3$.
	
	{\it Case $G\cong C_2^2\times C_9\times C_3$.}
	Let $\widetilde G_2$ be the Sylow $2$-subgroup of $\widetilde G$.
	Some element of order $9$ of $\widetilde G$ acts on some coordinate, say $x_1$, by $\xi_9$, and by Lemma~\ref{lem:quartic-sm}(1) we may number the coordinates so that $x_1^3x_2,x_2^3x_3\in F$.
	By invariance of these monomials, the elements of $\widetilde G_2$ have equal entries $\pm1$ at $x_1$, $x_2$, $x_3$.
	Hence, after changing the generators and multiplying some of them by $-1$, which keeps them in $\Aut(F)$, we may take the generators of $\widetilde G$ to be
	$\Diag(\xi_9,\xi_3^2,1,\xi_9^{a_1},\xi_9^{a_2})$, $\Diag(1,1,1,\xi_3^{a_3},\xi_3^{a_4})$,
	$\Diag(1,1,1,-1,1)$, and $\Diag(1,1,1,1,-1)$,
	where $a_1,a_2\in\{0,\dots,8\}$ and $a_3,a_4\in\{0,1,2\}$.
	By Lemma~\ref{lem:quartic-sm}(1) and invariance under the last two generators, $x_4^4,x_5^4\in F$, so $a_1=\cdots=a_4=0$ and the second generator is $I_5$, a contradiction.
	
	{\it Case $G\cong C_7\times C_3^2$.}
	We may assume that $\widetilde G$ is generated by
	$$
	\Diag(\xi_7,\xi_7^4,\xi_7^2,\xi_7,\xi_7^b),\quad\Diag(1,1,1,\xi_3,1),\quad\text{and}\quad\Diag(1,1,1,1,\xi_3),
	$$
	where $b\in\{1,2,4\}$ by Lemma~\ref{lem:quartic-sm}(1) applied to $x_5$.
	For each such $b$, the $\widetilde G$-invariant quartic monomials show that $F\in(x_2)+(x_3,x_5)^2$, contradicting Lemma~\ref{lem:quartic-sm}(3).
\end{proof}

\subsection{The abelian groups}
Many of the forms $F_i$ above are sums of the following forms. For $\ell\geq1$, we use
$$
K_\ell=\sum_{j=1}^{\ell}y_j^3y_{j+1}\quad(y_{\ell+1}:=y_1),
\qquad
T_\ell=\sum_{j=1}^{\ell-1}y_j^3y_{j+1}+y_\ell^4.
$$
In particular, $K_1=T_1=y_1^4$.
When we add such forms, we always take them in disjoint sets of variables. For example, $F_1=5K_1$, $F_{10}=T_2+T_3$, $F_{12}=K_5$, and $F_{16}=T_5$.
For a form $F$ in $r$ variables, the {\it projective diagonal group} of $F$ is
$$
G_F^{\rm d}:=\{[\Diag(\lambda_1,\dots,\lambda_r)]\mid\Diag(\lambda_1,\dots,\lambda_r)\in\Aut(F)\}\subset G_F.
$$

\begin{theorem}\label{thm:maximal-abelian}
	A finite abelian group acts faithfully on some smooth quartic threefold if and only if it is isomorphic to a subgroup of one of the following $16$ groups:
	$$
	\begin{gathered}
		C_4^4,\quad C_{12}\times C_4^2,\quad C_{36}\times C_4,\quad C_{12}^2,\quad C_8\times C_4^2,\quad C_{28}\times C_4,\quad C_{108},\quad C_{36}\times C_3,\\
		C_{24}\times C_4,\quad C_{84},\quad C_{81},\quad C_{80},\quad C_{72},\quad C_8^2,\quad C_{61},\quad C_{56}.
	\end{gathered}
	$$
\end{theorem}
\begin{proof}
	Suppose that a finite abelian group $G$ acts faithfully on $X=V(F)$ and let $\widetilde G$ be an $F$-lifting of $G$. We assume that $\widetilde G$ consists of diagonal matrices.
	Let $\widetilde G_3$ be the Sylow $3$-subgroup of $\widetilde G$ and $\widetilde G_{3'}$ the subgroup of elements of order prime to $3$, so that $\widetilde G=\widetilde G_{3'}\times\widetilde G_3$.

	By Lemma~\ref{lem:quartic-sm}(1), for each $i$ we can choose a monomial $x_i^3x_j\in F$, and then $\lambda_j=\lambda_i^{-3}$ for every $\Diag(\lambda_1,\dots,\lambda_5)\in\widetilde G$.
	Under repeated application of $x_i\mapsto x_j$, every coordinate eventually reaches a cycle. After renumbering, the coordinates on cycles are $x_1,\dots,x_r$, their chosen monomials form $K:=K_{r_1}+\cdots+K_{r_s}$ with $r_1+\cdots+r_s=r$, and each of $x_{r+1},\dots,x_5$ leads to $\{x_1,\dots,x_r\}$ in at most $5-r$ steps.
	
	{\it Case 1: $\widetilde G_3$ is cyclic.}
	Let $A=\Diag(\lambda_1,\dots,\lambda_5)\in\widetilde G_3$.
	If $x_i$ lies on a cycle of length $\ell$, then $\lambda_i^{(-3)^\ell}=\lambda_i$. Since $(-3)^\ell-1$ is prime to $3$ and the order of $\lambda_i$ is a power of $3$, we have $\lambda_i=1$.
	If $x_i$ does not lie on a cycle, then it reaches some $x_k$ with $k\leq r$ after $m\leq5-r$ steps, so $\lambda_i^{(-3)^m}=\lambda_k=1$.
	Thus $\lambda_i^{3^{5-r}}=1$ for all $i$, and the cyclic group $\widetilde G_3$ is isomorphic to a subgroup of $C_{3^{5-r}}$.
	Now let $A=\Diag(\lambda_1,\dots,\lambda_5)\in\widetilde G_{3'}$.
	If $x_i^3x_j$ is the chosen monomial, then $\lambda_i$ is the unique cube root of $\lambda_j^{-1}$ of order prime to $3$.
	Thus $A$ is determined by $\lambda_1,\dots,\lambda_r$, and $\Diag(\lambda_1,\dots,\lambda_r)\in\Aut(K)$.
	If $\lambda_1=\cdots=\lambda_r=c$, then $c^4=1$ and hence $\lambda_i=c$ for all $i$, i.e., $A=cI_5$. Thus $A=I_5$, since $\pi$ is injective on $\widetilde G$.
	Therefore, $A\mapsto[\Diag(\lambda_1,\dots,\lambda_r)]$ embeds $\widetilde G_{3'}$ into $G_K^{\rm d}$, and $G$ is isomorphic to a subgroup of $G_K^{\rm d}\times C_{3^{5-r}}$.
	
	The diagonal matrices in $\Aut(K_\ell)$ form a cyclic group of order $|(-3)^\ell-1|$, i.e., of order $4$, $8$, $28$, $80$, $244$ for $\ell=1,\dots,5$, containing $\xi_4I_\ell$.
	Hence $G_K^{\rm d}\cong(C_{n_1}\times\cdots\times C_{n_s})/C_4$, where $n_j=|(-3)^{r_j}-1|$ and $C_4$ is embedded diagonally.
	By direct computation for all partitions $(r_1,\dots,r_s)$, the group $G_K^{\rm d}\times C_{3^{5-r}}$ is isomorphic to a subgroup of one of the $16$ groups, except for the partitions $(4)$ and $(3)$, where it is isomorphic to $C_{60}$ and $C_{63}$, respectively.
	In these two cases, $G$ is cyclic, and by Theorem~\ref{thm:cyclic}, $G$ is isomorphic to a subgroup of $C_{20}$ or $C_{12}$ (resp.\ $C_{21}$ or $C_9$), each of which embeds into one of the $16$ groups.
	
{\it Case 2: $\widetilde G_3$ is noncyclic.}
Every noncyclic abelian $3$-group other than $C_3^2$ and $C_9\times C_3$ contains $C_3^3$, $C_9^2$, or $C_{27}\times C_3$, so $\widetilde G_3\cong C_3^2$ or $C_9\times C_3$ by Lemma~\ref{lem:abelian-obstruc-7}.
By Theorem~\ref{thm:cyclic}, every prime divisor of $|G|$ lies in $\{2,3,5,7,61\}$. Since $\widetilde G_3$ contains $C_3^2$, Lemma~\ref{lem:abelian-obstruc-7} excludes $7$, and Theorem~\ref{thm:cyclic} excludes $5$ and $61$, since neither $C_{15}$ nor $C_{183}$ is a subgroup of a group listed there.
Lemma~\ref{lem:abelian-obstruc-7} also excludes $C_2^3\times C_3^2$ and $C_8\times C_3^2$, so the Sylow $2$-subgroup of $G$ embeds into $C_4^2$. If $\widetilde G_3\cong C_9\times C_3$, it even embeds into $C_4$, since the lemma also excludes $C_2^2\times C_9\times C_3$. Hence $G$ embeds in $C_4^2\times C_3^2\cong C_{12}^2$ or in $C_4\times C_9\times C_3\cong C_{36}\times C_3$.

	Conversely, each of the $16$ groups is the projective diagonal group of a sum of the forms $K_\ell$ and $T_\ell$ in five variables in total, and such a sum is smooth because the forms $K_\ell$ and $T_\ell$ are.
	Finally, none of the $16$ groups is isomorphic to a subgroup of another.
\end{proof}

\begin{proposition}\label{prop:abelian-3-groups}
	Let $F$ be a smooth quartic form in five variables, and let $H\subset G_F$ be a nontrivial abelian $3$-group. Then $H$ has a unique $F$-lifting $\widetilde H$, which is conjugate in $\GL(5,\C)$ to the group generated by the matrices in one of the following cases. Moreover, $H$ is conjugate in $\PGL(5,\C)$ to a subgroup of one of $G_1$, $G_2$, $G_5$, $G_6$, $G_{10}$, $G_{13}$, $G_{15},\dots,G_{22}$, and $G_{24}$.
	\begin{enumerate}
		\item $H\cong C_3$:
		$\Diag(\xi_3,\xi_3^a,\xi_3^b,1,1)$,
		where $(a,b)\in\{(0,0),\allowbreak(1,0),\allowbreak(1,2),\allowbreak(2,0)\}$.
		
		\item $H\cong C_9$:
		$\Diag(\xi_9,\xi_3^2,\xi_9^a,\xi_3^b,1)$,
		where $(a,b)\in\{(0,0), (2,1), (3,0), (4,0), (4,1), (6,0), (8,1)\}$.
		
		\item $H\cong C_{27}$:
		$\Diag(\xi_{27},\xi_9^8,\xi_3,1,\xi_{27}^a)$,
		where $a\in\{0,6,15,18\}$.
		
		\item $H\cong C_{81}$:
		$\Diag(\xi_{81},\xi_{27}^{26},\xi_9,\xi_3^2,1)$.
		
		\item $H\cong C_3^2$:
		$\Diag(\xi_3,1,\xi_3^a,1,1)$ and $\Diag(1,\xi_3,\xi_3^b,1,1)$,
		where $(a,b)\in\{(0,0),\allowbreak(0,2),\allowbreak(1,1)\}$.
		
		\item $H\cong C_9\times C_3$:
		$\Diag(\xi_9,\xi_9^{8a},\xi_3^2,\xi_3^a,1)$ and $\Diag(1,\xi_3,1,1,1)$,
		where $a\in\{0,1\}$.
	\end{enumerate}
\end{proposition}
\begin{proof}
	By Theorem~\ref{thm:maximal-abelian}, $H$ is isomorphic to $C_3$, $C_9$, $C_{27}$, $C_{81}$, $C_3^2$, or $C_9\times C_3$. We may assume that $\widetilde H$ consists of diagonal matrices. By Lemma~\ref{lem:quartic-sm}(1), for each $k$ there is $j$ such that $x_k^3x_j\in F$, and then $\lambda_j=\lambda_k^{-3}$ for every $\Diag(\lambda_1,\dots,\lambda_5)\in\widetilde H$. Using these relations and Lemma~\ref{lem:quartic-sm}(2)--(3), a direct computation shows that $\widetilde H$ is conjugate to one of the listed groups. For the second assertion, one checks that, up to changes of generators and linear change of coordinates, the classes of the listed generators lie in one of these groups (Examples~\ref{ex:17maximal} and~\ref{ex:seven-projective-actions}).
\end{proof}

\section{The nonmonomial cases}\label{sec:nonmonomial}
Let $X$ be a smooth quartic threefold. Following \cite[Section~3]{YYZ25}, we fix a decomposition of $\C^5$ into blocks that are permuted by $\Aut(X)$ such that the stabilizer of each block acts primitively on it. The dimensions of these blocks, in nonincreasing order, form the subdegree sequence $l(\Aut(X))$ (see Definition~\ref{def:primitive-constituents}). If $l(\Aut(X))=(1,1,1,1,1)$, then $\Aut(X)$ is monomial; this case is treated in Section~\ref{sec:monomial}. In this section, we prove the following theorem, which covers the other cases.

\begin{theorem}
	\label{thm:nonmonomial}
	Let $X$ be a smooth quartic threefold with $l(\Aut(X))\neq(1,1,1,1,1)$.
	Then $\Aut(X)$ is conjugate to a subgroup of one of the groups $G_i$ of Examples~\ref{ex:17maximal} and~\ref{ex:seven-projective-actions} listed in the row of $l(\Aut(X))$ in the following table:
	$$
	\begin{array}{c|l}
		\text{subdegree sequence}&\text{possible ambient groups}\\ \hline
		(5)&\text{none}\\
		(4,1)&G_3,\ G_7,\ G_8,\ G_9\\
		(3,2)&G_2\\
		(3,1,1)&G_4\\
		(2,2,1)&G_6,\ G_{14},\ G_{24}\\
		(2,1,1,1)&G_2,\ G_5,\ G_6,\ G_{10}
	\end{array}
	$$
\end{theorem}
We determine the possible primitive constituents in Section~\ref{sec:primitive-constituents} (Proposition~\ref{prop:possi-primitive-cons}) and prove Theorem~\ref{thm:nonmonomial} in Section~\ref{sec:proof-nonmonomial}.
\subsection{Primitive constituents}
\label{sec:primitive-constituents}
We recall some notions from \cite[Section~3.1]{YYZ25} (see also the references therein).
\begin{definition}
	Let $\rho\colon G\to\GL(r,\C)$ be an irreducible representation of a finite group $G$. We say that $\rho$ is {\it primitive} (resp.\ {\it imprimitive}) if $\C^r$ cannot (resp.\ can) be decomposed as a direct sum of at least two nonzero subspaces that are permuted by $G$. A finite subgroup of $\GL(r,\C)$ is {\it primitive} if the underlying representation is primitive, and {\it monomial} if it consists of monomial matrices in a suitable basis of $\C^r$. A finite subgroup $H\subset\PGL(r,\C)$ is {\it primitive} (resp.\ {\it monomial}) if $H=\pi(\widetilde H)$ for some finite primitive (resp.\ monomial) subgroup $\widetilde H\subset\GL(r,\C)$.
\end{definition}

Let $X=X_F$ be a smooth quartic threefold, let $G\subset\Aut(X)$, and let $\widetilde G\subset\Aut(F)$ be an $F$-lifting of $G$. As in \cite[Setup~3.2]{YYZ25}, we fix a decomposition
\begin{equation}\label{eq:blocks}
	\C^5=W_1\oplus\cdots\oplus W_s,
\end{equation}
obtained by decomposing $\C^5$ into irreducible $\widetilde G$-stable subspaces and each of these into subspaces permuted transitively by $\widetilde G$, such that the stabilizer ${\rm Stab}_{\widetilde G}(W_a)=\{g\in\widetilde G\mid gW_a=W_a\}$ acts primitively, though not necessarily faithfully, on the {\it block} $W_a$. Replacing $F$ by an isomorphic form, we assume that $W_1,\dots,W_s$ are spanned by $e_1,\dots,e_5$ in this order.

\begin{definition}\label{def:primitive-constituents}
	Let $r_a=\dim W_a$ and
	$$
	H_a=\{[g|_{W_a}]\mid g\in{\rm Stab}_{\widetilde G}(W_a)\}\subset\PGL(r_a,\C).
	$$
	We call $H_1,\dots,H_s$ the {\it primitive constituents} of $G$. The sequence of the $r_a$ in nonincreasing order is called the {\it subdegree sequence} $l(G)$ of $G$. The subgroup
	$$
	P=\bigcap_{a=1}^s{\rm Stab}_{\widetilde G}(W_a)
	$$
	is called the {\it principal subgroup} of $G$.
\end{definition}
Let $\Pi\subset S_s$ be the group of permutations of $W_1,\dots,W_s$ induced by $\widetilde G$. Let $J\subset\prod_{a=1}^sH_a$ be the image of $P$ under $g\mapsto([g|_{W_1}],\dots,[g|_{W_s}])$, and let $N$ be the kernel of this map. We have the exact sequences
\begin{align}
	&1\longrightarrow P\longrightarrow\widetilde G\longrightarrow \Pi\longrightarrow1,\label{eq:first-sequence} \\
	&1\longrightarrow N\longrightarrow P\longrightarrow J\longrightarrow1.
\end{align}
The projection $J_a$ of $J$ to $H_a$ is a normal subgroup of $H_a$. Every element of $N$ is a block-diagonal matrix $\Diag(\lambda_1I_{r_1},\dots,\lambda_sI_{r_s})$ with $\lambda_a\in\C^*$. In particular, $N$ is abelian.

By (2.3), any two $F$-liftings of $G$ differ by scalar matrices. Hence \eqref{eq:blocks} is a block decomposition for every $F$-lifting of $G$, and the groups $H_a$, $\Pi$, $J$, $J_a$, and $P'$ do not depend on the choice of $\widetilde G$; only $P$ and $N$ do. We choose $\widetilde G$ conveniently in each proof. If $G=\Aut(X)$, then $\Aut(F)=\widetilde G\times\langle\xi_4I_5\rangle$. In particular, $[A|_{W_a}]\in H_a$ for every $A\in\Aut(F)$ with $AW_a=W_a$.

\begin{example}\label{ex:221-blocks}
	For $G_6$, $G_{14}$, and $G_{24}$, the blocks are $\langle e_1,e_2\rangle$, $\langle e_3,e_4\rangle$, and $\langle e_5\rangle$, so the subdegree sequence is $(2,2,1)$, and both binary primitive constituents are isomorphic to $A_4$. For $G_6$, $J$ is their product, $N\cong C_4^2$, and $\Pi\cong C_2$. For $G_{14}$ and $G_{24}$, $J$ is the graph of an isomorphism between them, while $N\cong C_8$ and $C_{12}$, and $\Pi\cong C_2$ and $1$, respectively.
\end{example}

Let $\C^5=U\oplus W$, where $U$ and $W$ are each spanned by some of $e_1,\dots,e_5$ (for example, sums of blocks), and let $y_1,\dots,y_m$ be the coordinates on $W$. Write $F=\sum_{b=0}^4F_{4-b,b}$, where $F_{4-b,b}$ has degree $b$ in $y_1,\dots,y_m$, and $F_{3,1}=\sum_{j=1}^mC_jy_j$ with cubic forms $C_j$ on $U$. We set
$$
\mathcal{C}_F(U,W)={\rm Span}_\C\{C_1,\dots,C_m\},
$$
and note that $C_j$ is the restriction of $\partial F/\partial y_j$ to $U$.

The following lemma is a variant of \cite[Lemmas~5.1--5.2]{YYZ25}, without the assumption $\dim U\geq2$. For the reader's convenience, we include the short proof.

\begin{lemma}\label{lem:normal-derivatives}
	Let $\widetilde H\subset\GL(5,\C)$ be a subgroup preserving $U$, $W$, and the line $\C F$. Then $(A|_U)(\mathcal{C}_F(U,W))=\mathcal{C}_F(U,W)$ for all $A\in\widetilde H$, and $\dim \mathcal{C}_F(U,W)\leq\dim W$. If $F$ is smooth and $\mathcal{C}_F(U,W)=0$, then $F|_U$ is smooth.
\end{lemma}
\begin{proof}
	Let $A\in\widetilde H$ with $A(F)=cF$, where $c\in\C^*$. Since $A$ preserves $U$ and $W$, it preserves the degree in $y_1,\dots,y_m$, so $A(F_{3,1})=cF_{3,1}$. Writing $A(y_j)=\sum_ka_{jk}y_k$, we get $\sum_ja_{jk}(A|_U)(C_j)=cC_k$ for all $k$. Hence $\mathcal{C}_F(U,W)\subset(A|_U)(\mathcal{C}_F(U,W))$, with equality since both spaces have the same dimension. The bound $\dim\mathcal{C}_F(U,W)\leq\dim W$ holds since $\mathcal{C}_F(U,W)$ is spanned by the $m=\dim W$ forms $C_1,\dots,C_m$. If $\mathcal{C}_F(U,W)=0$, then each $\partial F/\partial y_j$ vanishes on $U$, while the partial derivatives of $F$ with respect to the coordinates of $U$ restrict to those of $F|_U$. Hence a nonzero common zero of the partial derivatives of $F|_U$ would be a singular point of $F$.
\end{proof}
\begin{corollary}\label{cor:involution-split}
	Let $F$ be a smooth quartic form, and let $\tau\in\Aut(F)$ be an involution whose eigenspaces are $U$ and $W$. Then $F=F_{4,0}+F_{2,2}+F_{0,4}$, the forms $F|_U$ and $F|_W$ are smooth, and every element of $\Aut(F)$ commuting with $\tau$ lies in $\Aut(F|_U)\times\Aut(F|_W)$.
\end{corollary}
\begin{proof}
	Since $\tau(F_{4-b,b})=(-1)^bF_{4-b,b}$, we have $F_{3,1}=F_{1,3}=0$. Hence $\mathcal C_F(U,W)=\mathcal C_F(W,U)=0$, and $F|_U$ and $F|_W$ are smooth by Lemma~\ref{lem:normal-derivatives}. An element of $\Aut(F)$ commuting with $\tau$ preserves $U$ and $W$, hence also $F_{4,0}=F|_U$ and $F_{0,4}=F|_W$.
\end{proof}
\begin{lemma}\label{lem:primitive-5}
	No smooth quartic threefold $X$ satisfies $l(\Aut(X))=(5)$.
\end{lemma}
\begin{proof}
	Let $G=\Aut(X)$ with $l(G)=(5)$. Then $G\subset\PGL(5,\C)$ is primitive, so by \cite[Section~9, (9A)]{Bra67} (see also \cite[Section~5]{CMS26}), $G$ is isomorphic to one of the following $11$ groups:
	$$
		\PSL(2,11), S_5, A_6, S_6, \PSp(4,3), C_5^2\rtimes H \text{ with } H\in\{C_3, C_6, Q_8, C_3\rtimes C_4, \SL(2,3), \SL(2,5)\}.
	$$
	Since $\PSL(2,11)\supset C_{11}$, $\PSp(4,3)\supset C_3^3$, and $C_5^2\rtimes H\supset C_5^2$, Theorem~\ref{thm:maximal-abelian} excludes $8$ of these groups. If $G\cong S_5$, $A_6$, or $S_6$, then $X$ is projectively equivalent to a member of the pencil of $S_6$-invariant quartic threefolds \cite[Remark~4.5 and Lemma~4.6]{CS16}, all of which are singular \cite[pp.~573--574]{CS16}.
\end{proof}

Up to conjugation, $\PGL(4,\C)$ contains exactly two primitive subgroups isomorphic to $S_5$, and likewise two isomorphic to $A_5$ (see, e.g.,~\cite[Section~2]{CS19} and~\cite[Section~6]{AOT24}). The {\it standard} $S_5$ is the image in $\PGL(4,\C)$ of the representation of $S_5$ on the hyperplane $t_1+\cdots+t_5=0$ of $\C^5$, where $S_5$ permutes the coordinates $t_1,\dots,t_5$. The {\it spin} $S_5$ is the image of a faithful four-dimensional representation of a double cover of $S_5$ whose restriction to the preimage $\SL(2,5)$ of $A_5$ is irreducible. Their subgroups of index $2$ are the standard and spin $A_5$. Both $S_5$'s occur in Example~\ref{ex:17maximal}: $G_{h_3}\cong C_2^4\rtimes S_5$ contains a standard $S_5$ (Lemma~\ref{lem:primitive-surfaces}), and $G_{h_8}$ and $G_{h_9}$ are spin $S_5$'s, since the derived subgroups of $\Aut(h_8)$ and $\Aut(h_9)$ are isomorphic to $\SL(2,5)$ and act irreducibly on $\C^4$. Thus $G_8$ and $G_9$ are central extensions of a spin $S_5$ by $\langle[\Diag(\xi_4I_4,1)]\rangle\cong C_4$.

\begin{proposition}\label{prop:possi-primitive-cons}
	Let $X$ be a smooth quartic threefold, and let $H\subset\PGL(r,\C)$ be a primitive constituent of $\Aut(X)$ with $r\geq2$.
	Then, up to isomorphism of $H$, the pair $(r,H)$ is one of
	$$
	(4,S_5),\quad(4,A_5),\quad(4,\PSL(2,7)),\quad(4,C_2^4\rtimes S_5),\quad(3,\PSL(2,7)),\quad(2,A_4).
	$$
\end{proposition}
\begin{proof}
Let $X=V(F)$, and assume that $H=H_1$. Take $U=W_1$ and $W=W_2\oplus\cdots\oplus W_s$, so that $U$ and $W$ are spanned by $e_1,\dots,e_r$ and $e_{r+1},\dots,e_5$. Let $S={\rm Stab}_{\widetilde G}(U)$ and $B=\{g|_U\mid g\in S\}\subset\GL(U)$. Then $H=\pi(B)$ is primitive, so $B$ is irreducible, and $S$ preserves $W$. Let $F_U=F|_U$ and $\mathcal C=\mathcal{C}_F(U,W)$.
	
Since $S$ preserves $U$ and $W$, it preserves each $F_{4-b,b}$, $b\in\{0,\dots,4\}$. By Lemma~\ref{lem:normal-derivatives}, $\mathcal C$ is $B$-stable of dimension at most $5-r$, and $F_U$ is smooth if $\mathcal C=0$. If $r\geq3$, then $F_U\neq0$, since otherwise $F\in(x_4,x_5)$, contrary to Lemma~\ref{lem:quartic-sm}(2). Finally, since $\pi^{-1}(H)=\C^*B$ and scalars are central, $B$ and $\pi^{-1}(H)$ have the same derived subgroup $B'$, which is the image of $S'$ under restriction to $U$.
	
{\it Case $r=4$.} We write $F=cx_5^4+x_5^3Q_1+x_5^2Q_2+x_5Q_3+F_U$, where each $Q_i$ is a form of degree $i$ on $U$. Since $S$ acts on $\C x_5$ by scalars, each nonzero $Q_i$ spans a $B$-stable line. If $Q_1\neq0$, then $\ker Q_1$ is a $B$-stable hyperplane of $U$, contradicting irreducibility. Hence $Q_1=0$. Lemma~\ref{lem:quartic-sm}(1) then gives $c\neq0$. Thus we may assume $F=x_5^4+x_5^2Q_2+x_5Q_3+F_U$.
	
If $H$ is the standard $A_5$ or $S_5$, then $B'$ is the standard $A_5$, so $Q_2,Q_3,F_U$ are $A_5$-invariant. Every such $A_5$-invariant form is also $S_5$-invariant. Indeed, by Molien's formula, on the hyperplane $t_1+\cdots+t_5=0$, the $A_5$-invariant quadratic, cubic, and quartic forms are respectively spanned by the restrictions of $p_2$, $p_3$, and $p_2^2,p_4$, where $p_d=\sum_{i=1}^5t_i^d$. Thus $\Diag(\sigma,1)$ preserves $F$ for every standard $\sigma\in S_5$, and $H$ is the standard $S_5$.
	
Suppose that $H$ is neither of these standard groups. The inclusion diagrams in \cite[Appendix~A]{CS19} show that, after conjugation, $H$ contains a spin $A_5$, the primitive $\PSL(2,7)$, the tensor-product $A_4\times A_4$, or the standard $C_2^4\rtimes C_5$. In each case, $-I_4$ belongs to the derived subgroup of $\pi^{-1}(H)$. Indeed, for the first two groups, $\pi^{-1}(H)$ contains the perfect double cover $\SL(2,5)$ or $\SL(2,7)$, whose center is $\langle-I_4\rangle$. For $A_4\times A_4$, it contains $\SL(2,3)\otimes I_2$, whose derived subgroup $Q_8\otimes I_2$ contains $-I_4$. Finally, for $C_2^4\rtimes C_5$, the preimage of $C_2^4$ is nonabelian, since diagonal subgroups of $\PGL(4,\C)$ of exponent $2$ have order at most $8$, and its commutators are scalar matrices with square $I_4$, so one of them is $-I_4$. Hence $-I_4\in B'$. As $S$ acts on $\C x_5$ by scalars, $S'$ acts trivially on $\C x_5$, so the preimage of $-I_4$ in $S'$ is $\Diag(-I_4,1)$. Hence $\Diag(-I_4,1)\in\Aut(F)$. Consequently $Q_3=0$, and $F_U$ is smooth by Corollary~\ref{cor:involution-split}.
	
If $Q_2\neq0$, its radical is $B$-stable, so $Q_2$ is nondegenerate. By \cite[Appendix~A, diagram~(A.1)]{CS19}, $H$ must then contain the tensor-product group $A_4\times A_4$, contradicting \cite[Theorem~8.1]{AOT24} because $F_U$ is smooth. Thus $Q_2=0$ and $F=x_5^4+F_U$. Every automorphism of $F_U$ extends by fixing $x_5$, so $H=G_{F_U}$. The possibilities follow from \cite[Theorems~8.1--8.3 and Corollary~8.5]{AOT24}.
	
{\it Case $r=3$.} By Blichfeldt's classification \cite[Chapter~V, \S~79, 82--83]{Bli17} (see also \cite[Theorem~4.8]{DI09}), $H$ is isomorphic to $A_5$, $\PSL(2,7)$, $A_6$, or one of the Hessian groups of orders $36$, $72$, and $216$. If $H$ is a Hessian group or $A_6$, then the derived subgroup of $\pi^{-1}(H)$ contains $\xi_3I_3$. Hence $\xi_3I_3\in B$, which is impossible since $(\xi_3I_3)(F_U)=\xi_3F_U$. If $H\cong A_5$, then $B'\cong A_5$. Let $g\in B'$ have order $5$. As $g$ is conjugate to $g^{-1}$ in $A_5$, its eigenvalues are $1$ and $\xi_5^{\pm k}$ for some $k\not\equiv0\pmod5$. Replacing $g$ by a power and changing coordinates on $U$, we may assume that $g=\Diag(1,\xi_5,\xi_5^{-1})$. Then the $g$-invariant cubic and quartic forms lie in $\C x_1^3+\C x_1x_2x_3$ and $\C x_1^4+\C x_1^2x_2x_3+\C x_2^2x_3^2$, respectively. Moreover, $A_5$ has no nontrivial representation of dimension at most two. Hence $B'$ fixes $\mathcal C$ pointwise, so $\mathcal C\subset\C x_1^3+\C x_1x_2x_3$. Since $F_U$ is $g$-invariant as well, $F\in(x_1,x_3,x_4,x_5)^2$, contrary to Lemma~\ref{lem:quartic-sm}(1). Therefore $H\cong\PSL(2,7)$.

{\it Case $r=2$.} By \cite{Kle84}, $H$ is isomorphic to $A_4$, $S_4$, or $A_5$. Suppose that $H\cong S_4$ or $A_5$. By \cite[Lemma~6.4]{YYZ25}, $F_U=0$. Choose coordinates $x_1,x_2$ on $U$ diagonalizing some $g\in B$ whose image in $H$ has order $4$ if $H\cong S_4$ and order $5$ if $H\cong A_5$. Then $g=\Diag(\alpha,\beta)$ with $\alpha/\beta$ of order $4$ or $5$, so $g$ multiplies the cubic monomials $x_1^3$, $x_1^2x_2$, $x_1x_2^2$, $x_2^3$ by the distinct scalars $\alpha^3$, $\alpha^2\beta$, $\alpha\beta^2$, $\beta^3$. Thus $\mathcal C$ is spanned by some of them. As $F_U=0$, Lemma~\ref{lem:quartic-sm}(1) gives $x_1^3,x_2^3\in\mathcal C$. Since $B$ is primitive, some element of $B$ maps $x_1$ or $x_2$ to $ax_1+bx_2$ with $ab\neq0$, so $(ax_1+bx_2)^3\in\mathcal C$. All four coefficients of this cubic are nonzero, so $\mathcal C$ contains all four cubic monomials, contrary to $\dim\mathcal C\leq3$. Thus $H\cong A_4$.
\end{proof}

\subsection{Proof of Theorem~\ref{thm:nonmonomial}}
We first consider binary quartics and the case $(2,2,1)$, and then recall the facts we need on plane quartic curves and quartic surfaces. We start with the classical normal form of smooth binary quartic forms and their projective automorphism groups.
\label{sec:proof-nonmonomial}
\begin{lemma}\label{lem:binary-forms}
		Every smooth binary quartic form $\phi$ is isomorphic to $x_1^4+ax_1^2x_2^2+x_2^4$ for some $a\in\C$ with $a^2\neq4$. The group $G_\phi\subset\PGL(2,\C)$ is isomorphic to $A_4$ if $a^2=-12$, to $D_8$ if $a\in\{0,\pm6\}$, and to $C_2^2$ otherwise. In particular, $\phi$ is isomorphic to $x_1^3x_2+x_2^4$ if and only if $G_\phi\cong A_4$, and otherwise $\Aut(\phi)$ is conjugate to a subgroup of $\Aut(x_1^4+x_2^4)$.
	\end{lemma}
	\begin{proof}
	 Since $\phi$ is smooth, its four roots in $\P^1$ are distinct. Four distinct points of $\P^1$ are determined up to $\PGL(2,\C)$ by their cross-ratio, and the cross-ratio $(t-t^{-1})^2/(t+t^{-1})^2$ of $t,-t,t^{-1},-t^{-1}$ takes every value other than $0$ and $1$. So after a linear change of coordinates the roots are $\pm t,\pm t^{-1}$ for some $t\in\C^*$ with $t^4\neq1$. Then $\phi$ is a multiple of $(x_1^2-t^2x_2^2)(x_1^2-t^{-2}x_2^2)=x_1^4+ax_1^2x_2^2+x_2^4$ with $a=-(t^2+t^{-2})$, so $a^2\neq4$. Thus, we may assume $\phi=x_1^4+ax_1^2x_2^2+x_2^4$. The group $G_\phi$ acts faithfully on the roots and contains $\langle[\Diag(1,-1)],[P_{(1\,2)}]\rangle\cong C_2^2$, which permutes them simply transitively. The stabilizer of a root consists of the permutations of the other three roots that preserve the cross-ratio $\lambda=(a+2)/(a-2)$ of the four roots. A transposition of the other three roots changes $\lambda$ into $1/\lambda$, $1-\lambda$, or $\lambda/(\lambda-1)$, and a $3$-cycle changes it into $1/(1-\lambda)$ or $(\lambda-1)/\lambda$. Hence the stabilizer is $C_3$ if $\lambda^2-\lambda+1=0$, that is, if $a^2=-12$; it is $C_2$ if $\lambda\in\{-1,2,\frac12\}$, that is, if $a\in\{0,\pm6\}$; and it is trivial otherwise. Accordingly, $G_\phi\subset S_4$ has order $12$, $8$, or $4$, so $G_\phi\cong A_4$, $D_8$, or $C_2^2$. Since $G_{x_1^3x_2+x_2^4}\cong A_4$ by \eqref{eq:A4-generators} and replacing $x_2$ by $ix_2$ changes $a$ into $-a$, the form $\phi$ is isomorphic to $x_1^3x_2+x_2^4$ if and only if $a^2=-12$. If $a\in\{0,\pm6\}$, then $\phi$ is isomorphic to $x_1^4+x_2^4$, since $(x_1+x_2)^4+(x_1-x_2)^4=2(x_1^4+6x_1^2x_2^2+x_2^4)$. For the remaining values of $a$, $G_\phi\cong C_2^2$ and we have $\Aut(\phi)=\langle\xi_4I_2,\Diag(1,-1),P_{(1\,2)}\rangle\subset\Aut(x_1^4+x_2^4)$.
	\end{proof}

	In the case $l(\Aut(X))=(2,2,1)$, we use the binary quartic forms
	$$
	\Phi_\pm=x_1^4\pm2\sqrt3\,i\,x_1^2x_2^2+x_2^4,
	$$
	also written $\Phi_\varepsilon$ with $\varepsilon=\pm1$, and the group $\mathcal T\subset\SL(2,\C)$ isomorphic to $\SL(2,3)$ generated by
	$$
	\Diag(i,-i),\qquad\begin{pmatrix}0&1\\-1&0\end{pmatrix},\qquad g=\frac{i-1}2\begin{pmatrix}1&-i\\1&i\end{pmatrix}.
	$$
	The first $2$ matrices generate a subgroup $\mathcal{Q}\cong Q_8$ and fix $\Phi_\pm$, and $g$ maps $\Phi_\pm$ to $\xi_3^{\mp1}\Phi_\pm$. Hence, by Lemma~\ref{lem:binary-forms}, $\mathcal T$ is the preimage of $G_{\Phi_+}=G_{\Phi_-}\cong A_4$ in $\SL(2,\C)$, and $\mathcal Q$ is the preimage of the normal subgroup $C_2^2$ of $A_4$. Note that $\mathcal Q$ fixes no nonzero binary quadratic form.

	\begin{lemma}\label{lem:13-invs}
		The quartic forms fixed by all $\Diag(h,h,1)$ with $h\in\mathcal{Q}$ have the basis $x_5^4$, $\theta x_5^2$, $\theta^2$, and $\Phi_{\varepsilon,k}$ with $\varepsilon=\pm1$ and $0\leq k\leq4$, where $\theta=x_1x_4-x_2x_3$, $\Phi_{\varepsilon,0}=\Phi_\varepsilon(x_1,x_2)$, $\Phi_{\varepsilon,4}=\Phi_\varepsilon(x_3,x_4)$,
		\begin{align*}
			\Phi_{\varepsilon,1}&=x_1^3x_3+x_2^3x_4+\varepsilon\sqrt3\,i\,x_1x_2(x_1x_4+x_2x_3),\\
			\Phi_{\varepsilon,2}&=x_1^2x_3^2+x_2^2x_4^2+\varepsilon\tfrac{i}{\sqrt3}(x_1^2x_4^2+4x_1x_2x_3x_4+x_2^2x_3^2),\\
			\Phi_{\varepsilon,3}&=x_1x_3^3+x_2x_4^3+\varepsilon\sqrt3\,i\,x_3x_4(x_1x_4+x_2x_3).
		\end{align*}
		Moreover, $\Diag(g,g,1)$ fixes $\theta$ and maps $\Phi_{\varepsilon,k}$ to $\xi_3^{-\varepsilon}\Phi_{\varepsilon,k}$.
	\end{lemma}
	\begin{proof}
		The function \texttt{invariants\_of\_degree} in Sage~\cite{Sage} shows that these invariant forms span a space of dimension $13$. One checks directly that the listed forms are linearly independent and transform as stated.
	\end{proof}

\begin{lemma}\label{lem:221}
	Let $X$ be a smooth quartic threefold with $l(\Aut(X))=(2,2,1)$. Then $\Aut(X)$ is conjugate to a subgroup of $G_6$, $G_{14}$, or $G_{24}$.
\end{lemma}
\begin{proof}
	Write $X=X_F$ and $G=\Aut(X)$, and let $F_{d_1,d_2}$ be the part of $F$ of degree $d_1$ in $x_1,x_2$ and $d_2$ in $x_3,x_4$. An element of $\widetilde G$ preserving one of the two planes also preserves the other, so the principal subgroup $P$ is the stabilizer of each plane. Hence both projections of $J$ are binary primitive constituents, which are isomorphic to $A_4$ by Proposition~\ref{prop:possi-primitive-cons}. We may assume that they are $G_{\Phi_+}$. Every element of $P$ has the form $\Diag(\lambda t_1,\mu t_2,c)$ with $t_1,t_2\in\mathcal T$ and $\lambda,\mu,c\in\C^*$, so $P'\subset\{\Diag(h_1,h_2,1)\mid h_1,h_2\in\mathcal Q\}$, as $\mathcal T'=\mathcal Q$. Since the normal subgroups of $A_4$ are $1$, $C_2^2$, and $A_4$, Goursat's lemma shows that either $J$ is the graph of an automorphism of $A_4$, or $J$ contains $C_2^2\times C_2^2$. If $J$ is not a graph, preimages in $P$ of two distinct nontrivial elements of $C_2^2\times1$ (resp.\ $1\times C_2^2$) have commutator $\Diag(-I_2,I_2,1)$ (resp.\ $\Diag(I_2,-I_2,1)$), and $P'$ maps onto $J'=C_2^2\times C_2^2$, so $P'=\{\Diag(h_1,h_2,1)\mid h_1,h_2\in\mathcal Q\}$. If $J$ is a graph, then, since every automorphism of $A_4$ is induced by conjugation by its normalizer $S_4$ in $\PGL(2,\C)$, after changing coordinates on the plane of $x_3,x_4$, $J$ is the diagonal of $A_4\times A_4$, and similarly $P'=\{\Diag(h,h,1)\mid h\in\mathcal Q\}$. In both cases $\Diag(-I_4,1)\in P'$, so Lemma~\ref{lem:quartic-sm}(1) gives $x_5^4\in F$. We normalize its coefficient to $1$ and take the $F$-lifting $\widetilde G$ fixing $x_5$.
	
	We will use the following fact repeatedly. Let $\ell_1,\dots,\ell_4$ be linearly independent linear forms in $x_1,\dots,x_4$. If $\widetilde G$ permutes the spans of $\ell_1,\ell_2$ and of $\ell_3,\ell_4$, and preserves $x_5^4+\phi(\ell_1,\ell_2)+\psi(\ell_3,\ell_4)$ for some binary forms $\phi$ and $\psi$ isomorphic to $T_2$, then $G$ is conjugate to a subgroup of $G_6$. If $J$ is not a graph, then $F$ is invariant under $\Diag(h_1,h_2,1)$ for all $h_1,h_2\in\mathcal Q$. As $\mathcal Q$ fixes no nonzero binary form of degree $1$, $2$, or $3$, we get $F=x_5^4+\phi(x_1,x_2)+\psi(x_3,x_4)$. Here $\phi$ and $\psi$ are smooth with $G_\phi,G_\psi\supset A_4$, so they are isomorphic to $T_2$ by Lemma~\ref{lem:binary-forms}, and we obtain $G_6$.
	
	From now on, $J$ is a graph, so $P\subset\{\Diag(\lambda h,\mu h,1)\mid h\in\mathcal T,\ \lambda,\mu\in\C^*\}$, and $F$ is a linear combination of the forms in Lemma~\ref{lem:13-invs}. For $\varepsilon=\pm1$, no linear combination of $x_5^4$, $\Phi_{\varepsilon,4}$, $\Phi_{\varepsilon,1}$, and $\Phi_{-\varepsilon,2}$ is smooth. Indeed, let $L$ be a linear factor of $\Phi_{\varepsilon,4}$. A direct computation shows that $\Phi_{\varepsilon,1}$ and $\Phi_{-\varepsilon,2}$ are divisible by $\ell_0^2$ modulo $L$ for some linear form $\ell_0$ in $x_1,x_2$. Hence the combination lies in $(L)+(\ell_0,x_5)^2$, contrary to Lemma~\ref{lem:quartic-sm}(3).
	
	We identify each element $\Diag(\lambda I_2,\mu I_2,1)$ of $N$ with the pair $(\lambda,\mu)$; it multiplies $F_{d_1,d_2}$ by $\lambda^{d_1}\mu^{d_2}$. By Lemma~\ref{lem:quartic-sm}(1), $F_{4,0}$ or $F_{3,1}$ is nonzero, and so is $F_{0,4}$ or $F_{1,3}$. Hence, interchanging the blocks if necessary, we may assume that $N$ lies in one of the groups
	$$
	\{\lambda^4=\mu^4=1\}\cong C_4^2,\qquad\{\lambda^3\mu=\mu^4=1\}\cong C_{12},\qquad\{\lambda^3\mu=\lambda\mu^3=1\}\cong C_8.
	$$
	
	If $3\mid|N|$, then $\Diag(\xi_3I_2,I_2,1)\in N$, and Lemma~\ref{lem:13-invs}, applied to this element and to some $\Diag(\lambda g,\mu g,1)\in P$, gives $F=x_5^4+\alpha\Phi_{\varepsilon,4}+\beta\Phi_{\varepsilon',1}$, where $\alpha\beta\neq0$ and $\varepsilon'=-\varepsilon$ by smoothness and the above. Replacing $x_2,x_4$ by $ix_2,ix_4$ if necessary, which interchanges $\Phi_{1,k}$ and $\Phi_{-1,k}$ and normalizes $\mathcal T$, and rescaling, we get $F=x_5^4+\Phi_{1,4}+\Phi_{-1,1}$. As $F_{4,0}=0\neq F_{0,4}$, no element of $\widetilde G$ interchanges the blocks, so $|G|=12|N|\leq144$. For $G_{24}$, the group $J$ is also a graph and $3\mid|N|$ (Example~\ref{ex:221-blocks}), so the same argument transforms $F_{24}$ into this form. Since $|G_{24}|=144$, $G$ is conjugate to $G_{24}$.
	
	Now let $3\nmid|N|$, and let $p\in P$ be of $3$-power order with image $([g],[g])$ in $J$. As $p^3\in N$, we have $p=\Diag(\xi_3^ag,\xi_3^bg,1)$ with $a,b\in\{0,\pm1\}$, and $P=\langle P',N,p\rangle$. By Lemma~\ref{lem:13-invs}, $p$ multiplies $\Phi_{\varepsilon,k}$ by $\xi_3^{a+(b-a)k-\varepsilon}$ and $\theta$ by $\xi_3^{a+b}$. Hence $a,b\in\{\pm1\}$. Indeed, if $a=b=0$, then $F$ is a linear combination of $x_5^4$, $\theta x_5^2$, and $\theta^2$, contrary to Lemma~\ref{lem:quartic-sm}(1). If exactly one of $a$ and $b$ is $0$, say $a=0$ after interchanging the blocks, then $F$ is a linear combination of $x_5^4$, $\Phi_{b,4}$, $\Phi_{b,1}$, and $\Phi_{-b,2}$, which is singular.
	
	If $\lambda^4=\mu^4=1$ for all $(\lambda,\mu)\in N$, then $\Phi_{a,0}$ and $\Phi_{b,4}$ span the $P$-invariant quartic forms in $x_1,x_2$ and in $x_3,x_4$, respectively. If some $t\in\widetilde G$ interchanges the blocks, then $t$ normalizes $P$, so $t(\Phi_{a,0})=\nu\Phi_{b,4}$ for some $\nu\in\C^*$, and $t(\Phi_{b,4})=\nu^{-1}\Phi_{a,0}$ since $t^2\in P$; otherwise, let $\nu=1$. Hence $\widetilde G$ preserves $x_5^4+\Phi_{a,0}+\nu\Phi_{b,4}$, giving $G_6$.
	
	Otherwise, $N$ has an element of order not dividing $4$. Since $3\nmid|N|$ and the elements of order dividing $4$ in the three groups above lie in the first one, $N$ is the third group, generated by $\Diag(\xi_8I_2,\xi_8^5I_2,1)$. The invariants of $N$ and $p$, smoothness, and rescaling then give $F=x_5^4+\Phi_{b,1}+\Phi_{a,3}$. If $a=b$, then $F=x_5^4+(\Phi_a(x_1+x_3,x_2+x_4)-\Phi_a(x_1-x_3,x_2-x_4))/8$, the elements of $P$ have the form $\Diag(\lambda h,\pm\lambda h,1)$, and $P_{(1\,3)(2\,4)}$ fixes $F$. Since $[\widetilde G:P]\leq2$, $\widetilde G$ permutes the spans of $x_1+x_3,x_2+x_4$ and of $x_1-x_3,x_2-x_4$, giving $G_6$. If $a=-b$, then, interchanging the blocks if necessary, $F=x_5^4+\Phi_{-1,1}+\Phi_{1,3}=F_{14}$, and $G=G_{14}$, as $|G|\leq2|N||J|=192=|G_{14}|$.
\end{proof}

The automorphism groups of smooth plane quartic curves, together with their normal forms, are classified in \cite[Theorem~6.5.2 and Table~6.1]{Dol12}.
We need the following consequence for the linear automorphism groups.

\begin{lemma}\label{lem:3-var-4-forms}
	Let $F$ be a smooth quartic form in three variables. Then $\Aut(F)$ is conjugate to a subgroup of $\Aut(F')$ for some $F'\in\{3K_1,T_2+K_1,T_3,K_3\}$. In particular, let $\Aut(T_2)\times\Aut(F)$ act block-diagonally on $\C^5$. Then its image in $\PGL(5,\C)$ is conjugate to a subgroup of $G_5$, $G_6$, $G_{10}$, or $G_2$ if $F'=3K_1$, $T_2+K_1$, $T_3$, or $K_3$, respectively.
\end{lemma}
\begin{proof}
	By \cite[Theorem~6.5.2 and Table~6.1]{Dol12}, the cases $|G_F|=96,48,9,168$ give the forms in the lemma, respectively.
	In the remaining cases, except when $G_F\cong C_3$, $C_6$, or $S_3$, the normal forms in \cite[Table~6.1]{Dol12} show that $\Aut(F)$ is generated by monomial matrices with entries in $\langle i\rangle$, and hence preserves $3K_1$. For $C_3$ and $C_6$, the generators $\Diag(1,1,\xi_3)$ and $\Diag(-1,1,\xi_3)$ preserve $x_1^4+x_2^4+x_2x_3^3$, which is isomorphic to $T_2+K_1$. For $S_3$, the generators $P_{(1\,2)}$ and $\Diag(\xi_3,\xi_3^2,1)$ give the permutation representation, so their group is conjugate to a group preserving $3K_1$. 
	The second assertion then follows from Example~\ref{ex:17maximal}.
\end{proof}
Recall that $G_j=\{[\Diag(A,1)]\mid A\in\Aut(h_j)\}$ for $j=3,7,8,9$ by Example~\ref{ex:17maximal}.

\begin{lemma}\label{lem:primitive-surfaces}
	Let $f\in\C[x_1,\dots,x_4]$ be a smooth quartic form such that $G_f$ is primitive. If $G_f$ is conjugate to the standard $S_5$ (resp.\ a spin $A_5$), then $\Aut(f)$ is conjugate to a subgroup of $\Aut(h_3)$ (resp.\ $\Aut(h_8)$). Otherwise, $\Aut(f)$ is conjugate to $\Aut(h_j)$ for some $j\in\{3,7,8,9\}$.
\end{lemma}
\begin{proof}
	If $G_f$ is conjugate neither to the standard $S_5$ nor to a spin $A_5$, then $f$ is isomorphic to $h_j$ for some $j\in\{3,7,8,9\}$ by the quartic semi-invariant computations in \cite[Section~7 and Theorems~8.1--8.3]{AOT24}, so $\Aut(f)$ is conjugate to $\Aut(h_j)$.
	
	Otherwise, we may assume that $G_f$ is the standard $S_5$ (resp.\ a spin $A_5$). Then the same computations show that $\Aut(f)$ is generated by $iI_4$ and the standard $S_5$ (resp.\ the image of the faithful irreducible four-dimensional representation of $\SL(2,5)$), and the generators in Example~\ref{ex:17maximal} show that this group is conjugate to a subgroup of $\Aut(h_3)$ (resp.\ $\Aut(h_8)$).
\end{proof}

\begin{proof}[Proof of Theorem~\ref{thm:nonmonomial}]
	Let $G=\Aut(X)$, where $X=V(F)$, and use the decomposition \eqref{eq:blocks}. Lemma~\ref{lem:primitive-5} excludes $(5)$, and Lemma~\ref{lem:221} treats $(2,2,1)$. In the remaining cases, a unique block, say $U=W_1$, has maximal dimension, so $U$ and $W=W_2\oplus\cdots\oplus W_s$ are preserved by $\Aut(F)=\widetilde G\times\langle\xi_4I_5\rangle$.
	
	\textit{Case $l(G)=(4,1)$.} By the proof of Proposition~\ref{prop:possi-primitive-cons}, we may write $F=x_5^4+x_5^2Q_2+x_5Q_3+f$, where $f=F_U\neq0$, and we take the $F$-lifting $\widetilde G$ fixing $x_5$. If $H_1$ is the standard $S_5$, the same proof shows that $\Diag(\sigma,1)\in\Aut(F)$ for every standard $\sigma\in S_5$, so every element of $\widetilde G$ is $\Diag(c\sigma,1)$ with such $\sigma$ and $c^4=1$. By the proof of Lemma~\ref{lem:primitive-surfaces}, the group generated by $iI_4$ and the standard $S_5$ is conjugate to a subgroup of $\Aut(h_3)$, so $G$ is conjugate to a subgroup of $G_3$. Otherwise, $Q_2=Q_3=0$ and $f$ is smooth with $G_f=H_1$, so $G=\{[\Diag(A,1)]\mid A\in\Aut(f)\}$ is conjugate to a subgroup of $G_3$, $G_7$, $G_8$, or $G_9$ by Lemma~\ref{lem:primitive-surfaces}.
	
	\textit{Cases $l(G)=(3,2)$ and $(3,1,1)$.} By Proposition~\ref{prop:possi-primitive-cons}, after conjugation $H_1=G_{K_3}$. Then $\Gamma=\Aut(K_3)'\cong\PSL(2,7)$ is the derived subgroup of the image of $\Aut(F)$ in $\GL(U)$. The image of $\Aut(F)$ in $\GL(W)$ is solvable, since its image in $\PGL(2,\C)$ is monomial or $H_2\cong A_4$. As $\Gamma$ is perfect, some term of the derived series of $\Aut(F)$ is $\{\Diag(\gamma,I_2)\mid\gamma\in\Gamma\}$. Since $\Gamma$ has no invariant forms of degrees $1$, $2$, $3$, and its invariant quartics are the multiples of $K_3$~\cite{Kle79}, we may assume $F=K_3+\phi$, where $\phi=F|_W$ is smooth. Then $\Aut(F)=\Aut(K_3)\times\Aut(\phi)$, so by Lemma~\ref{lem:binary-forms}, $G$ is conjugate to $G_2$ if $\phi$ is isomorphic to $T_2$, that is, if $l(G)=(3,2)$, and to a subgroup of $G_4$ otherwise.
	
	\textit{Case $l(G)=(2,1,1,1)$.} By Proposition~\ref{prop:possi-primitive-cons}, $H_1\cong A_4$. Since $H_1/J_1$ is a quotient of $\Pi\subset S_3$, the normal subgroup $C_2^2$ of $H_1$ lies in $J_1$. Take liftings in the principal subgroup $P$ of two distinct nontrivial elements of $C_2^2$. Their commutator is $-I_2$ on $U$ and trivial on $W$, since $P$ acts on $W$ by diagonal matrices, so it is $\Diag(-I_2,I_3)$. As $\Aut(F)$ preserves $U$ and $W$, it commutes with $\Diag(-I_2,I_3)$, so by Corollary~\ref{cor:involution-split}, $F|_U$ and $F|_W$ are smooth and $\Aut(F)\subset\Aut(F|_U)\times\Aut(F|_W)$. As $G_{F|_U}\supset H_1\cong A_4$, we may assume $F|_U=T_2$ by Lemma~\ref{lem:binary-forms}, so $G$ is conjugate to a subgroup of $G_2$, $G_5$, $G_6$, or $G_{10}$ by Lemma~\ref{lem:3-var-4-forms}.
\end{proof}

\section{The monomial case}\label{sec:monomial}
In this section, we prove the following theorem.

\begin{theorem}\label{thm:monomial}
	Let $X$ be a smooth quartic threefold, and let $H\subset\Aut(X)$ be a monomial subgroup. Then $H$ is conjugate to a subgroup of one of
	$$
	G_1,\ G_2,\ G_4,\ G_5,\ G_6,\ G_{10},\ G_{11},\ \dots,\ G_{24}.
	$$
\end{theorem}

Throughout this section, unless otherwise stated, $X=X_F$ is a smooth quartic threefold and $H\subset\Aut(X)$ is a monomial subgroup consisting, after a linear change of coordinates, of classes of monomial matrices. We fix an $F$-lifting $\widetilde H$ of $H$, which consists of monomial matrices. Let $P$ and $\Pi$ be as in Section~\ref{sec:primitive-constituents}, defined with respect to the decomposition $\C^5=\C e_1\oplus\cdots\oplus\C e_5$ in place of \eqref{eq:blocks}, where we identify $S_5$ with the group of permutations of the lines $\C e_1,\dots,\C e_5$. Then $P$ is the subgroup of diagonal matrices in $\widetilde H$. We set $D:=\pi(P)$. Since $\pi$ is injective on $\widetilde H$, the exact sequence \eqref{eq:first-sequence} becomes
\begin{equation*}
	1\longrightarrow D\longrightarrow H\longrightarrow\Pi\longrightarrow1.
\end{equation*}

If a group of diagonal matrices acts on $x_1,\dots,x_5$ by characters $\chi_1,\dots,\chi_5$, then $x_1^{m_1}\cdots x_5^{m_5}$ is invariant if and only if $\chi_1^{m_1}\cdots\chi_5^{m_5}=1$. In particular, if such a group fixes $F$ and $x_i^3x_j\in F$, then $\chi_j=\chi_i^{-3}$. Since $\widetilde H$ consists of monomial matrices, it permutes the quartic monomials up to scalars. Hence, for every set of quartic monomials permuted by $\widetilde H$ up to scalars, the sum of the terms of $F$ with these monomials is $\widetilde H$-invariant. If this sum is smooth, then $H$ preserves it. For example, if $F=5K_1$ (resp.\ $K_4+K_1$, resp.\ $K_5$), then the classes of monomial matrices in $G_F$ form $G_1$ (resp.\ $G_{11}$, resp.\ $G_{12}$). Indeed, for this group, $D\cong C_4^4$ (resp.\ $C_{80}$, resp.\ $C_{61}$) by the proof of Theorem~\ref{thm:maximal-abelian}, and $\Pi=S_5$ (resp.\ $\langle(1\,2\,3\,4)\rangle$, resp.\ $\langle(1\,2\,3\,4\,5)\rangle$).

The proof of Theorem~\ref{thm:monomial}, at the end of this section, first considers the cases where $|D|$ is divisible by $5$, $7$, or $61$, or where $\Pi$ is transitive. In the remaining case, $|H|=2^a3^b$. If $H$ has a nontrivial normal $2$-subgroup, the eigenspaces of a suitable subgroup of it split $F$ into smooth forms in fewer variables (Lemmas~\ref{lem:monomial-smooth-sum} and~\ref{lem:monomial-small-blocks} and Proposition~\ref{prop:monomial-low-residue}); otherwise, Lemma~\ref{lem:monomial-o2-free} applies.

\subsection{Normal subgroups and smooth sums}
We first give a general criterion for $p$-subgroups to be abelian.
\begin{proposition}\label{prop:p-subgroups-abelian}
	Let $F\in\C[x_1,\dots,x_r]$ be a smooth form of degree $d\geq2$, and let $p$ be a prime such that $p\nmid d$ and $r<2p$. Then every finite $p$-subgroup of $G_F$ is abelian.
\end{proposition}
\begin{proof}
	Let $G\subset G_F$ be a finite $p$-subgroup. Since the scalar matrices in $\Aut(F)$ form a group of order $d$, which is prime to $p$, a Sylow $p$-subgroup $\widetilde G$ of $\pi^{-1}(G)\cap\Aut(F)$ maps isomorphically onto $G$, that is, $\widetilde G$ is an $F$-lifting of $G$. Suppose that $G$ is nonabelian. Since the irreducible representations of a $p$-group have dimension a power of $p$ and $r<2p\leq p^2$, after a linear change of coordinates we have a $\widetilde G$-stable decomposition $\C^r=U\oplus\C e_{p+1}\oplus\cdots\oplus\C e_r$ with $U=\langle e_1,\dots,e_p\rangle$ irreducible. The image $\overline{G}$ of $\widetilde{G}$ in $\GL(U)$ is nonabelian, so its derived subgroup $\overline{G}'$ is a nontrivial normal subgroup of the $p$-group $\overline G$ and hence meets the center of $\overline{G}$ nontrivially. By Schur's lemma, this center consists of scalar matrices. Since the elements of $\overline{G}'$ have determinant $1$, we get $\xi_pI_p\in\overline{G}'$. As $\widetilde{G}$ acts on each $\C e_j$ with $j>p$ by a character, which is trivial on its derived subgroup, we obtain $\Diag(\xi_pI_p,I_{r-p})\in\widetilde G$. This element multiplies every monomial of degree $d$ in $x_1,\dots,x_p$ by $\xi_p^d\neq1$, so $F|_U=0$. Hence, on $U$, the partial derivatives of $F$ with respect to $x_1,\dots,x_p$ vanish, and those with respect to $x_{p+1},\dots,x_r$ are $r-p<p$ forms of degree $d-1$ in $p$ variables. These forms have a common nonzero zero, which is a singular point of $F$, a contradiction.
\end{proof}
Taking $d=4$, $r=5$, and $p=3$, we see that every $3$-subgroup of $\Aut(X)$ is abelian.

\begin{remark}
	The bound $r<2p$ is sharp. The smooth form $\sum_{j=1}^p(x_j^py_j+y_j^{p+1})$ of degree $p+1$ in $2p$ variables is preserved by the multiplying each $x_j$ by $\xi_p$ and by the cyclic permutation of the pairs $(x_j,y_j)$. These generate a nonabelian group $C_p^p\rtimes C_p$ containing no nontrivial scalar matrix, so its image in $G_F$ is a nonabelian $p$-group.
\end{remark}

\begin{lemma}\label{lem:monomial-smooth-sum}
	Let $F\in\C[x_1,\dots,x_r]$ be a smooth form of degree $d\geq2$, and let $H\subset\Aut(F)$ be a finite abelian subgroup generated by matrices whose orders divide $d$. Let $\chi_1,\dots,\chi_k$ be the distinct characters of $H$ on $\C^r$, so that
	$$
	\C^r=V_1\oplus\cdots\oplus V_k,\qquad V_j=\{v\in\C^r\mid hv=\chi_j(h)v\text{ for all }h\in H\}.
	$$
	Let $F^{(j)}$ be the sum of the terms of $F$ involving only coordinates on $V_j$. Then each $F^{(j)}$ is smooth, and the normalizer of $H$ in $\Aut(F)$ preserves $F^{(1)}+\cdots+F^{(k)}$.
\end{lemma}
\begin{proof}
	Let $j\neq l$. Every monomial of degree $d-1$ in the coordinates on $V_j$ and degree $1$ in those on $V_l$ spans an $H$-stable line on which $H$ acts by the character $\chi_j^{d-1}\chi_l=\chi_j^{-1}\chi_l$, using $\chi_j^d=1$. This character is nontrivial, so no such monomial occurs in $F$. Thus, on $V_j$, the partial derivatives of $F$ with respect to the coordinates on $V_l$ vanish, while those with respect to the coordinates on $V_j$ are the partial derivatives of $F^{(j)}$. Hence a nonzero common zero of the partial derivatives of $F^{(j)}$ would be a singular point of $F$, so $F^{(j)}$ is smooth, and so is $F^{(1)}+\cdots+F^{(k)}$. Finally, every element of the normalizer of $H$ in $\Aut(F)$ permutes $V_1,\dots,V_k$ and fixes $F$, so it permutes $F^{(1)},\dots,F^{(k)}$ and preserves their sum.
\end{proof}
If $H=\langle\tau\rangle$ for an involution $\tau$, Lemma~\ref{lem:monomial-smooth-sum} reduces to Corollary~\ref{cor:involution-split}. Below, we also apply the lemma to elementary abelian $2$-groups of larger rank.

\begin{lemma}\label{lem:monomial-small-blocks}
	Let $F\in\C[x_1,\dots,x_5]$ be a smooth quartic form that is a sum of forms in disjoint sets of variables, each of size at most three. Then the subgroup of $G_F$ permuting the corresponding coordinate subspaces is conjugate to a subgroup of $G_1$, $G_2$, $G_4$, $G_5$, $G_6$, $G_{10}$, or $G_{13}$.
\end{lemma}
\begin{proof}
	Let $H$ be this subgroup. Each of these forms is smooth. Combining the forms in one variable when necessary, we may assume that $F=5K_1$, which gives $G_1$, or $F=\phi(x_1,x_2)+\psi(x_3,x_4,x_5)$, or $F=\phi(x_1,x_2)+\psi(x_3,x_4)+x_5^4$.
	
	In the second case, $H=\pi(\Aut(\phi)\times\Aut(\psi))$. By Lemmas~\ref{lem:binary-forms} and~\ref{lem:3-var-4-forms}, it is conjugate to a subgroup of $\pi(\Aut(\phi_0)\times\Aut(\psi_0))$ for some $\phi_0\in\{2K_1,T_2\}$ and $\psi_0\in\{3K_1,T_2+K_1,T_3,K_3\}$. Comparing the $8$ forms $\phi_0+\psi_0$ with Example~\ref{ex:17maximal} gives the stated groups.
	
	In the third case, if an element of $H$ interchanges the two planes, then $\phi$ and $\psi$ are isomorphic, and we may assume $\phi=\psi$. Apply Lemma~\ref{lem:binary-forms} to $\phi$ and $\psi$, using the same change of coordinates when $\phi=\psi$. Then $H$ preserves the form obtained by replacing each of $\phi$ and $\psi$ by $T_2$ if it is isomorphic to $T_2$, and by $2K_1$ otherwise. This form is $5K_1$, $T_2+3K_1$, or $2T_2+K_1$, giving $G_1$, $G_5$, or $G_6$, respectively.
\end{proof}

\begin{lemma}\label{lem:monomial-o2-free}
	Let $H\subset\Aut(X)$ be a subgroup of order $2^a3^b$ with no nontrivial normal $2$-subgroup. Then $H$ is conjugate to a subgroup of one of the groups in Proposition~\ref{prop:abelian-3-groups}.
\end{lemma}
\begin{proof}
	Assume $H\neq1$, and let $S$ be the largest normal $3$-subgroup of $H$, that is, the intersection of its Sylow $3$-subgroups. By Burnside's theorem, $H$ is solvable, and as $H$ has no nontrivial normal $2$-subgroup, $S$ is the Fitting subgroup of $H$, so $C_H(S)\subset S$. By Proposition~\ref{prop:p-subgroups-abelian}, every $3$-subgroup of $H$ is abelian. Hence $C_H(S)=S$, and $S$ is a Sylow $3$-subgroup of $H$, since a Sylow $3$-subgroup containing $S$ centralizes $S$. Thus $H=S\rtimes T$, where $T$ is a $2$-group acting faithfully on $S$. If $T=1$, we apply Proposition~\ref{prop:abelian-3-groups}.
	
	Otherwise, by Proposition~\ref{prop:abelian-3-groups}, we may assume that $\widetilde S$ is generated by the matrices in one of the cases listed there. Since $\widetilde S$ is normal in $\widetilde H$, conjugation by each $B\in\widetilde H$ is an automorphism of $\widetilde S$ that permutes the characters of $\widetilde S$ on the five coordinate lines. This automorphism is nontrivial for $B\notin\widetilde S$, since $T$ acts faithfully on $S$. By direct computation, such a nontrivial automorphism exists only in the cases $(a,b)=(2,0)$ of~(1), $(a,b)=(8,1)$ of~(2), $a=1$ of~(6), and all cases of~(5), and it is unique in each of them. Hence $T\cong C_2$, and we write $T=\langle[A]\rangle$ with $A\in\widetilde H$, so that $A^2=I_5$.
	
	For $S\cong C_3$, we may assume that $\widetilde H$ is generated by $A$ and $s=\Diag(\xi_3,\xi_3^2,1,1,1)$. Since $AsA^{-1}=s^{-1}$, the matrix $A$ interchanges $\C e_1$ and $\C e_2$ and preserves $\langle e_3,e_4,e_5\rangle$. As $A^2=I_5$, after rescaling $x_2$ and changing the coordinates $x_3,x_4,x_5$, we have $A=P_{(1\,2)}\Diag(1,1,\varepsilon_3,\varepsilon_4,\varepsilon_5)$ with $\varepsilon_k=\pm1$. The permutation matrix $P_{(1\,2\,3)}\in\GL(3,\C)$ has eigenvalues $\xi_3,\xi_3^2,1$, and $P_{(2\,3)}$ interchanges its eigenlines for $\xi_3$ and $\xi_3^2$ and fixes that for $1$. Hence $\widetilde H$ is conjugate to the group generated by $\Diag(P_{(1\,2\,3)},1,1)$ and $\Diag(\varepsilon_3P_{(2\,3)},\varepsilon_4,\varepsilon_5)$, which consists of signed permutation matrices and therefore preserves $F_1$. Thus $H$ is conjugate to a subgroup of $G_1$.
	
	For $S\cong C_9$ or $C_9\times C_3$, after interchanging $x_2$ and $x_3$ in the first case, we may assume that $\widetilde H$ is generated by $A$, $g_1=\Diag(\xi_9,\xi_9^8,\xi_3^2,\xi_3,1)$, and, in the second case, $g_2=\Diag(1,\xi_3,1,1,1)$. The characters of $S$ on the five coordinate lines are distinct, and the only nontrivial automorphism of $S$ permuting them interchanges those on $\C e_1$ and $\C e_2$ and those on $\C e_3$ and $\C e_4$. Hence $A$ interchanges $\C e_1$ and $\C e_2$, interchanges $\C e_3$ and $\C e_4$, and preserves $\C e_5$. As $A^2=I_5$, after replacing $A$ by $-A$ if necessary and rescaling $x_2$ and $x_4$, we have $A=P_{(1\,2)(3\,4)}$. Since $\langle g_1,g_2\rangle=\langle g_1,\Diag(\xi_3,\xi_3,1,1,1)\rangle$, the group $H$ is contained in $G_{17}$ by its displayed generators.
	
	Let $S\cong C_3^2$. For $(a,b)=(0,2)$, the matrix $A$ preserves $\C e_1$ and $\langle e_4,e_5\rangle$ and interchanges $\C e_2$ and $\C e_3$; after rescaling $x_3$, it interchanges $x_2$ and $x_3$. By invariance and Lemma~\ref{lem:quartic-sm}(1), the terms of $F$ divisible by $x_1^3$ are $x_1^3\ell$ for a nonzero linear form $\ell$ in $x_4,x_5$, and $A(\ell)=\pm\ell$ with the same sign as $A(x_1)=\pm x_1$. Changing $x_4,x_5$, we may assume that $\ell=x_4$ and $A(x_5)=\pm x_5$. Then $\widetilde H$ preserves $x_1^3x_4+x_4^4$ and acts on $\langle e_2,e_3,e_5\rangle$ through $\Diag(\xi_3,\xi_3^2,1)$ and $A$, so, as in the case $S\cong C_3$, it is conjugate to a group preserving $x_1^3x_4+x_4^4+x_2^4+x_3^4+x_5^4$, which gives $G_5$.
	
	For $(a,b)=(0,0)$ or $(1,1)$, the matrix $A$ interchanges $\C e_1$ and $\C e_2$; after rescaling $x_2$, it interchanges $x_1$ and $x_2$. By invariance, the terms of $F$ divisible by $x_1^3$ or $x_2^3$ are $x_1^3\ell_1+x_2^3\ell_2$, where $\ell_1$ and $\ell_2$ are linear forms in the coordinates fixed by $\widetilde S$, namely $x_3,x_4,x_5$ for $(0,0)$ and $x_4,x_5$ for $(1,1)$, and $A$ interchanges $\ell_1$ and $\ell_2$. They are linearly independent, since otherwise $F$ is singular at a point of $\langle e_1,e_2\rangle$. For $(a,b)=(0,0)$, changing $x_3,x_4,x_5$, we may assume that $\ell_1=x_3$, $\ell_2=x_4$, and $A(x_5)=\pm x_5$. Then $\widetilde H$ permutes $\langle e_1,e_3\rangle$ and $\langle e_2,e_4\rangle$ and preserves $x_1^3x_3+x_3^4+x_2^3x_4+x_4^4+x_5^4$, which becomes $F_6$ after interchanging $x_2$ and $x_3$, giving $G_6$. For $(a,b)=(1,1)$, changing $x_4,x_5$, we may assume that $\ell_1=x_4$ and $\ell_2=x_5$. As $A$ preserves $\C e_3$, replacing $A$ by $-A$ and $x_2,x_5$ by $-x_2,-x_5$ if necessary, we may assume that $A(x_3)=-x_3$. In the coordinates $(x_4+x_5,x_4-x_5,x_3,x_1,x_2)$, the generators of $\widetilde H$ become those of $G_{19}$, after changing the generators of $\widetilde S$.
\end{proof}

\subsection{The $(4,1)$ case}
We now consider the case where $\widetilde H$ contains a central involution with eigenspaces of dimensions $4$ and $1$ (Proposition~\ref{prop:monomial-low-residue}). We first need two lemmas.
\begin{lemma}\label{lem:monomial-irreducible-four}
	Let $L\subset\GL(4,\C)$ be a finite irreducible monomial group preserving a smooth quartic form $h$. Then either \textup{(1)} $L$ is conjugate to a group of monomial matrices preserving $4K_1$ or $K_4$, or \textup{(2)} $L$ permutes two planes $U$ and $W$ with $\C^4=U\oplus W$, and the stabilizer of $U$ in $L$ acts primitively on $U$ and on $W$.
\end{lemma}
\begin{proof}
	Let $N$ be a maximal abelian normal subgroup of $L$, and for a character $\chi$ of $N$, let $V_\chi=\{v\in\C^4\mid nv=\chi(n)v\text{ for all }n\in N\}$. By Clifford's theorem, $L$ permutes the nonzero subspaces $V_\chi$ transitively, so they are four lines, two planes, or $\C^4$ itself.
	
	In the first case, choose coordinates $x_1,\dots,x_4$ such that these lines are the coordinate lines. Since the characters of $N$ on them are distinct, Lemma~\ref{lem:quartic-sm}(1), applied to $h+x_5^4$, gives for each $i$ a unique $j$ with $x_i^3x_j\in h$. The map $i\mapsto j$ commutes with the transitive action of $L$, so it is a permutation with cycles of equal length, and the corresponding terms of $h$ form an $L$-invariant smooth form isomorphic to $4K_1$, $2K_2$, or $K_4$. As $K_2$ is isomorphic to $2K_1$ and $\Aut(4K_1)$ consists of monomial matrices, \textup{(1)} follows.
	
	If the subspaces $V_\chi$ are two planes, their common stabilizer $L_0$ fixes the two characters, so $L_0=C_L(N)$, and $L_0$ acts irreducibly on both planes by Clifford's theorem. If $V_\chi=\C^4$, then $N$ is scalar, so $N$ is the diagonal subgroup of $L$ by maximality, and $L/N\subset S_4$. Since $\C^4$ is irreducible and $N$ is central, $[L:N]\geq16$, so $L/N\cong S_4$. The preimage $L_0$ of $A_4$ acts reducibly, as $[L_0:N]=12<16$, so by Clifford's theorem, $\C^4$ is the direct sum of two planes on which $L_0$ acts irreducibly and which are interchanged by $L$.
	
	In both cases, $L_0$ centralizes $N$ and is the stabilizer of each of the two planes. If $L_0$ acted imprimitively on one of them, then $L$ would permute four lines, two in each plane, preserving their partition into pairs. The subgroup of $L$ preserving each of these lines is abelian and normal and contains $N$, so it equals $N$, and $L/N$ embeds into $D_8$. Then the nontrivial normal subgroup $L_0/N$ of the $2$-group $L/N$ meets $Z(L/N)$ nontrivially, and for $g\in L_0$ whose class is a nontrivial central element, $\langle N,g\rangle$ is a larger abelian normal subgroup, a contradiction. Hence \textup{(2)} holds.
\end{proof}

The next lemma extends the arguments of Section~\ref{sec:proof-nonmonomial} for the subdegree sequences $(2,2,1)$ and $(2,1,1,1)$ to subgroups of $\Aut(X)$.

\begin{lemma}\label{lem:monomial-binary-blocks}
	Let $H\subset\Aut(X)$ be a subgroup such that $l(H)=(2,2,1)$ or $(2,1,1,1)$ for some decomposition \eqref{eq:blocks}. Then $H$ is conjugate to a subgroup of one of $G_2$, $G_5$, $G_6$, $G_{10}$, $G_{14}$, $G_{24}$.
\end{lemma}
\begin{proof}
	We may replace $H$ by the subgroup of $\Aut(X)$ permuting the blocks $W_a$, since the stabilizers of the blocks in this larger group still act primitively on them. The proof of Proposition~\ref{prop:possi-primitive-cons} for $r=2$ applies to any subgroup of $\Aut(X)$, so the binary primitive constituents of $H$ are isomorphic to $A_4$.
	
	For $(2,2,1)$, the proof of Lemma~\ref{lem:221} then applies to $H$, except where it uses that $G=\Aut(X)$ contains a conjugate of $G_{24}$ in the case $3\mid|N|$ and $G_{14}$ in the last case. These groups permute the blocks, so they lie in $H$, and the corresponding bounds $|H|\leq144$ and $|H|\leq192$ show that $H$ is conjugate to $G_{24}$ and $G_{14}$, respectively. In the other cases, $H$ is conjugate to a subgroup of $G_6$. For $(2,1,1,1)$, the argument for this case in the proof of Theorem~\ref{thm:nonmonomial}, with $\widetilde H$ in place of $\Aut(F)$, gives $G_2$, $G_5$, $G_6$, or $G_{10}$.
\end{proof}

\begin{proposition}\label{prop:monomial-low-residue}
	Let $H\subset\Aut(X)$ be a monomial subgroup of order $2^a3^b$ whose center contains an element conjugate to $[\Diag(I_4,-1)]$. Then $H$ is conjugate to a subgroup of one of the groups in Theorem~\ref{thm:monomial}.
\end{proposition}
\begin{proof}
	The preimage $\tau\in\widetilde H$ of the central element is a monomial central involution. Replacing two basis vectors $e_j,e_k$ by $e_j\pm\lambda e_k$ if necessary, we may assume that $\tau=\pm\Diag(I_4,-1)$ and $\widetilde H$ is still monomial. By Corollary~\ref{cor:involution-split}, we may assume that $F=h+x_5^4$. After adjoining $[\Diag(iI_4,1)]$ to $H$, we may also assume that $H=\{[\Diag(B,1)]\mid B\in L\}\cong L$ for a monomial group $L\subset\Aut(h)$ containing $iI_4$. If $L$ is irreducible, the assertion follows from Lemmas~\ref{lem:monomial-irreducible-four} and~\ref{lem:monomial-binary-blocks}, so we assume that $L$ is reducible.
	
	Let $N=N_2\times S$ be a maximal abelian normal subgroup of $L$, with $N_2$ a $2$-group and $S$ a $3$-group. Then $iI_4\in N$. If $N_2$ has an involution other than $-I_4$, Lemmas~\ref{lem:monomial-smooth-sum} and~\ref{lem:monomial-small-blocks} apply to the group of matrices $\Diag(n,1)$ with $n\in N_2$ and $n^2=I_4$. Indeed, each character of this group has multiplicity at most $3$, since $\Diag(-I_4,1)$ separates $\C e_5$ from $\C^4$, and such an involution has both eigenvalues $1$ and $-1$ on $\C^4$. Otherwise, $N_2$ is cyclic of order $4$, $8$, or $16$ by Theorem~\ref{thm:cyclic}. We may also assume that $C_L(N)=N$. If not, $C_L(N)/N$ is a nontrivial normal subgroup of the solvable group $L/N$, so it contains a minimal normal subgroup $M/N$ of $L/N$. By maximality and Proposition~\ref{prop:p-subgroups-abelian}, $M/N$ is a noncyclic elementary abelian $2$-group and the commutators in $M$ are $\pm I_4$. Hence, as $L$ is reducible, $\C^4=V_1\oplus V_2$ for $L$-stable planes on which $N$ acts by scalars, and $M/N$ embeds into the product of the images of $M$ in $\PGL(V_1)$ and $\PGL(V_2)$. By minimality, an element of order $3$ acts nontrivially on $M/N$, so the image of $L$ in $\PGL(V_j)$ contains $A_4$ for some $j$. Thus $l(H)=(2,2,1)$ or $(2,1,1,1)$ for a decomposition~\eqref{eq:blocks} with blocks $V_j$ and $\C e_5$, and Lemma~\ref{lem:monomial-binary-blocks} applies.
	
	If $N_2\cong C_{16}$, the eigenvalues of a generator $g$ with $g^4=iI_4$ are fourth roots of $i$, which form one cycle under $\lambda\mapsto\lambda^{-3}$. So by Lemma~\ref{lem:quartic-sm}(1), each occurs once, and in an eigenbasis of $g$, the terms $x_k^3x_j$ of $h$ form $K_4$ after rescaling. Since $L$ normalizes $\langle g\rangle$, it preserves $K_4$, giving $G_{11}$.
	
	Suppose $N_2=\langle iI_4\rangle$, so $L/N$ acts faithfully on $S$. If $S=1$, then $L=\langle iI_4\rangle$ and $H\subset G_1$. Otherwise, apply Proposition~\ref{prop:abelian-3-groups} to $[\Diag(S,1)]$, taking as $x_5$ a coordinate fixed by $S$. Since $F=h+x_5^4$, Lemma~\ref{lem:quartic-sm}(1)--(3), applied to the invariant monomials, leaves only~(1) with $(a,b)\neq(1,2)$, (2) with $(a,b)\in\{(0,0),(3,0),(4,0)\}$, (3) with $a=0$, and~(5) with $(a,b)=(0,0)$. If $L=N$, these cases together with $\Diag(iI_4,1)$ give subgroups of $G_5$, $G_6$, $G_{13}$, $G_{15}$, $G_{23}$, or $G_{24}$ by Examples~\ref{ex:17maximal} and~\ref{ex:seven-projective-actions}. Otherwise, only $(a,b)=(2,0)$ of~(1) and $(a,b)=(0,0)$ of~(5) admit a nontrivial automorphism of $S$ permuting its characters. In both cases, it interchanges the characters on $\C e_1$ and $\C e_2$, and $S$ acts trivially on $\langle e_3,e_4\rangle$. Thus $L=\langle N,t\rangle$ for some $t$ interchanging $\C e_1$ and $\C e_2$, and after normalization, $t(x_1)=\delta x_2$, $t(x_2)=x_1$, and $t^2=\delta I_4$ with $\delta\in\{1,i\}$. For $S\cong C_3$ and $\delta=1$, the argument of Lemma~\ref{lem:monomial-o2-free} gives $G_1$. Otherwise $x_1^2x_2^2\notin h$, and by invariance under $S$, the terms of $h$ divisible by $x_1^3$ or $x_2^3$ are $x_1^3\ell_1+x_2^3\ell_2$ for linear forms $\ell_1,\ell_2$ in $x_3,x_4$. Smoothness forces $\ell_1$ and $\ell_2$ to be linearly independent, so after changing $x_3,x_4$, we may assume that $\ell_1=x_3$ and $\ell_2=x_4$. Invariance under $t$ gives $t(x_3)=\delta x_4$ and $t(x_4)=x_3$. Hence $L$ preserves $x_1^3x_3+x_3^4+x_2^3x_4+x_4^4$ and permutes $\langle e_1,e_3\rangle$ and $\langle e_2,e_4\rangle$, giving $G_6$.
	
	Finally, let $N_2\cong C_8$, generated by $g$ with $g^2=iI_4$. Every $g$-invariant quartic monomial has odd degree in the coordinates of each eigenspace of $g$. Hence, by Lemma~\ref{lem:quartic-sm}(2), the eigenspaces $U$ and $W$ of $g$ for $\xi_8$ and $\xi_8^5$ are planes, and $L$ permutes them. By Theorem~\ref{thm:maximal-abelian}, $S$ is cyclic of order dividing $9$. Choose coordinates diagonalizing $S$ with $U=\langle e_1,e_2\rangle$ and $W=\langle e_3,e_4\rangle$. Then Lemma~\ref{lem:quartic-sm}(1) gives for each $k$ some $x_k^3x_j\in h$, where $e_j$ lies in the other plane and the character of $S$ on $\C e_j$ is the $(-3)$rd power of that on $\C e_k$. So each plane contains a coordinate fixed by $S$, and Lemma~\ref{lem:quartic-sm}(3) excludes characters of order $9$ and a nontrivial action on only one plane. Thus $g=\Diag(\xi_8,\xi_8,\xi_8^5,\xi_8^5)$, and either $S=1$ or $S=\langle s\rangle$ with $s=\Diag(\xi_3,1,\xi_3^{\varepsilon},1)$ for some $\varepsilon=\pm1$. The stabilizer of $U$ in $L$ centralizes $g$ and, since inverting $s$ changes its characters on $U$, also $S$, so it equals $N$. Moreover, multiplying an element that interchanges $U$ and $W$ by a suitable element of $N$, we obtain an involution. Hence, after choosing suitable bases of $U$ and $W$, $L\subset\langle g,S,P_{(1\,3)(2\,4)}\rangle$. If $S=1$, this group preserves $K_2(x_1,x_3)+K_2(x_2,x_4)$, which is isomorphic to $4K_1$, giving $G_1$. If $S=\langle s\rangle$ with $\varepsilon=1$, it preserves $T_2(x_1+x_3,x_2+x_4)-T_2(x_1-x_3,x_2-x_4)$, giving $G_6$. If $\varepsilon=-1$, let $A_1$, $A_2$, $A_3$ be as in Example~\ref{ex:17maximal}(14). Then $A_2=\Diag(g,1)$, and a matrix commuting with $A_2$ conjugates $\Diag(s,1)$ and $\Diag(P_{(1\,3)(2\,4)},1)$ to $A_1^8$ and $A_3A_1^5A_3^2$, giving $G_{14}$.
\end{proof}
\subsection{Proof of Theorem~\ref{thm:monomial}}
\begin{proof}[Proof of Theorem~\ref{thm:monomial}]
	Suppose that $p\in\{5,7,61\}$ divides $|D|$. Then the Sylow $p$-subgroup of $D$ is cyclic of order $p$ and normal in $H$. By Lemma~\ref{lem:quartic-sm}(1)--(3), after changing the generator and reordering the coordinates, its $F$-lifting is generated by
	$\Diag(\xi_5,\xi_5^2,\xi_5^4,\xi_5^3,1)$, $\Diag(\xi_7,\xi_7^4,\xi_7^2,1,1)$ or  $\Diag(\xi_{61},\xi_{61}^{58},\xi_{61}^9,\xi_{61}^{34},\xi_{61}^{20})$.
	Hence $\widetilde H$ preserves a form isomorphic to $K_4+K_1$, to $K_3+\phi$ with $\phi$ a smooth binary quartic, or to $K_5$, respectively, and $H$ is conjugate to a subgroup of $G_{11}$, of $G_2$ or $G_4$ (by Lemma~\ref{lem:binary-forms}), or of $G_{12}$.
	
	From now on, we assume that $|D|$ has only prime divisors $2$ and $3$. First suppose that $\Pi$ is transitive. If $x_i^4\in F$ for some $i$, then $x_1^4,\dots,x_5^4\in F$, and these terms form an $\widetilde H$-invariant form isomorphic to $5K_1$, so $H$ is conjugate to a subgroup of $G_1$. Otherwise, by Lemma~\ref{lem:quartic-sm}(1), $x_i^3x_j\in F$ for some $i\neq j$. As $\Pi$ is transitive, it contains a $5$-cycle $\sigma$ with $\sigma(i)=j$, so $F$ contains $x_k^3x_{\sigma(k)}$ for all $k$. Thus the diagonal matrices in $\widetilde H$ preserve a form isomorphic to $K_5$, so $D\subset C_{61}$ by the proof of Theorem~\ref{thm:maximal-abelian}. Hence $D=1$ and $\widetilde H\cong\Pi$.
	
	Let $S={\rm Stab}_{\widetilde H}(\C e_5)$, and let $\psi$ be the character by which $S$ acts on $\C e_5$. Choose $g_j\in\widetilde H$ with $g_j(\C e_5)=\C e_j$, and replace $e_j$ by $g_je_5$. If $g\in\widetilde H$ maps $\C e_j$ to $\C e_k$, then $gg_j=g_ks$ for some $s\in S$, so $g(g_je_5)=\psi(s)\,g_ke_5$. Thus all nonzero entries of the matrices in $\widetilde H$ lie in $\psi(S)$. The transitive subgroups of $S_5$ are $C_5$, $D_{10}$, $C_5\rtimes C_4$, $A_5$, and $S_5$ (see e.g. the library of transitive groups in GAP~\cite{GAP}), so $S$ is isomorphic to $1$, $C_2$, $C_4$, $A_4$, or $S_4$. Since $\psi$ is trivial on $S'$, its order divides $|S/S'|$, which is $1$, $2$, $4$, $3$, or $2$, respectively. Hence, unless $\Pi\cong A_5$ and $\psi$ has order $3$, these entries are fourth roots of unity, so $\widetilde H$ preserves $x_1^4+\cdots+x_5^4$, and $H$ is conjugate to a subgroup of $G_1$. In the remaining case, $\C^5$ is the representation of $A_5$ induced from $\psi\neq1$, so it has no trivial constituent by Frobenius reciprocity, and it is irreducible, as the nontrivial irreducible representations of $A_5$ have dimensions $3$, $3$, $4$, and $5$. Hence an element $A\in\widetilde H$ whose image in $\Pi$ is a $3$-cycle has trace $-1$ and eigenvalues $1,\xi_3,\xi_3,\xi_3^2,\xi_3^2$. Since $\langle A\rangle$ is the $F$-lifting of $\langle[A]\rangle$, this contradicts Proposition~\ref{prop:abelian-3-groups}(1), where every nontrivial matrix has eigenvalue $1$ with multiplicity at least two.

	Finally, let $\Pi$ be intransitive. Then $|H|$ has only prime divisors $2$ and $3$. Let $Q$ be the largest normal $2$-subgroup of $H$. If $Q=1$, we may apply Lemma~\ref{lem:monomial-o2-free}.
	Otherwise, the group $\Omega:=\{g\in Z(Q)\mid g^2=1\}$ is a nontrivial normal elementary abelian subgroup. The $F$-lifting of $\Omega$ acts on $\C^5$ by at least two distinct characters. If each has multiplicity at most $3$, apply Lemmas~\ref{lem:monomial-smooth-sum} and~\ref{lem:monomial-small-blocks}.
	Otherwise, there are exactly two, of multiplicities $4$ and $1$. The elements of the $F$-lifting of $\Omega$ act by $\pm1$ on the corresponding subspaces, and such elements form a group of order $4$ containing $-I_5$, whose image under $\pi$ has order $2$. Hence $\Omega\cong C_2$, so $\Omega$ is central, and Proposition~\ref{prop:monomial-low-residue} applies.
\end{proof}

\section{Proofs of the main results}\label{sec:maximal-projective-actions}
In this section, we prove Theorems~\ref{thm:main} and~\ref{thm:24-projective-actions} and Corollary~\ref{cor:jordan}, and we deduce from Theorem~\ref{thm:24-projective-actions} the classification of finite groups of projective automorphisms of smooth quartic surfaces (Theorem~\ref{thm:quartic-surfaces}).

\subsection{Smooth quartic threefolds}
We first show that none of the groups $G_1,\dots,G_{24}$ of Examples~\ref{ex:17maximal} and~\ref{ex:seven-projective-actions} can be omitted from the list.

\begin{lemma}\label{lem:24-irredundant}
	None of the groups $G_1,\dots,G_{24}$ is conjugate to a subgroup of another. Moreover, none of $G_1,\dots,G_{17}$ is isomorphic to a subgroup of another.
\end{lemma}
\begin{proof}
	Let $i\neq j$ be such that $|G_i|$ divides $|G_j|$. Using the generators in Examples~\ref{ex:17maximal} and~\ref{ex:seven-projective-actions}, we verify in GAP~\cite{GAP} that $G_i$ is not isomorphic to a subgroup of $G_j$ if $i\leq17$, and that it is not conjugate to a subgroup of $G_j$ if $i\geq18$. For the latter, we use that two finite subgroups of $\PGL(5,\C)$ are conjugate if and only if their full preimages in $\SL(5,\C)$ admit a trace-preserving isomorphism. GAP shows that no such isomorphism exists.
\end{proof}

\begin{proof}[Proof of Theorem~\ref{thm:24-projective-actions}]
	The theorem follows from Theorems~\ref{thm:nonmonomial} and~\ref{thm:monomial}, Lemma~\ref{lem:24-irredundant}, and the smoothness of $X_i$ (for suitable parameters in (17), (18), and (20)--(23)).
\end{proof}

\begin{proof}[Proof of Theorem~\ref{thm:main}]
	By Lemma~\ref{lem:24-irredundant}, it suffices to show that $G_{18},\dots,G_{24}$ are isomorphic to subgroups of $G_1,\dots,G_{17}$. The groups $G_{18}\cong C_9$ and $G_{23}\cong C_{36}$ embed into $G_{10}\cong(C_4.A_4)\times C_9$, the group $G_{19}\cong C_3\times S_3$ embeds into $G_{17}\cong C_3\times D_{18}$, and the groups $G_{20}\cong G_{21}\cong G_{22}\cong C_{27}$ embed into $G_{15}\cong C_{108}$. Finally, a computation in GAP (\cite{GAP}) shows that $G_{24}\cong\id{144}{36}$ is isomorphic to a subgroup of $G_{10}\cong\id{432}{329}$.
\end{proof}

\begin{proof}[Proof of Corollary~\ref{cor:jordan}]
	By Theorems~\ref{thm:main} and~\ref{thm:24-projective-actions}, $\Aut(X)$ is isomorphic to a subgroup of some $G_i$ with $i\leq17$, and $G_i=\Aut(X_i)$. Hence it suffices to show that $J(G_i)\leq2016$ and $\overline J(G_i)\leq120$ for all $i\leq17$, and that $J(G_2)=2016$ and $\overline J(G_1)=120$. This is a direct computation in GAP~(\cite{GAP}).
\end{proof}

\subsection{Smooth quartic surfaces}\label{sec:quartic-surfaces}
We deduce Theorem~\ref{thm:quartic-surfaces} from Theorem~\ref{thm:24-projective-actions} by the $C_d$-covering construction of~\cite[Definition~6.3]{YYZ24} with $d=4$. Since $F_i=h_i+x_5^4$, the forms $h_i$ defined before Theorem~\ref{thm:quartic-surfaces} are smooth. The groups $G_{h_3}$, $G_{h_5}$, $G_{h_6}$, $G_{h_{11}}$, and $G_{h_{24}}$ of Theorem~\ref{thm:quartic-surfaces} are $\id{1920}{240993}$, $\id{384}{5816}$, $\id{1152}{157515}$, $\id{80}{30}$, and $\id{36}{3}$, respectively.

\begin{proof}[Proof of Theorem~\ref{thm:quartic-surfaces}]
	Let $f\in\C[x_1,x_2,x_3,x_4]$ be a smooth quartic form. Let $F=f+x_5^4$ and let $\delta=[\Diag(I_4,\xi_4)]$. Since the eigenvalues of $\Diag(I_4,\xi_4)$ have multiplicities $4$ and $1$, every element of $\PGL(5,\C)$ commuting with $\delta$ preserves the two eigenspaces, and hence has the form $[\Diag(A,1)]$ with $A\in\GL(4,\C)$. Such an element lies in $G_F$ if and only if $A\in\Aut(f)$. Thus $A\mapsto[\Diag(A,1)]$ is an isomorphism from $\Aut(f)$ to $C_{G_F}(\delta)$, and restricting to the four-dimensional eigenspace gives an exact sequence
	$$
	1\longrightarrow\langle\delta\rangle\longrightarrow C_{G_F}(\delta)\longrightarrow G_f\longrightarrow1.
	$$
	By Theorem~\ref{thm:24-projective-actions}, $G_F$ is conjugate to a subgroup of some $G_j$. Let $\delta_j\in G_j$ be the image of $\delta$ under this conjugation. Then $G_f$ is conjugate to a subgroup of $C_{G_j}(\delta_j)/\langle\delta_j\rangle$, which acts on the four-dimensional eigenspace of $\delta_j$. In a basis in which $\delta_j=[\Diag(I_4,\xi_4)]$, the form $F_j$ becomes $h+ax_5^4$ with $a\in\C^*$ and $h$ smooth. Indeed, $\Diag(I_4,\xi_4)$ acts on each monomial of degree $k$ in $x_5$ by $\xi_4^k$ and maps $F_j$ to a multiple of itself. As $F_j$ is not divisible by $x_5$, this multiple is $1$, so only $k=0$ and $k=4$ occur. Since $G_j=\Aut(X_j)$ by Theorem~\ref{thm:24-projective-actions}, we have $C_{G_j}(\delta_j)/\langle\delta_j\rangle\cong G_h$.
	
	Using the generators in Examples~\ref{ex:17maximal} and~\ref{ex:seven-projective-actions}, we determined in GAP~\cite{GAP} the elements of $G_1,\dots,G_{24}$ that are conjugate to $\delta$ or $\delta^{-1}$ in $\PGL(5,\C)$. Such elements exist in $G_j$ only for $j\leq15$ with $j\neq2,10,12$ and for $j=23,24$. Up to conjugacy in $G_j$, they are $[\Diag(I_4,\xi_4^{\pm1})]$. Hence, after conjugating $\delta_j$ within $G_j$, we may assume that $\delta_j=[\Diag(I_4,\xi_4^{\pm1})]$, so that $C_{G_j}(\delta_j)/\langle\delta_j\rangle=G_{h_j}$. So $G_f$ is conjugate to a subgroup of $G_{h_j}$ for one of these $14$ indices $j$. We then compare these $14$ groups in GAP. It turns out that $G_{h_9}$ is conjugate to $G_{h_8}$ and $G_{h_{14}}$ is conjugate to a subgroup of $G_{h_6}$, while no other of these groups is conjugate to a subgroup of another. Since each $G_{h_i}$ preserves the smooth quartic surface $V(h_i)$, this proves~(1). Finally, $G_{h_7}$, $G_{h_8}$, and $G_{h_{23}}$ are isomorphic to subgroups of $G_{h_4}$, $G_{h_3}$, and $G_{h_{15}}$, respectively, and none of the remaining $9$ groups is isomorphic to a subgroup of another, which proves~(2).
\end{proof}

Among the $12$ groups in Theorem~\ref{thm:quartic-surfaces}(1), only $G_{h_3}$, $G_{h_7}$, and $G_{h_8}$ are primitive, in agreement with \cite[Section~8]{AOT24}. Although $G_{h_8}$ and $G_{h_9}$ are conjugate, the surfaces $V(h_8)$ and $V(h_9)$ are not projectively equivalent, since $\Aut(h_8)\cong G_8$ and $\Aut(h_9)\cong G_9$ are not isomorphic.

\begin{remark}\label{rem:quartic-surface-liftings}
	Among the $12$ groups in Theorem~\ref{thm:quartic-surfaces}(1), $G_{h_i}$ is $h_i$-liftable if and only if $i\in\{4,13,15,23\}$. The other $8$ projective actions admit no lifting. In fact, for each of them, $-I_4$ lies in the derived subgroup of $\Aut(h_i)$, and since multiplying matrices by scalars does not change their commutators, every lifting of $G_{h_i}$ contains this derived subgroup, and hence $-I_4$.
\end{remark}

\begin{remark}\label{rem:quartic-double-solids}
	The same construction classifies the finite groups acting faithfully on smooth quartic double solids. For a smooth quartic form $f\in\C[x_1,\dots,x_4]$, let $Y_f=\{z^2+f(x_1,\dots,x_4)=0\}\subset\P(1,1,1,1,2)$. The double cover $Y_f\to\P^3$ is given by the ample generator of $\Pic(Y_f)$, so every automorphism of $Y_f$ induces an element of $G_f$, and every element of $G_f$ lifts to $Y_f$. Hence $A\mapsto([x:z]\mapsto[Ax:z])$ induces an isomorphism $\Aut(f)/\langle-I_4\rangle\cong\Aut(Y_f)$, under which $iI_4$ gives the covering involution, and $|\Aut(Y_f)|=2|G_f|$. In the proof of Theorem~\ref{thm:quartic-surfaces}, $A\mapsto[\Diag(A,1)]$ maps $\Aut(f)$ onto $C_{G_F}(\delta)$ and $-I_4$ to $\delta^2$, and the conjugations used there preserve $\delta^2$. Hence a finite group acts faithfully on a smooth quartic double solid if and only if it is isomorphic to a subgroup of $\Aut(h_j)/\langle-I_4\rangle\cong\Aut(Y_{h_j})$ for some $j\in\{1,3,4,5,6,7,8,9,11,13,14,15,23,24\}$. In particular, $|\Aut(Y_f)|\leq2|G_{h_3}|=3840$, with equality for the double solid branched over Burnside's quartic. This bound also follows from \cite[Theorem~2.3 and Example~2.8]{EL25}.
\end{remark}

\section{Full automorphism groups}\label{sec:full-automorphism-groups}
In this section, we determine which finite groups occur as the {\it full} automorphism group of a smooth quartic threefold, and how this group varies in moduli. We follow the approach of Fu--Wang--Zheng for cubic threefolds and fourfolds~\cite{FWZ26}.

Let $U\subset\P(\C[x_1,\dots,x_5]_4)\cong\P^{69}$ be the open subset of smooth quartic forms, and let $\mathcal M=U/\PGL(5,\C)$ be the coarse moduli space of smooth quartic threefolds, of dimension $69-24=45$.

\begin{definition}[{\cite{FWZ26}}]\label{def:saturated}
	Let $H\subset\GL(5,\C)$ be a finite subgroup containing $\xi_4I_5$, let $W(H)\subset\C[x_1,\dots,x_5]_4$ be the space of $H$-invariant quartic forms, let $C(H):=C_{\GL(5,\C)}(H)$, and set $m(H)=\dim W(H)-\dim C(H)$. We say that $H$ is {\it saturated} if $W(H)$ contains a smooth form and $\Aut(F)=H$ for a general smooth form $F\in W(H)$. We denote by $\mathcal M(H)\subset\mathcal M$ the set of classes $[V(F)]$ with $F\in W(H)$ smooth, and by $\mathcal M_H\subset\mathcal M$ the set of classes $[V(F)]$ with $F$ smooth and $\Aut(F)$ conjugate to $H$.
\end{definition}

By Schur's lemma, $\dim C(H)$ is the sum of the squares of the multiplicities of the irreducible constituents of $\C^5$. Since $[U/\PGL(5,\C)]$ is a separated Deligne--Mumford stack \cite{Ben13}, the stabilizer group scheme over $U$, with fibers $G_F$, is finite and unramified, so $|\Aut(F)|=4|G_F|$ is upper semicontinuous on $U$. Hence $\Aut(F)$ is saturated for every smooth $F$, and $\mathcal M$ is the disjoint union of the sets $\mathcal M_H$ over the conjugacy classes of saturated groups $H$. The threefolds with $\Aut(X)=1$ correspond to $H=\langle\xi_4I_5\rangle$, with $m(H)=70-25=45$. Following~\cite{FWZ26}, we have:

\begin{proposition}[{\cite[(2.2) and Proposition~2.2]{FWZ26}}]\label{prop:saturated-strata}
	Let $H\subset\GL(5,\C)$ be a finite subgroup containing $\xi_4I_5$ such that $W(H)$ contains a smooth form.
	\begin{enumerate}
		\item $\dim\mathcal M(H)=m(H)$.
		\item The group $H$ is saturated if and only if $m(K)<m(H)$ for every finite subgroup $K\supsetneq H$ of $\GL(5,\C)$ such that $W(K)$ contains a smooth form.
	\end{enumerate}
\end{proposition}

For $1\leq i\leq24$, let $F_i$ be the form of Examples~\ref{ex:17maximal} and~\ref{ex:seven-projective-actions}, with general parameters for $i=17,18,20,\dots,23$. Since $G_{F_i}=G_i$ by Theorem~\ref{thm:24-projective-actions}, (2.3) gives $\Aut(F_i)=\langle\widetilde G_i,\xi_4I_5\rangle\cong C_4\times G_i$ for every $F_i$-lifting $\widetilde G_i$ of $G_i$. 

\begin{theorem}\label{thm:full-groups}
	\begin{enumerate}
    	\item Up to isomorphism, exactly $194$ finite groups occur as $\Aut(X)$ for a smooth quartic threefold $X$, namely $G_1,\dots,G_{24}$ (with $G_{20}\cong G_{21}\cong G_{22}$) and the $172$ groups in Table~\ref{tab:full-groups}.
		\item Exactly $311$ conjugacy classes of finite subgroups of $\PGL(5,\C)$ occur as $\Aut(X)$, and $H\mapsto\pi(H)$ is a bijection from the conjugacy classes of saturated groups onto them.
		\item Every saturated group is conjugate to a subgroup of $\Aut(F_i)$ for some $1\leq i\leq24$.
	\end{enumerate}
\end{theorem}

\begin{table}[htbp]
	\caption{The $172$ groups of Theorem~\ref{thm:full-groups}(1) other than $G_1,\dots,G_{24}$. In the row of order $n$, $[k]$ denotes $\id{n}{k}$; no identifier is available for $n=512$, $1536$, $2688$. Here $E=\id{16}{13}$, the factor $C_2$ of $E^2\rtimes C_2$ interchanges the two copies of $E$, and $S_m$ in $C_4^{m-1}\rtimes S_m$ permutes the coordinates of $\{z\in C_4^m\mid z_1\cdots z_m=1\}$.}
	\label{tab:full-groups}
	\centering
	\begingroup
	\normalfont\fontsize{9}{10.5}\selectfont
	\setlength{\tabcolsep}{0pt}
	\renewcommand{\arraystretch}{1.02}
	\begin{minipage}[t]{0.487\linewidth}
		\begin{tabular}[t]{@{}>{\raggedleft\arraybackslash}p{20pt}@{\hspace{6pt}}p{\dimexpr\linewidth-26pt\relax}@{}}
			\toprule $n$ & Groups $[k]$\\\midrule
			$1$ & \mbox{$1$\,[$1$]}\\
			\addlinespace[1.2pt]
			$2$ & \mbox{$C_{2}$\,[$1$]}\\
			\addlinespace[1.2pt]
			$3$ & \mbox{$C_{3}$\,[$1$]}\\
			\addlinespace[1.2pt]
			$4$ & \mbox{$C_{4}$\,[$1$]}\hspace{6pt}\mbox{$C_{2}^{2}$\,[$2$]}\\
			\addlinespace[1.2pt]
			$5$ & \mbox{$C_{5}$\,[$1$]}\\
			\addlinespace[1.2pt]
			$6$ & \mbox{$S_{3}$\,[$1$]}\hspace{6pt}\mbox{$C_{6}$\,[$2$]}\\
			\addlinespace[1.2pt]
			$8$ & \mbox{$C_{8}$\,[$1$]}\hspace{6pt}\mbox{$C_{4}\times C_{2}$\,[$2$]}\hspace{6pt}\mbox{$D_{8}$\,[$3$]}\hspace{6pt}\mbox{$Q_{8}$\,[$4$]}\hspace{6pt}\mbox{$C_{2}^{3}$\,[$5$]}\\
			\addlinespace[1.2pt]
			$9$ & \mbox{$C_{3}^{2}$\,[$2$]}\\
			\addlinespace[1.2pt]
			$10$ & \mbox{$D_{10}$\,[$1$]}\\
			\addlinespace[1.2pt]
			$12$ & \mbox{$C_{3}\rtimes C_{4}$\,[$1$]}\hspace{6pt}\mbox{$C_{12}$\,[$2$]}\hspace{6pt}\mbox{$A_{4}$\,[$3$]}\hspace{6pt}\mbox{$D_{12}$\,[$4$]}\hspace{6pt}\mbox{$C_{6}\times C_{2}$\,[$5$]}\\
			\addlinespace[1.2pt]
			$16$ & \mbox{$C_{16}$\,[$1$]}\hspace{6pt}\mbox{$C_{4}^{2}$\,[$2$]}\hspace{6pt}\mbox{$(C_{4}\times C_{2})\rtimes C_{2}$\,[$3$]}\hspace{6pt}\mbox{$C_{8}\times C_{2}$\,[$5$]}\\
			& \mbox{$C_{8}\rtimes C_{2}$\,[$6$]}\hspace{6pt}\mbox{$D_{16}$\,[$7$]}\hspace{6pt}\mbox{$C_8\rtimes C_2$\,[$8$]}\hspace{6pt}\mbox{$C_{4}\times C_{2}^{2}$\,[$10$]}\\
			& \mbox{$C_{2}\times D_{8}$\,[$11$]}\hspace{6pt}\mbox{$(C_4\times C_2)\rtimes C_2$\,[$13$]}\hspace{6pt}\mbox{$C_{2}^{4}$\,[$14$]}\\
			\addlinespace[1.2pt]
			$18$ & \mbox{$D_{18}$\,[$1$]}\hspace{6pt}\mbox{$C_{18}$\,[$2$]}\hspace{6pt}\mbox{$C_{6}\times C_{3}$\,[$5$]}\\
			\addlinespace[1.2pt]
			$20$ & \mbox{$C_{5}\rtimes C_{4}$\,[$3$]}\hspace{6pt}\mbox{$D_{20}$\,[$4$]}\\
			\addlinespace[1.2pt]
			$21$ & \mbox{$C_{7}\rtimes C_{3}$\,[$1$]}\\
			\addlinespace[1.2pt]
			$24$ & \mbox{$C_{3}\rtimes C_{8}$\,[$1$]}\hspace{6pt}\mbox{$C_{24}$\,[$2$]}\hspace{6pt}\mbox{$\SL(2,3)$\,[$3$]}\hspace{6pt}\mbox{$C_{4}\times S_{3}$\,[$5$]}\\
			& \mbox{$D_{24}$\,[$6$]}\hspace{6pt}\mbox{$(C_{6}\times C_{2})\rtimes C_{2}$\,[$8$]}\hspace{6pt}\mbox{$C_{12}\times C_{2}$\,[$9$]}\\
			& \mbox{$C_{3}\times D_{8}$\,[$10$]}\hspace{6pt}\mbox{$S_{4}$\,[$12$]}\hspace{6pt}\mbox{$C_{2}\times A_{4}$\,[$13$]}\hspace{6pt}\mbox{$C_{2}^{2}\times S_{3}$\,[$14$]}\\
			& \mbox{$C_{6}\times C_{2}^{2}$\,[$15$]}\\
			\addlinespace[1.2pt]
			$27$ & \mbox{$C_{9}\times C_{3}$\,[$2$]}\\
			\addlinespace[1.2pt]
			$32$ & \mbox{$(C_{4}\times C_{2})\rtimes C_{4}$\,[$2$]}\hspace{6pt}\mbox{$C_{8}\rtimes C_{4}$\,[$4$]}\\
			& \mbox{$(C_{8}\times C_{2})\rtimes C_{2}$\,[$5$]}\hspace{6pt}\mbox{$C_{4}^{2}\rtimes C_{2}$\,[$11$]}\hspace{6pt}\mbox{$C_{16}\rtimes C_{2}$\,[$17$]}\\
			& \mbox{$C_{4}^{2}\times C_{2}$\,[$21$]}\hspace{6pt}\mbox{$C_{2}\times ((C_{4}\times C_{2})\rtimes C_{2})$\,[$22$]}\\
			& \mbox{$C_{4}\times D_{8}$\,[$25$]}\hspace{6pt}\mbox{$C_{2}^{4}\rtimes C_{2}$\,[$27$]}\hspace{6pt}\mbox{$(C_{4}\times C_{2}^{2})\rtimes C_{2}$\,[$28$]}\\
			& \mbox{$C_{2}\times (C_{8}\rtimes C_{2})$\,[$37$]}\hspace{6pt}\mbox{$(C_{8}\times C_{2})\rtimes C_{2}$\,[$38$]}\\
			& \mbox{$(C_{8}\times C_{2})\rtimes C_{2}$\,[$42$]}\hspace{6pt}\mbox{$C_{8}\rtimes C_{2}^{2}$\,[$43$]}\hspace{6pt}\mbox{$C_{4}\times C_{2}^{3}$\,[$45$]}\\
			& \mbox{$C_{2}^{2}\times D_{8}$\,[$46$]}\hspace{6pt}\mbox{$C_2\times E$\,[$48$]}\hspace{6pt}\mbox{$C_{2}^{3}\rtimes C_{2}^{2}$\,[$49$]}\\
			\addlinespace[1.2pt]
			$36$ & \mbox{$C_{18}\times C_{2}$\,[$5$]}\hspace{6pt}\mbox{$C_{3}\times (C_{3}\rtimes C_{4})$\,[$6$]}\hspace{6pt}\mbox{$C_{12}\times C_{3}$\,[$8$]}\\
			& \mbox{$C_{6}\times S_{3}$\,[$12$]}\hspace{6pt}\mbox{$C_{6}^{2}$\,[$14$]}\\
			\addlinespace[1.2pt]
			$40$ & \mbox{$C_{4}\times D_{10}$\,[$5$]}\hspace{6pt}\mbox{$C_{2}\times (C_{5}\rtimes C_{4})$\,[$12$]}\\
			\addlinespace[1.2pt]
			$42$ & \mbox{$C_{2}\times (C_{7}\rtimes C_{3})$\,[$2$]}\\
			\addlinespace[1.2pt]
			$48$ & \mbox{$C_{24}\rtimes C_{2}$\,[$5$]}\hspace{6pt}\mbox{$C_{12}\times C_{4}$\,[$20$]}\\
			& \mbox{$C_{3}\times ((C_{4}\times C_{2})\rtimes C_{2})$\,[$21$]}\hspace{6pt}\mbox{$C_{3}\times (C_{8}\rtimes C_{2})$\,[$24$]}\\
			& \mbox{$\GL(2,3)$\,[$29$]}\hspace{6pt}\mbox{$C_4.A_4$\,[$33$]}\hspace{6pt}\mbox{$C_{2}\times C_{4}\times S_{3}$\,[$35$]}\\
			& \mbox{$(C_{12}\times C_{2})\rtimes C_{2}$\,[$37$]}\hspace{6pt}\mbox{$D_{8}\times S_{3}$\,[$38$]}\\
			& \mbox{$C_{12}\times C_{2}^{2}$\,[$44$]}\hspace{6pt}\mbox{$C_{6}\times D_{8}$\,[$45$]}\hspace{6pt}\mbox{$C_3\times E$\,[$47$]}\\
			& \mbox{$C_{2}\times S_{4}$\,[$48$]}\\
			\addlinespace[1.2pt]
			$54$ & \mbox{$C_{54}$\,[$2$]}\hspace{6pt}\mbox{$C_{18}\times C_{3}$\,[$9$]}\\
			\addlinespace[1.2pt]
			$63$ & \mbox{$C_{3}\times (C_{7}\rtimes C_{3})$\,[$3$]}\\
			\addlinespace[1.2pt]
			$64$ & \mbox{$(C_{8}\times C_{4})\rtimes C_{2}$\,[$6$]}\hspace{6pt}\mbox{$C_{16}\rtimes C_{4}$\,[$28$]}\\
			& \mbox{$C_{4}\times ((C_{4}\times C_{2})\rtimes C_{2})$\,[$58$]}\\
			& \mbox{$(C_4^2\rtimes C_2)\times C_{2}$\,[$101$]}\hspace{6pt}\mbox{$(C_{8}\times C_{2}).C_{2}^{2}$\,[$125$]}\\
			\addlinespace[1.2pt]
			\bottomrule
		\end{tabular}
	\end{minipage}\hfill
	\begin{minipage}[t]{0.487\linewidth}
		\begin{tabular}[t]{@{}>{\raggedleft\arraybackslash}p{20pt}@{\hspace{6pt}}p{\dimexpr\linewidth-26pt\relax}@{}}
			\toprule $n$ & Groups $[k]$\\\midrule
			$64$ & \mbox{$(C_{2}^{4}\rtimes C_{2})\rtimes C_{2}$\,[$138$]}\hspace{6pt}\mbox{$C_{2}\times C_{4}\times D_{8}$\,[$196$]}\\
			& \mbox{$(C_{4}\times C_{2}^{3}).C_{2}$\,[$206$]}\hspace{6pt}\mbox{$D_{8}^{2}$\,[$226$]}\hspace{6pt}\mbox{$(C_{4}\times C_{2}^{2}).C_{2}^{2}$\,[$227$]}\\
			& \mbox{$(C_{8}\times C_{2}).C_{2}^{2}$\,[$256$]}\hspace{6pt}\mbox{$C_2^2\times E$\,[$263$]}\\
			& \mbox{$(C_{4}\times C_{2}^{2}).C_{2}^{2}$\,[$266$]}\\
			\addlinespace[1.2pt]
			$72$ & \mbox{$C_{36}\times C_{2}$\,[$9$]}\hspace{6pt}\mbox{$C_{9}\times D_{8}$\,[$10$]}\hspace{6pt}\mbox{$C_{3}\times (C_{3}\rtimes C_{8})$\,[$12$]}\\
			& \mbox{$C_{12}\times S_{3}$\,[$27$]}\hspace{6pt}\mbox{$C_{3}\times D_{24}$\,[$28$]}\\
			& \mbox{$C_{3}\times ((C_{6}\times C_{2})\rtimes C_{2})$\,[$30$]}\hspace{6pt}\mbox{$C_{3}\times S_{4}$\,[$42$]}\\
			\addlinespace[1.2pt]
			$80$ & \mbox{$C_{5}\rtimes (C_{8}\times C_{2})$\,[$28$]}\hspace{6pt}\mbox{$C_{4}\times (C_{5}\rtimes C_{4})$\,[$30$]}\\
			\addlinespace[1.2pt]
			$84$ & \mbox{$C_{4}\times (C_{7}\rtimes C_{3})$\,[$2$]}\\
			\addlinespace[1.2pt]
			$96$ & \mbox{$((C_{4}\times C_{2})\rtimes C_{4})\rtimes C_{3}$\,[$3$]}\hspace{6pt}\mbox{$(C_4^2\rtimes C_2)\times C_{3}$\,[$54$]}\\
			& \mbox{$C_4^2\rtimes S_3$\,[$64$]}\hspace{6pt}\mbox{$\SL(2,3)\rtimes C_{4}$\,[$67$]}\hspace{6pt}\mbox{$C_{12}\times D_{8}$\,[$165$]}\\
			& \mbox{$C_{4}\times S_{4}$\,[$186$]}\hspace{6pt}\mbox{$C_{4}.S_{4}$\,[$192$]}\hspace{6pt}\mbox{$C_{4}.S_{4}$\,[$193$]}\\
			& \mbox{$(C_4.A_4)\times C_{2}$\,[$200$]}\hspace{6pt}\mbox{$S_3\times E$\,[$215$]}\\
			& \mbox{$C_6\times E$\,[$223$]}\hspace{6pt}\mbox{$C_{2}^{2}\times S_{4}$\,[$226$]}\\
			\addlinespace[1.2pt]
			$120$ & \mbox{$S_{5}$\,[$34$]}\\
			\addlinespace[1.2pt]
			$128$ & \mbox{$(C_{4}^{2}\times C_{2}).C_{2}^{2}$\,[$222$]}\hspace{6pt}\mbox{$(C_{4}\times C_{2}^{3}).C_{2}^{2}$\,[$742$]}\\
			& \mbox{$C_{4}^{2}.D_{8}$\,[$928$]}\hspace{6pt}\mbox{$(C_4^2\rtimes C_2)\times C_2^2$\,[$1631$]}\\
			& \mbox{$(C_{4}\times C_{2}^{3}).C_{2}^{2}$\,[$1757$]}\hspace{6pt}\mbox{$D_8\times E$\,[$2200$]}\\
			& \mbox{$(C_{4}^{2}\times C_{2}).C_{2}^{2}$\,[$2213$]}\\
			\addlinespace[1.2pt]
			$144$ & \mbox{$C_9\times E$\,[$50$]}\hspace{6pt}\mbox{$(C_4.A_4)\times C_{3}$\,[$157$]}\\
			& \mbox{$C_{3}\times ((C_{12}\times C_{2})\rtimes C_{2})$\,[$161$]}\hspace{6pt}\mbox{$C_{6}\times S_{4}$\,[$188$]}\\
			\addlinespace[1.2pt]
			$160$ & \mbox{$(C_{2}^{4}\rtimes C_{5})\rtimes C_{2}$\,[$234$]}\\
			\addlinespace[1.2pt]
			$192$ & \mbox{$(C_4^2\rtimes C_2)\times S_3$\,[$379$]}\hspace{6pt}\mbox{$(C_4^2\rtimes C_2)\times C_{6}$\,[$853$]}\\
			& \mbox{$C_2\times(C_4^2\rtimes S_3)$\,[$944$]}\hspace{6pt}\mbox{$C_{2}\times C_{4}\times S_{4}$\,[$1469$]}\\
			& \mbox{$D_{8}\times S_{4}$\,[$1472$]}\hspace{6pt}\mbox{$(C_{4}\times C_{2}).S_{4}$\,[$1482$]}\\
			& \mbox{$C_{2}^{3}.S_{4}$\,[$1493$]}\hspace{6pt}\mbox{$(C_4.A_4)\times C_2^2$\,[$1500$]}\\
			& \mbox{$(C_{4}\times C_{2}).(C_{2}\times A_{4})$\,[$1504$]}\\
			\addlinespace[1.2pt]
			$240$ & \mbox{$\SL(2,5)\rtimes C_{2}$\,[$93$]}\hspace{6pt}\mbox{$C_{2}\times S_{5}$\,[$189$]}\\
			\addlinespace[1.2pt]
			$256$ & \mbox{$(C_{4}\times C_{2}^{3}).D_{8}$\,[$384$]}\hspace{6pt}\mbox{$(C_4^2\rtimes C_2)\times D_8$\,[$15421$]}\\
			& \mbox{$(C_{4}\times C_{2}^{3}).D_{8}$\,[$26534$]}\hspace{6pt}\mbox{$E^2$\,[$55705$]}\\
			\addlinespace[1.2pt]
			$288$ & \mbox{$C_3\times(C_4^2\rtimes S_3)$\,[$397$]}\hspace{6pt}\mbox{$(C_4.A_4)\times S_3$\,[$925$]}\\
			& \mbox{$(C_4.A_4)\times C_{6}$\,[$983$]}\\
			\addlinespace[1.2pt]
			$384$ & \mbox{$C_{2}^{4}.S_{4}$\,[$5602$]}\hspace{6pt}\mbox{$(C_4.A_4)\times D_8$\,[$18173$]}\\
			& \mbox{$S_4\times E$\,[$20063$]}\hspace{6pt}\mbox{$(C_{4}\times C_{2}^{2}).S_{4}$\,[$20095$]}\\
			\addlinespace[1.2pt]
			$480$ & \mbox{$C_{4}\times S_{5}$\,[$943$]}\\
			\addlinespace[1.2pt]
			$512$ & \mbox{$E^2\rtimes C_2$}\\
			\addlinespace[1.2pt]
			$768$ & \mbox{$(C_4.A_4)\times E$\,[$1085077$]}\\
			& \mbox{$(C_4^2\rtimes C_2)\times S_4$\,[$1087299$]}\\
			& \mbox{$(C_{4}\times C_{2}^{3}).S_{4}$\,[$1088543$]}\\
			\addlinespace[1.2pt]
			$1152$ & \mbox{$(C_4.A_4)\times S_4$\,[$157495$]}\\
			\addlinespace[1.2pt]
			$1536$ & \mbox{$E\times(C_4^2\rtimes S_3)$}\hspace{6pt}\mbox{$C_4^3\rtimes S_4$}\\
			\addlinespace[1.2pt]
			$1920$ & \mbox{$C_{2}^{4}\rtimes S_{5}$\,[$240996$]}\\
			\addlinespace[1.2pt]
			$2688$ & \mbox{$\PSL(2,7)\times E$}\\
			\addlinespace[1.2pt]
			\bottomrule
		\end{tabular}
	\end{minipage}
	\endgroup
\end{table}

\begin{proof}[Proof of Theorem~\ref{thm:full-groups}]
For a subgroup $G\subset G_i$, fix an $F_i$-lifting $\widetilde G\subset\Aut(F_i)$ of $G$. For a homomorphism $\lambda\colon G\to\C^*$, let $\lambda\widetilde G:=\{\lambda([A])A\mid A\in\widetilde G\}$ and $H_\lambda:=\langle\lambda\widetilde G,\xi_4I_5\rangle$. Every lifting of $G$ has the form $\lambda\widetilde G$, the group $H_\lambda$ depends only on $\lambda^4$, and $H_\lambda=\pi^{-1}(G)\cap\Aut(F_i)$ for trivial $\lambda$. Let $F$ be a smooth quartic form. By Theorem~\ref{thm:24-projective-actions}, we may assume that $G_F\subset G_i$ for some $i$, and by (2.3), $\Aut(F)=\langle\widetilde H,\xi_4I_5\rangle$ for an $F$-lifting $\widetilde H$ of $G_F$. Hence $\Aut(F)=H_\lambda$ with $G=G_F$ for some $\lambda$, and every saturated group is conjugate to some $H_\lambda$.
	
		The ancillary file \texttt{SaturatedFamilies-code.txt}, run in GAP~\cite{GAP} with input $G_1,\dots,G_{24}$, enumerates these groups in the following five steps, numbered as in its comments, and writes \texttt{SaturatedFamilies-results.txt}.
	\begin{enumerate}
		\item Applying \texttt{MaximalSubgroupClassReps} repeatedly gives $1398$ subgroups $G$ of the $G_i$ up to conjugacy in $\PGL(5,\C)$. As in the proof of Lemma~\ref{lem:24-irredundant}, conjugacy is tested by trace-preserving isomorphisms between the full preimages $\widehat G$ in $\SL(5,\C)$.
		\item For each $G$ and $\lambda$, $\dim W(H_\lambda)$ and $\dim C(H_\lambda)$ are computed from characters. Keeping the pairs with $m(H_\lambda)\geq0$ (Proposition~\ref{prop:saturated-strata}(1)) leaves $1788$ pairs.
		\item For $1444$ of these pairs, we exhibit a smooth form in $W(H_\lambda)$. For the remaining $344$ pairs, Lemma~\ref{lem:quartic-sm} applies, in suitable coordinates, to every form in $W(H_\lambda)$.
		\item The group $H_\lambda$ is determined by the character $\theta$ of $\widehat G$ on $W(H_\lambda)$. Two such groups are conjugate if and only if their characters differ by a trace-preserving automorphism of $\widehat G$. This gives $1404$ conjugacy classes, and restricting $\theta$ along the inclusions from~(1) gives all inclusions among them.
	 \item By Proposition~\ref{prop:saturated-strata}(2), a class $H$ from~(4) is saturated if and only if $m(K)<m(H)$ for every class $K$ from~(4) that strictly contains a conjugate of $H$. Indeed, for any finite group $K\supsetneq H$ such that $W(K)$ contains a smooth form, the group $\Aut(F)$ of a general smooth $F\in W(K)$ is such a class and has the same $m$ as $K$. This leaves $311$ classes.
	\end{enumerate}
	The program checks that the images of these $311$ classes in $\PGL(5,\C)$ are pairwise nonconjugate and have $194$ isomorphism types, namely those of $G_1,\dots,G_{24}$ and of Table~\ref{tab:full-groups}, and that the maximal classes under containment up to conjugacy are those of $\Aut(F_1),\dots,\Aut(F_{24})$. Since $\mathcal M$ is the disjoint union of the sets $\mathcal M_H$ and $\Aut(X_F)=\pi(\Aut(F))$, this proves~(1)--(3). The results file lists the numbers in (1)--(5) and, for each class, generators, a basis of $W(H)$, a smooth member, and the covering relations.
\end{proof}
	\begin{remark}\label{rem:nontrivial-locus}
	Let $H$ be saturated. Then $\mathcal M(H)$, the image of $\P(W(H))\cap U$, is irreducible. It is closed, since by the finiteness above, the forms $F\in U$ such that $\Aut(F)$ contains a conjugate of $H$ form a closed subset. By upper semicontinuity, $\mathcal M_H$ is open and dense in $\mathcal M(H)$. For saturated $K$, $\mathcal M_K\subset\mathcal M(H)$ if and only if $K$ contains a conjugate of $H$.
	
	Up to conjugacy, the $310$ saturated classes other than $\langle\xi_4I_5\rangle$ have $907$ covering relations (see \texttt{SaturatedFamilies-results.txt}). The $7$ minimal ones are $\langle\xi_4I_5,A\rangle$ with
	$$
	\begin{gathered}
		A=\Diag(1,1,1,1,-1),\quad\Diag(1,1,1,-1,-1),\quad\Diag(1,1,1,1,\xi_3),\quad\Diag(1,1,1,\xi_3,\xi_3^2),\\
		\Diag(1,1,\xi_3,\xi_3,\xi_3^2),\quad\Diag(1,1,1,\xi_3,\xi_3),\quad\Diag(1,\xi_5,\xi_5^2,\xi_5^3,\xi_5^4).
	\end{gathered}
	$$
	Hence the locus where $\Aut(X)\neq1$, or equivalently $\Bir(X)\neq1$, has exactly $7$ irreducible components $\mathcal M(\langle\xi_4I_5,A\rangle)$, of dimensions $29$, $25$, $22$, $17$, $15$, $14$, and $9$, respectively. The $24$ maximal classes are those of $\Aut(F_i)$. Their strata are points, except for $i=17,\dots,23$, where they have dimensions $2$, $4$, $3$, $2$, $1$, $1$, and $2$.
	
	Let $1\neq g\in\Aut(X)$. By Luna's \'etale slice theorem, $\mathcal M$ is \'etale locally at $[X]$ the quotient of $H^1(X,T_X)\cong\C^{45}$ by $\Aut(X)$. The fixed space of $g$ maps finitely into the locus above, so it has dimension at most $29$. Hence $\Aut(X)$ contains no pseudoreflections, and by the Chevalley--Shephard--Todd theorem, this locus is the singular locus of $\mathcal M$, of codimension $16$. For smooth cubic threefolds and fourfolds, there are $40$ and $156$ equivalence classes of saturated families~\cite{FWZ26}.
\end{remark}
\begin{remark}\label{rem:terminal-moduli}
	For every smooth quartic threefold $X$, each nontrivial element of $\Aut(X)$ acts on $H^1(X,T_X)$ with age at least $8$, and this bound is attained by the involution $[\Diag(1,1,1,1,-1)]$ of the Fermat quartic threefold. This can be checked by a direct computation over the cyclic groups in Appendix~\ref{app:abelian-representations}. Since $\mathcal M$ is \'etale locally at $[X]$ the quotient of $H^1(X,T_X)$ by $\Aut(X)$, and $\Aut(X)$ contains no pseudoreflections (Remark~\ref{rem:nontrivial-locus}), the Reid--Tai criterion~\cite[(4.11)]{Rei87} shows that $\mathcal M$ has only terminal quotient singularities.
	\end{remark}

\begingroup
\appendix
\section{Representations of abelian groups}
\label{app:abelian-representations}
We list, up to conjugacy, the nontrivial abelian subgroups of $\PGL(5,\C)$ that preserve a smooth quartic threefold. By Theorem~\ref{thm:24-projective-actions}, every such subgroup is conjugate to a subgroup of one of $G_1,\dots,G_{24}$. We therefore obtain the list in GAP~\cite{GAP} from the generators in Examples~\ref{ex:17maximal} and~\ref{ex:seven-projective-actions}: for each $i$, we compute the nontrivial abelian subgroups of $G_i$ up to conjugacy in $G_i$, and we identify two of them when they are conjugate in $\PGL(5,\C)$. Up to isomorphism, they are the $71$ nontrivial groups that embed into one of the groups in Theorem~\ref{thm:maximal-abelian}, listed below by nonincreasing order. For each group, the table gives diagonal matrices whose classes generate a representative of each conjugacy class. In total, there are $276$ conjugacy classes.

\small
\setlength{\tabcolsep}{3pt}
\begin{longtable}{p{0.5cm}p{1.80cm}p{13.05cm}}
	\hline
	No. & Group & Diagonal matrix generators\\
	\hline
	\endfirsthead
	\hline
	No. & Group & Diagonal matrix generators\\
	\hline
	\endhead
	\hline
	\endfoot
	1 & $C_{4}^{4}$ &
	$\langle \Diag(\xi_4,1,1,1,1),$
	$\Diag(1,\xi_4,1,1,1),$
	$\Diag(1,1,\xi_4,1,1),$
	$\Diag(1,1,1,\xi_4,1)\rangle$
	\\
	\hline
	2 & $C_{12}\times C_{4}^{2}$ &
	$\langle \Diag(\xi_3,\xi_4,1,1,1),$
	$\Diag(1,1,\xi_4,1,1),$
	$\Diag(\xi_{4},1,1,\xi_4,1)\rangle$
	\\
	\hline
	3 & $C_{12}^{2}$ &
	$\langle \Diag(\xi_{12},\xi_4^{3},1,1,1),$
	$\Diag(1,1,\xi_{12},\xi_4^{3},1)\rangle$
	\\
	\hline
	4 & $C_{36}\times C_{4}$ &
	$\langle \Diag(\xi_{9},\xi_{3}^{2},\xi_4,1,1)$, $\Diag(1,1,1,\xi_4,1)\rangle$
	\\
	\hline
	5 & $C_{8}\times C_{4}^{2}$ &
	$\langle \Diag(\xi_{8},\xi_{8}^{5},1,1,1)$,
	$\Diag(1,1,\xi_4,1,1),$
	$\Diag(1,1,1,\xi_4,1)\rangle$
	\\
	\hline
	6 & $C_{4}^{3}\times C_2$ &
	$\langle\Diag(\xi_4,1,1,1,1),\Diag(1,\xi_4,1,\xi_4^a,1),\Diag(1,1,\xi_4,\xi_4^a,1),\Diag(1,1,1,-1,1)\rangle (a=0,1)$
	\\
	\hline
	7 & $C_{28}\times C_{4}$ &
	$\langle \Diag(\xi_{7},\xi_{7}^{4},\xi_{7}^{2},\xi_4,1),$
	$\Diag(1,1,1,1,\xi_4)\rangle$
	\\
	\hline
	8 & $C_{108}$ &
	$\langle \Diag(\xi_{108},\xi_{36}^{35},\xi_{12},\xi_4^{3},1)\rangle$
	\\
	\hline
	9 & $C_{36}\times C_{3}$ &
	$\langle \Diag(\xi_{9},\xi_{3}^{2},1,\xi_4,\xi_4),$
	$\Diag(1,1,1,\xi_{3},1)\rangle$
	\\
	\hline
	10 & $C_{24}\times C_{4}$ &
	$\langle \Diag(\xi_{3},1,\xi_{8},\xi_{8}^{5},1),$
	$\Diag(\xi_4,\xi_4,1,1,1)\rangle$
	\\
	\hline
	11 & $C_{12}{\small\times} C_{4}{\small\times} C_{2}$ &
	$\langle \Diag(\xi_{3},1,\xi_4,\xi_4^{a},1),\Diag(\xi_4,\xi_4,1,\xi_4^{b},1),\Diag(1,1,1,-1,1)\rangle(a,b){\small\in}\{(0,0),(0,1),(1,1)\}$
	\\
	\hline
	12 & $C_{84}$ &
	$\langle \Diag(\xi_{7},\xi_{7}^{4},\xi_{7}^{2},\xi_{12},\xi_4^{3})\rangle$
	\\
	\hline
	13 & $C_{81}$ &
	$\langle \Diag(\xi_{81},\xi_{27}^{26},\xi_{9},\xi_{3}^{2},1)\rangle$
	\\
	\hline
	14 & $C_{80}$ &
	$\langle \Diag(\xi_{80},\xi_{80}^{77},\xi_{80}^{9},\xi_{80}^{53},1)\rangle$
	\\
	\hline
	15 & $C_{72}$ &
	$\langle \Diag(\xi_{9},\xi_{3}^{2},1,\xi_{8},\xi_{8}^{5})\rangle$
	\\
	\hline
	16 & $C_{12}\times C_{6}$ &
	$\langle\Diag(\xi_{12},\xi_4^{3},\xi_4^a,\xi_4^a,1),$
	$\Diag(1,1,\xi_{6},-1,1)\rangle$
	$(a\in\{0,1\})$
	\\
	\hline
	17 & $C_{36}\times C_{2}$ &
	$\langle \Diag(\xi_{36},\xi_{12}^{11},\xi_4,\xi_4^{a},1),$
	$\Diag(1,1,1,-1,1)\rangle$
	$(a\in\{0,1\})$
	\\
	\hline
	18 & $C_{8}^{2}$ &
	$\langle \Diag(\xi_{8},\xi_{8}^{5},1,1,1),$
	$\Diag(1,1,\xi_{8},\xi_{8}^{5},1)\rangle$
	\\
	\hline
	19 & $C_{4}^{3}$ &
	$\langle \Diag(\xi_4,1,1,\xi_4^{a},1),$
	$\Diag(1,\xi_4,1,\xi_4^{b},1),$
	$\Diag(1,1,\xi_4,\xi_4^{c},1)\rangle$
	\newline
	$(a,b,c)\in
	\{(0,0,0),(0,0,2),(0,1,1),(0,1,2),
	(0,3,3),(1,1,1),(2,2,2)\}$
	\\
	\hline
	20 & $C_{8}{\small\times} C_{4}{\small\times} C_{2}$ &
	$\langle \Diag(\xi_{8},\xi_{8}^{5},1,\xi_4^{a},1),$
	$\Diag(1,1,\xi_4,1,1),$
	$\Diag(1,1,1,-1,1)\rangle$
	$(a\in\{0,1\})$
	\\
	\hline
	21 & $C_{4}^{2}\times C_{2}^{2}$ &
	$\langle \Diag(\xi_4,1,1,\xi_4^{a},1),$
	$\Diag(1,\xi_4,1,\xi_4^{b},1),$
	$\Diag(1,1,-1,1,1),$
	$\Diag(1,1,1,-1,1)\rangle$
	\newline
	$(a,b)\in\{(0,0),(0,1),(1,1)\}$
	\\
	\hline
	22 & $C_{61}$ &
	$\langle \Diag(\xi_{61},\xi_{61}^{58},\xi_{61}^{9},\xi_{61}^{34},\xi_{61}^{20})\rangle$
	\\
	\hline
	23 & $C_{56}$ &
	$\langle \Diag(\xi_{7},\xi_{7}^{4},\xi_{7}^{2},\xi_{8},\xi_{8}^{5})\rangle$
	\\
	\hline
	24 & $C_{28}\times C_{2}$ &
	$\langle \Diag(\xi_{28},\xi_{28}^{25},\xi_{28}^{9},\xi_4^{a},1),$
	$\Diag(1,1,1,-1,1)\rangle$
	$(a\in\{0,1\})$
	\\
	\hline
	25 & $C_{54}$ &
	$\langle \Diag(\xi_{54},\xi_{18}^{17},\xi_{6},-1,1)\rangle$
	\\
	\hline
	26 & $C_{18}\times C_{3}$ &
	$\langle \Diag(\xi_{9},\xi_{3}^{2},1,-1,-1),$
	$\Diag(1,1,1,\xi_{3},1)\rangle$
	\\
	\hline
	27 & $C_{12}\times C_{4}$ &
	$\langle \Diag(\xi_{3},1,\xi_4,1,\xi_{12}^{a}),$
	$\Diag(1,1,1,\xi_4,\xi_4^{b})\rangle$
	\newline
	$(b=0,\ a\in\{0,7,8,11\})$ or
	$(b\in\{1,2\},\ a\in\{0,6,9\})$ or
	$(b=3,\ a\in\{0,9\})$
	\\
	\hline
	28 & $C_{24}\times C_{2}$ &
	$\langle \Diag(\xi_{12}^{a},\xi_4^{b},\xi_{8},\xi_{8}^{5},1),$
	$\Diag(-1,-1,1,1,1)\rangle$
	$(a,b)\in\{(4,0),(1,3)\}$
	\\
	\hline
	29 & $C_{12}\times C_{2}^{2}$ &
	$\langle \Diag(\xi_{12},\xi_4^{3},\xi_4^{a},\xi_4^{b},1),\Diag(1,1,-1,1,1),\Diag(1,1,1,-1,1)\rangle(a,b){\small\in}\{(0,0),(0,1),(1,1)\}$
	\\
	\hline
	30 & $C_{42}$ &
	$\langle \Diag(\xi_{7},\xi_{7}^{4},\xi_{7}^{2},\xi_{6},-1)\rangle$
	\\
	\hline
	31 & $C_{40}$ &
	$\langle \Diag(\xi_{40},\xi_{40}^{37},\xi_{40}^{9},\xi_{40}^{13},1)\rangle$
	\\
	\hline
	32 & $C_{36}$ &
	$\langle \Diag(\xi_{36},\xi_{12}^{11},\xi_4,\xi_{36}^{a},1)\rangle$
	$(a\in\{0,9,12,13,18,21,24,27\})$
	\\
	\hline
	33 & $C_{6}^{2}$ &
	$\langle \Diag(\xi_{6},-1,1,1,1),$
	$\Diag(1,1,\xi_{6},-1,1)\rangle$
	\\
	\hline
	34 & $C_{12}\times C_{3}$ &
	$\langle \Diag(\xi_{12},\xi_4^{3},\xi_4^{a},\xi_4^{a},1),$
	$\Diag(1,1,\xi_{3},\xi_{3}^{b},1)\rangle$
	$(a,b)\in\{(0,0),(0,2),(1,0),(2,0),(3,0)\}$
	\\
	\hline
	35 & $C_{18}\times C_{2}$ &
	$\langle \Diag(\xi_{9},\xi_{3}^{2},1,-1,1),$
	$\Diag(1,1,-1,1,1)\rangle$
	\\
	\hline
	36 & $C_{8}\times C_{4}$ &
	$\langle \Diag(\xi_{8},\xi_{8}^{5},\xi_{8}^{a},\xi_{8}^{b},1),$
	$\Diag(1,1,\xi_4,\xi_4^{a},1)\rangle$
	\newline
	$(a,b)\in\{(0,0),(0,2),(0,4),(0,6),(1,5),(2,0),(2,2),(2,4)\}$
	\\
	\hline
	37 & $C_{4}^{2}\times C_{2}$ &
	$\langle \Diag(\xi_4,1,\xi_4^{a},\xi_4,1),$
	$\Diag(1,\xi_4,(-1)^{b},\xi_4^{c},1),$
	$\Diag(1,1,1,-1,1)\rangle$
	\newline
	$(b\in\{0,1\},\
	(a,c)\in\{(0,0),(1,0),(2,0),(0,1)\})$
	\newline
	$\langle \Diag(\xi_4,1,\xi_4^{a},1,1),\Diag(1,\xi_4,\xi_4^{b},1,1),\Diag(1,1,-1,-1,1)\rangle\;(a\in\{0,3\},\ b\in\{1,2,3\})$
	\\
	\hline
	38 & $C_{8}\times C_{2}^{2}$ &
	$\langle \Diag(\xi_{8},\xi_{8}^{5},1,\xi_4^{a},1),$
	$\Diag(1,1,-1,1,1),$
	$\Diag(1,1,1,-1,1)\rangle$
	$(a\in\{0,1\})$
	\\
	\hline
	39 & $C_{4}\times C_{2}^{3}$ &
	$\langle \Diag(\xi_4,1,1,\xi_4^{a},1),\Diag(1,-1,1,1,1),\Diag(1,1,-1,1,1),\Diag(1,1,1,-1,1)\rangle(a=0,1)$
	\\
	\hline
	40 & $C_{28}$ &
	$\langle \Diag(\xi_{28},\xi_{28}^{25},\xi_{28}^{9},\xi_4^{a},1)\rangle$
	$(a\in\{0,1,2,3\})$
	\\
	\hline
	41 & $C_{14}\times C_{2}$ &
	$\langle \Diag(\xi_{7},\xi_{7}^{4},\xi_{7}^{2},-1,1),$
	$\Diag(1,1,1,1,-1)\rangle$
	\\
	\hline
	42 & $C_{27}$ &
	$\langle \Diag(\xi_{27},\xi_{9}^{8},\xi_{3},1,1)\rangle$,
	$\langle \Diag(\xi_{27},\xi_{9}^{8},\xi_{9}^{a},\xi_{3},1)\rangle$
	$(a\in\{5,6\})$,
	$\langle \Diag(\xi_{27},\xi_{9}^{2},\xi_{9}^{8},\xi_{3},1)\rangle$
	\\
	\hline
	43 & $C_{9}\times C_{3}$ &
	$\langle \Diag(\xi_{9},\xi_{3}^{2},
	\xi_{3}^{a},\xi_{9}^{2a},1),$
	$\Diag(1,1,1,\xi_{3},1)\rangle$
	$(a\in\{0,1\})$
	\\
	\hline
	44 & $C_{24}$ &
	$\langle \Diag(\xi_{8},\xi_{8}^{5},
	\xi_{24}^{a},\xi_{24}^{b},1)\rangle$
	$(a,b)\in\{(0,8),(2,18),(4,12),(6,8),(7,11),(7,19),(8,16)\}$
	\\
	\hline
	45 & $C_{12}\times C_{2}$ &
	$\langle \Diag(\xi_{3},1,\xi_4,\xi_{12}^{a},\xi_{12}^{b}),$
	$\Diag(1,1,1,-1,(-1)^{c})\rangle$
	$(c=0,\ (a,b)\in\{(0,0),(0,3), (0,6)$, $(0,7), (0,8),(0,9),(0,11),(3,0),(3,3),(3,6),(3,8),(3,9)\})$
	\newline
	or $(c=1,\ a=0,\ b\in\{0,3,4,6,8,9\})$
	\\
	\hline
	46 & $C_{6}\times C_{2}^{2}$ &
	$\langle \Diag(\xi_{6},-1,1,1,1),$
	$\Diag(1,1,-1,1,1),$
	$\Diag(1,1,1,-1,1)\rangle$
	\\
	\hline
	47 & $C_{21}$ &
	$\langle \Diag(\xi_{7},\xi_{7}^{4},\xi_{7}^{2},\xi_{3},1)\rangle$
	\\
	\hline
	48 & $C_{20}$ &
	$\langle \Diag(\xi_{20},\xi_{20}^{17},\xi_{20}^{9},\xi_{20}^{13},1)\rangle$
	\\
	\hline
	49 & $C_{18}$ &
	$\langle \Diag(\xi_{9},\xi_{3}^{2},
	\xi_{18}^{a},-1,1)\rangle$
	$(a\in\{0,3,6,8,9,15\})$
	\\
	\hline
	50 & $C_{6}\times C_{3}$ &
	$\langle \Diag(\xi_{6},-1,(-1)^{a},(-1)^{a},1),$
	$\Diag(1,1,\xi_{3},\xi_{3}^{b},1)\rangle$
	$(a,b)\in\{(0,0),(0,2),(1,0)\}$
	\\
	\hline
	51 & $C_{16}$ &
	$\langle \Diag(\xi_{16},\xi_{16}^{13},\xi_{16}^{9},\xi_{16}^{5},1)\rangle$
	\\
	\hline
	52 & $C_{4}^{2}$ &
	$\langle \Diag(\xi_4,1,\xi_4^{a},\xi_4,1),$
	$\Diag(1,\xi_4,\xi_4^{b},(-1)^{c},1)\rangle$
	$(a\in\{0,1,3\},\ b\in\{0,3\},\ c\in\{0,1\})$
	\newline
	$\langle \Diag(\xi_4,1,\xi_4^{a},-1,1),$
	$\Diag(1,\xi_4,-1,\xi_4,1)\rangle$
	$(a\in\{0,1\})$
	\newline
	$\langle \Diag(\xi_4,1,1,\xi_4^{a},1),$
	$\Diag(1,\xi_4,\xi_4^{3},(-1)^{b},1)\rangle$
	$(a\in\{0,3\},\ b\in\{0,1\})$
	\\
	\hline
	53 & $C_{8}\times C_{2}$ &
	$\langle \Diag(\xi_{8},\xi_{8}^{5},\xi_{8}^{a},\xi_{8}^{b},1),$
	$\Diag(1,1,(-1)^{a},-1,1)\rangle$
	\newline
	$(a,b)\in\{(0,0),(0,2),(1,5),(2,0),(3,7),(4,0),(4,2),(6,0)\}$
	\\
	\hline
	54 & $C_{4}\times C_{2}^{2}$ &
	$\langle \Diag(\xi_4,\xi_4^{a},1,\xi_4^{b},1),$
	$\Diag(1,-1,1,(-1)^{c},1),$
	$\Diag(1,1,-1,(-1)^{c},1)\rangle$
	\newline
	$(c=0,\ (a,b)\in\{(0,0),(0,1),(0,2),(0,3),(1,0),(1,2)\})$ or
	\newline
	$(c=1,\ a=0,\ b\in\{0,1,2,3\})$
	\\
	\hline
	55 & $C_{2}^{4}$ &
	$\langle \Diag(-1,1,1,1,1),$
	$\Diag(1,-1,1,1,1),$
	$\Diag(1,1,-1,1,1),$
	$\Diag(1,1,1,-1,1)\rangle$
	\\
	\hline
	56 & $C_{14}$ &
	$\langle \Diag(\xi_{7},\xi_{7}^{4},\xi_{7}^{2},-1,(-1)^{a})\rangle$
	$(a\in\{0,1\})$
	\\
	\hline
	57 & $C_{12}$ &
	$\langle \Diag(\xi_{12},\xi_4^{3},
	\xi_{12}^{a},\xi_{12}^{b},1)\rangle$
	$(a\in\{0,5\},\ b\in\{0,6,8,9\})$ or
	\newline
	$(a=3,\ b\in\{0,3,4,5,6,7,8,9,11\})$ or
	$(a\in\{1,4,6,9\},\ b=9)$
	\\
	\hline
	58 & $C_{6}\times C_{2}$ &
	$\langle \Diag(\xi_{3},1,-1,1,\xi_{6}^{a}),\Diag(1,1,1,-1,(-1)^{b})\rangle$
	\newline
	$(a,b)\in\{(0,0),(0,1),(1,0),(2,1),(3,1),(4,0)\}$
	\\
	\hline
	59 & $C_{10}$ &
	$\langle \Diag(\xi_{10},\xi_{10}^{7},\xi_{10}^{9},\xi_{10}^{3},1)\rangle$
	\\
	\hline
	60 & $C_{9}$ &
	$\langle \Diag(\xi_{9},\xi_{3}^{2},
	\xi_{9}^{a},\xi_{9}^{b},1)\rangle$
	$(a,b)\in\{(0,0),(0,3),(0,4),(0,6),(2,3),(3,4)\}$
	\\
	\hline
	61 & $C_{3}^{2}$ &
	$\langle \Diag(\xi_{3},1,\xi_3^a,1,1),$
	$\Diag(1,\xi_{3},\xi_3^b,1,1)\rangle$
	$(a,b)\in\{(0,0),(0,2),(1,1)\}$
	\\
	\hline
	62 & $C_{8}$ &
	$\langle \Diag(\xi_{8},\xi_{8}^{5},\xi_{8}^{a},\xi_{8}^{b},1)\rangle$
	$(a,b)\in\{(0,0),(0,2),(0,4),(0,6),(1,5),(2,4),(3,7)\}$
	\\
	\hline
	63 & $C_{4}\times C_{2}$ &
	$\langle \Diag(\xi_4,1,\xi_4^{a},\xi_4^{b},1),$
	$\Diag(1,-1,1,(-1)^{c},1)\rangle$
	\newline
	$(c=0,\ (a,b)\in
	\{(0,0),(0,1),(0,2),(0,3),(1,1),(1,2),(1,3),(2,3)\})$
	\newline
	$(c=1,\ (a,b)\in
	\{(0,0),(0,1),(0,2),(0,3),(2,1),(2,2)\})$
	\\
	\hline
	64 & $C_{2}^{3}$ &
	$\langle \Diag(-1,1,1,1,1),$
	$\Diag(1,-1,1,(-1)^{a},1),$
	$\Diag(1,1,-1,(-1)^{a},1)\rangle$
	$(a\in\{0,1\})$
	\\
	\hline
	65 & $C_{7}$ &
	$\langle \Diag(\xi_{7},\xi_{7}^{4},\xi_{7}^{2},1,1)\rangle$
	\\
	\hline
	66 & $C_{6}$ &
	$\langle \Diag(\xi_{6},-1,\xi_{6}^{a},\xi_{6}^{b},1)\rangle\;(a,b){\small\in}\{(0,0),(0,2),(0,3),(0,4),(0,5),(1,3),(2,4),(3,3),(3,5)\}$
	\\
	\hline
	67 & $C_{5}$ &
	$\langle \Diag(\xi_{5},\xi_{5}^{2},\xi_{5}^{4},\xi_{5}^{3},1)\rangle$
	\\
	\hline
	68 & $C_{4}$ &
	$\langle \Diag(\xi_4,\xi_4^{a},1,1,1)\rangle$
	$(a\in\{0,1,2,3\})$,\;\;\;\;\;\;
	$\langle \Diag(\xi_4,-1,\xi_4^a,1,1)\rangle$
	$(a\in\{1,2,3\})$
	\\
	\hline
	69 & $C_{2}^{2}$ &
	$\langle \Diag(-1,1,(-1)^{a},1,1),$
	$\Diag(1,-1,(-1)^{b},1,1)\rangle$
	$(a,b)\in\{(0,0),(0,1),(1,1)\}$
	\\
	\hline
	70 & $C_{3}$ &
	$\langle \Diag(\xi_{3},\xi_{3}^{a},1,1,1)\rangle$
	$(a\in\{0,1,2\})$, \;\;\;\;\;
	$\langle \Diag(\xi_{3},\xi_{3},\xi_{3}^{2},1,1)\rangle$
	\\
	\hline
	71 & $C_{2}$ &
	$\langle \Diag(1,1,1,-1,(-1)^{a})\rangle$
	$(a\in\{0,1\})$
	\\
	\hline
\end{longtable}

\endgroup


\begin{thebibliography}{GLMV24}
	\bibitem[AM72]{AM72}
	M. Artin and D. Mumford,
	\emph{Some elementary examples of unirational varieties which are not
		rational}, Proc. London Math. Soc. (3) \textbf{25} (1972), 75--95.
	
	\bibitem[AOT24]{AOT24}
	J. Avila, G. Ortiz and S. Troncoso,
	\emph{Invariant smooth quartic surfaces by all finite primitive subgroups of
		$\PGL_4(\C)$}, J. Pure Appl. Algebra \textbf{228} (2024), Paper No. 107534.
	
	\bibitem[Ben13]{Ben13}
	O. Benoist,
	\emph{S\'eparation et propri\'et\'e de Deligne--Mumford des champs de modules d'intersections compl\`etes lisses},
	J. Lond. Math. Soc. (2) \textbf{87} (2013), no.~1, 138--156.
	
	\bibitem[Bli17]{Bli17}
	H. Blichfeldt, \emph{Finite collineation groups}, University of Chicago
	Press, Chicago, 1917.
	
	\bibitem[BH23]{BH23}
	S. Brandhorst and T. Hofmann,
	\emph{Finite subgroups of automorphisms of K3 surfaces},
	Forum Math. Sigma \textbf{11} (2023), Paper No.~e54.
	
	\bibitem[Bra67]{Bra67}
	R. Brauer, \emph{\"Uber endliche lineare Gruppen von Primzahlgrad},
	Math. Ann. \textbf{169} (1967), 73--96.
	
	\bibitem[Bur11]{Bur11}
	W. Burnside,
	\emph{Theory of Groups of Finite Order},
	second edition, Cambridge University Press, Cambridge, 1911.
	
	\bibitem[Cha78]{Cha78}
	H. C. Chang,
	\emph{On plane algebraic curves},
	Chinese J. Math. \textbf{6} (1978), 185--189.
	
	\bibitem[Che05]{Che05}
	I. Cheltsov,
	\emph{Birationally rigid Fano varieties},
	Russian Math. Surveys \textbf{60} (2005), no.~5, 875--965.
	
	\bibitem[CMS26]{CMS26}
	I. Cheltsov, F. Mangolte and C. Shramov,
	\emph{On $G$-birational rigidity of projective spaces},
	arXiv:2604.20427 (2026).
	
	\bibitem[CS16]{CS16}
	I. Cheltsov and C. Shramov,
	\emph{Two rational nodal quartic $3$-folds},
	Q. J. Math. \textbf{67} (2016), no.~4, 573--601.
	
	\bibitem[CS19]{CS19}
	I. Cheltsov and C. Shramov,
	\emph{Finite collineation groups and birational rigidity},
	Selecta Math. (N.S.) \textbf{25} (2019), no.~5, Paper No.~71, 68 pp.
	
	\bibitem[CG72]{CG72}
	C. H. Clemens and P. A. Griffiths,
	\emph{The intermediate Jacobian of the cubic threefold},
	Ann. of Math. (2) \textbf{95} (1972), 281--356.
	
	\bibitem[dF13]{dF13}
	T. de Fernex,
	\emph{Birationally rigid hypersurfaces},
	Invent. Math. \textbf{192} (2013), 533--566.
	
	\bibitem[dF16]{dF16}
	T. de Fernex,
	\emph{Erratum to: Birationally rigid hypersurfaces},
	Invent. Math. \textbf{203} (2016), no.~2, 675--680.
	
	\bibitem[Dol12]{Dol12}
	I. V. Dolgachev,
	\emph{Classical algebraic geometry. A modern view},
	Cambridge University Press, Cambridge, 2012.
	
	\bibitem[Dol18]{Dol18}
	I. V. Dolgachev, \emph{Quartic surfaces with icosahedral symmetry},
	Adv. Geom. \textbf{18} (2018), 119--132.
	
	\bibitem[Dol19]{Dol19}
	I. V. Dolgachev, \emph{A brief introduction to automorphisms of algebraic varieties},
	lecture notes, Talca, 2019.
	
	\bibitem[DI09]{DI09}
	I. V. Dolgachev and V. A. Iskovskikh,
	\emph{Finite subgroups of the plane Cremona group}, in Algebra, Arithmetic,
	and Geometry, Vol. I, Progr. Math. \textbf{269}, Birkh\"auser, 2009,
	443--548.
	
	\bibitem[Edge47]{Edge47}
	W. L. Edge,
	\emph{The Klein group in three dimensions},
	Acta Math. \textbf{79} (1947), 153--223.
	
	\bibitem[EL25]{EL25}
	L. Esser and J. Li,
	\emph{Hypersurfaces with large automorphism groups},
	Trans. Amer. Math. Soc. \textbf{378} (2025), no.~5, 3667--3698.
	
	\bibitem[FMPK16]{FMPK16}
	G. Faina, S. Marcugini, F. Pambianco and H. Kaneta,
	\emph{$S_5$-invariant nonsingular quartic surfaces},
	arXiv:1605.08720, 2016.
	
	\bibitem[FWZ26a]{FWZ26}
	J. Fu, S. Wang and Z. Zheng,
	\emph{Classification of automorphism groups of smooth cubic threefolds and fourfolds},
	arXiv:2609.20403, 2026.
	
	\bibitem[FWZ26b]{FWZ26b}
	J. Fu, S. Wang and Z. Zheng,
	\emph{Automorphism groups of quartic K3 surfaces},
	to appear, 2026.

	\bibitem[GAP]{GAP}
	The GAP Group,
	\emph{GAP -- Groups, Algorithms, and Programming, Version 4.15.1},
	2025, \url{https://www.gap-system.org}.
	
	\bibitem[GAL13]{GAL13}
	V. Gonz\'alez-Aguilera and A. Liendo,
	\emph{On the order of an automorphism of a smooth hypersurface},
	Israel J. Math. \textbf{197} (2013), no.~1, 29--49.
	
	\bibitem[GLM23]{GLM23}
	V. Gonz\'alez-Aguilera, A. Liendo and P. Montero,
	\emph{On the liftability of the automorphism group of smooth hypersurfaces
		of the projective space}, Israel J. Math. \textbf{255} (2023), 283--310.
	
	\bibitem[GLMV24]{GLMV24}
	V. Gonz\'alez-Aguilera, A. Liendo, P. Montero and R. Villaflor Loyola,
	\emph{On a Torelli principle for automorphisms of Klein hypersurfaces},
	Trans. Amer. Math. Soc. \textbf{377} (2024), no.~8, 5483--5511.
	
	\bibitem[Hos97]{Hos97}
	T. Hosoh,
	\emph{Automorphism groups of cubic surfaces},
	J. Algebra \textbf{192} (1997), 651--677.
	
	\bibitem[HS80]{HS80}
	A. Howard and A. Sommese,
	\emph{On the orders of the automorphism groups of certain projective
		manifolds}, in Manifolds and Lie groups (Notre Dame, Ind., 1980), 145--158.
	
	\bibitem[IM71]{IM71}
	V. A. Iskovskikh and Yu. I. Manin,
	\emph{Three-dimensional quartics and counterexamples to the L\"uroth
		problem}, Math. USSR-Sb. \textbf{15} (1971), 141--166.
	
	\bibitem[IP99]{IP99}
	V. A. Iskovskikh and Yu. G. Prokhorov, \emph{Fano varieties}, in
	Algebraic Geometry V, Encyclopaedia Math. Sci. \textbf{47}, Springer, 1999,
	1--247.
	
	\bibitem[Jor78]{Jor78}
	C. Jordan,
	\emph{M\'emoire sur les \'equations diff\'erentielles lin\'eaires \`a
		int\'egrale alg\'ebrique},
	J. Reine Angew. Math. \textbf{84} (1878), 89--215.
	
	\bibitem[Kle79]{Kle79}
	F. Klein,
	\emph{Ueber die Transformation siebenter Ordnung der elliptischen Functionen},
	Math. Ann. \textbf{14} (1879), 428--471.
	
	\bibitem[Kle84]{Kle84}
	F. Klein,
	\emph{Vorlesungen \"uber das Ikosaeder und die Aufl\"osung der Gleichungen
		vom f\"unften Grade}, B. G. Teubner, Leipzig, 1884.
	
	\bibitem[Kol19]{Kol19}
	J. Koll\'ar,
	\emph{The rigidity theorem of
		Fano--Segre--Iskovskikh--Manin--Pukhlikov--Corti--Cheltsov--de
		Fernex--Ein--Musta\c{t}\u{a}--Zhuang},
	in Birational Geometry of Hypersurfaces,
	Lect. Notes Unione Mat. Ital. \textbf{26}, Springer, Cham, 2019,
	129--164.
	
	\bibitem[KOY26]{KOY26}
	J. Keum, K. Oguiso and X. Yu,
	\emph{Sixteen generators of the automorphism group of the Fermat quartic surface},
	arXiv:2608.18535 (2026).
	
	\bibitem[LZ22]{LZ22}
	R. Laza and Z. Zheng,
	\emph{Automorphisms and periods of cubic fourfolds},
	Math. Z. \textbf{300} (2022), 1455--1507.
	
	\bibitem[MPK16]{MPK16}
	S. Marcugini, F. Pambianco and H. Kaneta,
	\emph{Projective automorphism groups of nonsingular quartic surfaces},
	arXiv:1611.10101 (2016).
	
	\bibitem[MM63]{MM63}
	H. Matsumura and P. Monsky, \emph{On the automorphisms of hypersurfaces},
	J. Math. Kyoto Univ. \textbf{3} (1963), 347--361.
	
	\bibitem[Muk88]{Muk88}
	S. Mukai,
	\emph{Finite groups of automorphisms of K3 surfaces and the Mathieu
		group}, Invent. Math. \textbf{94} (1988), 183--221.
	
	\bibitem[Ogu05]{Ogu05}
	K. Oguiso,
	\emph{A characterization of the Fermat quartic K3 surface by means of finite symmetries},
	Compos. Math. \textbf{141} (2005), 404--424.
	
	\bibitem[OY19]{OY19}
	K. Oguiso and X. Yu, \emph{Automorphism groups of smooth quintic
		threefolds}, Asian J. Math. \textbf{23} (2019), 201--256.
	
	\bibitem[PZ21]{PZ21}
	T. Peng and Z. Zheng,
	\emph{Abelian automorphism groups of quartic surfaces and cubic fourfolds},
	arXiv:2106.14214 (2021).
	
	\bibitem[Pro12]{Pro12}
	Yu. Prokhorov,
	\emph{Simple finite subgroups of the Cremona group of rank 3},
	J. Algebraic Geom. \textbf{21} (2012), 563--600.
	
	\bibitem[PS14]{PS14}
	Yu. Prokhorov and C. Shramov,
	\emph{Jordan property for groups of birational selfmaps},
	Compos. Math. \textbf{150} (2014), no.~12, 2054--2072.
	
	\bibitem[PS17]{PS17}
	Yu. Prokhorov and C. Shramov,
	\emph{Jordan constant for Cremona group of rank 3},
	Mosc. Math. J. \textbf{17} (2017), no.~3, 457--509.
	
	\bibitem[Puk98]{Puk98}
	A. V. Pukhlikov,
	\emph{Birational automorphisms of Fano hypersurfaces},
	Invent. Math. \textbf{134} (1998), no.~2, 401--426.
    
    \bibitem[Rei87]{Rei87}
    M. Reid,
    \emph{Young person's guide to canonical singularities},
    in Algebraic Geometry, Bowdoin, 1985, Proc. Sympos. Pure Math. \textbf{46}, Part~1, Amer. Math. Soc., Providence, RI, 1987, 345--414.
    
    \bibitem[Sage]{Sage}
    The Sage Developers,
    \emph{SageMath, the Sage Mathematics Software System (Version 10.0)},
    2023, \url{https://www.sagemath.org}.

	\bibitem[Sch82]{Sch82}
	F. Schur,
	\emph{Ueber eine besondre Classe von Fl\"achen vierter Ordnung},
	Math. Ann. \textbf{20} (1882), 254--296.
	
	\bibitem[Seg42]{Seg42}
	B. Segre,
	\emph{The non-singular cubic surfaces},
	Oxford University Press, Oxford, 1942.
	
	\bibitem[Shi88]{Shi88}
	T. Shioda,
	\emph{Arithmetic and geometry of Fermat curves},
	in Algebraic Geometry Seminar (Singapore, 1987), World Sci. Publishing,
	Singapore, 1988, 95--102.
	
	\bibitem[WY20]{WY20}
	L. Wei and X. Yu, \emph{Automorphism groups of smooth cubic threefolds},
	J. Math. Soc. Japan \textbf{72} (2020), 1327--1343.
	
	\bibitem[Wo]{Wo}
	Wolfram Research, Inc.,
	\emph{Mathematica, Version 10.0}, Champaign, IL, 2014.
	
	\bibitem[YYZ24]{YYZ24}
	S. Yang, X. Yu and Z. Zhu,
	\emph{Automorphism groups of cubic fivefolds and fourfolds},
	J. Lond. Math. Soc. (2) \textbf{110} (2024), Paper No.~e12997, 35 pp.
	
	\bibitem[YYZ25]{YYZ25}
	S. Yang, X. Yu and Z. Zhu,
	\emph{On automorphism groups of smooth hypersurfaces},
	J. Algebraic Geom. \textbf{34} (2025), 579--611.
	
	\bibitem[Zhe22]{Zhe22}
	Z. Zheng, \emph{On abelian automorphism groups of hypersurfaces},
	Israel J. Math. \textbf{247} (2022), 479--498.
	
\end{thebibliography}
\end{document}